\documentclass[11pt,reqno]{amsart}

\usepackage{stmaryrd}
\usepackage[utf8]{inputenc}
\usepackage[T1]{fontenc}
\usepackage{graphicx} 
\usepackage[english]{babel}
\usepackage{csquotes,nicefrac}
\usepackage{amsmath,amssymb, bm}
\usepackage{tikz-cd}
\usepackage{thmtools}
\usepackage{footmisc}

\usepackage{todonotes}
\usepackage[normalem]{ulem}

\usepackage[most]{tcolorbox}
\usetikzlibrary{shapes.geometric}

\usepackage{array,booktabs,tabularx}
\usepackage{placeins}
\usepackage{float}

\usepackage{xcolor,mathrsfs,amsmath,mathtools,amssymb,hyperref,stmaryrd}
\def\E{\mathbb E}

\usepackage{enumitem}
\usepackage[capitalise,nameinlink]{cleveref} 
\crefname{enumi}{item}{items}
\crefname{equation}{}{}
\crefname{subsection}{Subsection}{Subsections}

\usepackage[
  backend=biber,
  style=alphabetic,
  maxalphanames=4,
  minalphanames=4, 
  maxbibnames=99
]{biblatex}

\usepackage{amsthm}
\usepackage{mathabx}
\usepackage{enumitem}
\usepackage{bbm}
\usepackage{amssymb}
\usepackage{amsfonts}
\usepackage[english]{babel}
\usepackage{geometry}
 \allowdisplaybreaks
 \usepackage{appendix}
\usepackage{hyperref}
\hypersetup{
colorlinks   = true,
citecolor    = blue,
linkcolor=blue
}

\allowdisplaybreaks

 \usepackage[most]{tcolorbox}
\theoremstyle{definition}
\newtheorem{theorem}{Theorem}[section]
\AfterEndEnvironment{theorem}{\noindent\ignorespaces}
\theoremstyle{definition}
\newtheorem{definition}[theorem]{Definition}
\AfterEndEnvironment{definition}{\noindent\ignorespaces}

\theoremstyle{definition}
\newtheorem{example}[theorem]{Example}
\AfterEndEnvironment{Example}{\noindent\ignorespaces}

\def\l@subsection{\@tocline{2}{0pt}{2.5pc}{5pc}{}}

\theoremstyle{definition}
\newtheorem{proposition}[theorem]{Proposition}
\AfterEndEnvironment{proposition}{\noindent\ignorespaces}

\theoremstyle{definition}
\newtheorem{lemma}[theorem]{Lemma}
\AfterEndEnvironment{lemma}{\noindent\ignorespaces}

\theoremstyle{definition}
\newtheorem{corollary}[theorem]{Corollary}
\AfterEndEnvironment{corollary}{\noindent\ignorespaces}

\newtheorem{remark}[theorem]{Remark}
\AfterEndEnvironment{remark}{\noindent\ignorespaces}

\newtheorem{assumption}{Assumption}



\usepackage{accents}

\DeclareMathOperator{\Lip}{\operatorname{Lip}}

\DeclareMathOperator*{\esssup}{\operatorname{esssup}}

\newcommand{\N}{\mathbb{N}}

\renewcommand{\P}{\mathbb{P}}

\newcommand{\R}{\mathbb{R}}

\newcommand{\X}{\mathbb{X}}

\newcommand{\cB}{\mathcal{B}}

\newcommand{\cD}{\mathcal{D}}

\newcommand{\cI}{\mathcal{I}}
\newcommand{\cJ}{\mathcal{J}}

\newcommand{\cL}{\mathcal{L}}
\newcommand{\cM}{\mathcal{M}}

\newcommand{\cP}{\mathcal{P}}

\newcommand{\cV}{\mathscr{V}}

\newcommand{\cY}{\mathcal{Y}}

\newcommand{\bM}{\mathbf{M}}
\newcommand{\bN}{\mathbf{N}}

\newcommand{\bX}{\mathbf{X}}

\newcommand{\1}{\mathbf{1}}
\renewcommand{\d}{\mathrm{d}}
\newcommand{\dd}{\,\mathrm{d}}

\renewcommand{\epsilon}{\varepsilon}

\newcommand{\sH}{\mathscr{H}}
\newcommand{\sJ}{\mathscr{J}}
\newcommand{\sV}{\mathscr{V}}

\newcommand{\rF}{\mathfrak{F}}
\newcommand{\rG}{\mathfrak{G}}

\newcommand{\edot}{\boldsymbol{\cdot}}

\newcommand{\tp}{\tilde{p}}
\newcommand{\tq}{\tilde{q}}
\newcommand{\tf}{\tilde{f}}

\newcommand{\tY}{\widetilde{Y}}
\newcommand{\tN}{\widetilde{N}}
\newcommand{\tbX}{\widetilde{\mathbf{X}}}
\newcommand{\tM}{\widetilde{M}}
\newcommand{\tX}{\widetilde{X}}

\newcommand{\D}{\mathrm{D}}

\newcommand{\Id}{\operatorname{Id}}

\usepackage[english]{babel}
\usepackage{geometry}
\usepackage{multicol}

\title{Bridging rough stochastic differential equations and rough semimartingales}

\makeatletter
\newcommand{\authorsym}[1]{\textsuperscript{\@fnsymbol{#1}}}
\makeatother 

\author{Jannis R. Dause\authorsym{1}}
\email{jannis.dause@maths.ox.ac.uk}

\author{Peter K. Friz\authorsym{1}\textsuperscript{,}\authorsym{2}}
\email{friz@math.tu-berlin.de}

\author{Jost Pieper \authorsym{3}}
\email{jost.pieper@durham.ac.uk}
\thanks{\authorsym{1} Institut für Mathematik,
Technische Universität Berlin, Berlin, Germany.}

\thanks{\authorsym{2} Weierstraß-Institut,
Berlin, Germany.}
\thanks{\authorsym{3}
Department of Mathematical Sciences, Durham University, Durham, UK}
\date{\today}


\begin{document}
    
\begin{abstract}
We bridge the gap between the mixed-moment theory of càdlàg rough stochastic differential equations (short:RSDEs), introduced by Friz, Hocquet and Lê (2021) and extended to the càdlàg setting by Allan and Pieper (2026), and the theory of rough semimartingales (short:RSMs), introduced by Friz and Zorin–Kranich (2023), by investigating the intersection of these two theories for the first time. By introducing the concept of \emph{good} solutions to RSDEs, solutions which are simultaneously both mixed-moment solutions and RSMs, we develop a solution concept, which allows for good control on both moment-based and pathwise properties of solutions. The well-posedness of good solutions is proven using a fixed point argument on the space of $L^{2}$-martingales by suitably leveraging previously known mixed-moment estimates. Based on these insights we further characterize a large class of good solutions of semi-linear RSDEs through a novel \emph{rough stochastic Duhamel's principle}.  
\end{abstract}
\maketitle
\tableofcontents
\section{Introduction}
\label{section:introduction}
 \addtocontents{toc}{\protect\setcounter{tocdepth}{1}}
The theory of \emph{rough stochastic differential equations} (RSDEs), introduced by Friz, Hocquet and L\^e
\cite{FHL21}, provides an intrinsic solution theory for dynamics driven simultaneously by
martingale noise $N$ and a deterministic rough path $\mathbf{X}$,
\begin{equation}
\label{introduction:RSDE}
    \dd Y_t = b_t(Y_t) \dd t +\sigma_t(Y_{t-}) \dd N_t \,+\, f_t(Y_{t}) \dd \mathbf{X}_t; \quad Y_{0}=\xi, 
\end{equation}
providing a hybrid theory of It\^{o}'s theory of stochastic differential equations and Lyons' theory of rough paths \cite{Lyons1998}. 
RSDEs have successfully been applied in stochastic filtering \cite{AllanPieperTeichmann2025, BFLZ25}, stochastic control \cite{horst2025pontryagin, FLZ24x, BDDF26}, mean-field games \cite{FGHT26, bayraktar2026meanfieldgamesroughcommon}, option pricing \cite{BBFP25}, stochastic numerics \cite{DauseFrizJentzenSong2026} and Hörmander theory for conditioned SDEs \cite{BCN24}.
On the theoretical side it also made contributions to the study of rough PDEs and stochastic evolution equations \cite{liang2024mildsolutionsemilinearrough, BuginiFrizStannat2024, becherer2026pathwiseapproachsemilinearspdes, bugini2026youngparabolicpdesmartingale}, BSDEs \cite{song2025backwardstochasticdifferentialequations, becherer2025roughbackwardsdesdiscontinuous, friz2026roughbackwardstochasticdifferential} and mean-field systems with common noise \cite{FHL25p, BFS25}, \\
The original well-posedness theory of \cite{FHL21} and its extension to the càdlàg setting \cite{AllanPieper2026} rely on scales
of {\em stochastic controlled rough paths} (SCRPs) and delicate estimates based on mixed-moment versions of the stochastic sewing lemma \cite{LE20}. While it is able to establish well-posedness to \eqref{introduction:RSDE} under optimal regularity assumptions on $(b, \sigma, f)$, its strong reliance on moments leads to serious restrictions when one aims to work with solutions of \eqref{introduction:RSDE} in a pathwise manner. These restrictions are particularly grave in the setting of \cite{AllanPieper2026} when considering $N$ with Lévy-type jumps, as the moments of the jumps are usually not sufficiently regular.  \\
Contrary to this, the somewhat parallel 
theory of {\em rough semimartingales} (RSMs) \cite{fzk23} is rooted
in multilevel BDG-type estimates for joint lifts of martingales and rough paths. As shown in \cite{DauseFrizJentzenSong2026} the theory of (strong) RSMs provides a powerful and relatively easy approach to constructing a pathwise \emph{rough stochastic calculus} by explicitly formalizing the connection of RSMs to Gubinelli's \emph{controlled paths} \cite{GUBINELLI200486}. However, naive attempts to solving \eqref{introduction:RSDE} in the space of RSMs lead to in-optimal regularity assumptions on $(b, \sigma, f)$ (with some recent improvements in \cite{friz2026roughbackwardstochasticdifferential}).\\
As in classical stochastic analysis both pathwise and moment-based tools are extremely relevant. We therefore propose the novel class of \emph{good} solutions to \eqref{introduction:RSDE}, which are \textbf{both} (strong) RSMs in the sense of \cite{fzk23, DauseFrizJentzenSong2026} \textbf{and} $L^{m,n}$-solutions in the sense of \cite{FHL21}. The present work then aims to answer the following question\footnote{Already posed independently in \cite{AllanPieper2026, DauseFrizJentzenSong2026} in other forms.}:
\medskip

\begin{center}
\textbf{Under which conditions can we construct solutions to \eqref{introduction:RSDE}, which are \emph{good}}?    
\end{center}

\medskip

Inspired by the classical fixed point argument for McKean--Vlasov SDEs (\cref{inspiration_McKean}), the central idea to answer this question is to solve the
additive-noise RSDE
\begin{equation}
\label{introduction:additive}
    \dd Y_{t}^{M} = b_{t}(Y_{t}^{M}) \dd t + M_{t} + f_{t}(Y^{M}_{t}) \dd \mathbf{X}_{t}; \quad Y^{M}_{0}=\xi 
\end{equation}
pathwise i.e. as a random RDE along the joint rough path lift of $\mathbf{X}$ and a given
martingale $M$. We then obtain the full dynamics \eqref{introduction:RSDE} through a contraction
\begin{equation*}
    \Phi: \mathcal{M}^{2}\to \mathcal{M}^{2}; \quad M \mapsto \int \sigma_{s} (Y_{s-}^{M})\dd N_{s}
\end{equation*}
on the martingale space $\mathcal{M}^2$. All solutions obtained in such a manner can directly be seen to be RSMs and in particular \emph{good} (assuming suitable regularity of the data). \\
To establish the contractivity of $\Phi$ (on small time intervals) the entire
argument hinges on establishing a stability estimate of the form 
\begin{equation}
 \label{introduction:stability}
\left  \| \sup_{t \in [0,T]} |Y^{M}_{t} - Y^{\bar M}_{t}| \right \|_{L^{2}} \lesssim \left \| \sup_{t \in [0,T]} |M_{t} - \bar M_{t}| \right \|_{L^2},
\end{equation}
or suitable variants thereof. We establish this non-trivial estimate in two regimes: for bounded, and deterministic coefficient fields $f$ and for linear, possibly stochastic $f$ (\eqref{introduction:RSDE} is then called "semi-linear") via leveraging suitably adjusted moment estimates from  \cite{FHL21, BCN24,AllanPieper2026}. The main technical tools allowing us to use these estimates are suitable càdlàg $p$-variation Kolmogorov-type criteria and a weak stability property of the solutions with respect to discretisations of the rough path, which together allow us to identify $L^{m,n}$- and RSM-solutions of \eqref{introduction:additive}. For linear and deterministic $f$ we additionally provide a semi-explicit characterisations of solutions of \eqref{introduction:RSDE} \& \eqref{introduction:additive} through a novel \emph{rough stochastic Duhamel's principle}.\\
Due to the aforementioned relevance in the càdlàg/ Lévy-noise setting the entire manuscript is formulated in terms of càdlàg rough (stochastic) analysis (and thus of course in $p$-Var). However, we emphasize that all of the major results are novel in the continuous ($p$-Var and $\alpha$-Hölder\footnote{Note that the corresponding $\alpha$-Hölder results come however with the usual loss of pathwise regularity due to Kolmogorov's criterion.}) setting as well.  
\subsection{Outline:}
In \cref{section:notation} we collect notation used throughout this work. In \cref{section:Fundamentals of Rough Stochastic Analysis} we introduce the theory of rough stochastic analysis by unifying both the mixed-moment theory of \cite{FHL21} and RSM-theory of \cite{fzk23, DauseFrizJentzenSong2026}. This culminates in the definition of good solutions to \eqref{introduction:RSDE}. \cref{Section:main_theorem} introduces our well-posedness approach based on the aforementioned fixed point argument. \cref{section:On semi-linear RSDEs} is concerned with the semi-linear setting ($f$ linear in \eqref{introduction:RSDE}) and provides both the necessary stability estimates (\cref{subsection:Well-posedness and stability of semi-linear RSDEs}) for \eqref{introduction:additive} and rough stochastic Duhamel's representation (\cref{subsection: Duhamel formula}). \cref{section:Well-posedness of fully non-linear RSDEs} provides the stability estimates for \eqref{introduction:additive} in the non-linear setting ($f$ being non-linear in \eqref{introduction:RSDE}). \cref{appendix:deterministic_rough_path_theory} extends the theory càdlàg rough paths of \cite{Friz_Zhang_2018} by introducing strongly controlled càdlàg rough paths and a corresponding It\^{o}-formula. \cref{appendix:section_flow_transform} explores well-posedness of good solutions by the method of \emph{rough flow-transforms}. 

In \cref{appendix: A discrete approximation approach to consistency} we introduce the weak stability property used throughout \cref{section:Well-posedness of fully non-linear RSDEs} and \cref{subsection:Well-posedness and stability of semi-linear RSDEs}. At last in \cref{section:Further technical results} we collect some useful results applied throughout this work namely a Kolmogorov-criterion in càdlàg $p$-Var for remainders and a càdlàg rough Grönwall lemma.
\subsection{Main methods:}
We give a brief description of the arguments for semi-linear RSDEs \cref{section:On semi-linear RSDEs} and RSDEs with bounded coefficients \cref{section:Well-posedness of fully non-linear RSDEs}.
\subsubsection{RSDEs with bounded coefficients:}
In \cref{section:Well-posedness of fully non-linear RSDEs}, both the noise $N$ and the mixed-moment solutions belong to $V^{p}L^{m,n}$ with $n=\infty$. The connection between this space and $\mathrm{BMO}^{p\text{-var}}$, discussed in \cite{Le2022a,AllanPieperTeichmann2025}, allows us to improve the integrability of the latter to an exponent $m$ exceeding $p$. Applying the rough stochastic Doob--Meyer decomposition \cref{prop: Doob--Meyer} with this improved integrability, we can show that any $L^{m,\infty}$-solution is good; see \cref{theorem: RSDE solution is L^q RSM}. Alternatively we show under the same assumptions how to establish this result based on the martingale fixed point argument of \cref{subsection: Well-posedness via Martingale fixed points}; see \cref{corollary: the fixed point argument is applicable for mixed moment assumptions}.  
\subsubsection{Semi-linear RSDEs:}
\cref{section:On semi-linear RSDEs} is less straightforward since here for noise $N\in V^{p}L^{m}$ (even for  $N \in V^{p}L^{m, \infty}$) the solutions can only be expected to live in $V^{p}L^{m}=V^{p}L^{m,m}$. The critical coefficient to consider here is $c \coloneqq \nicefrac{m}{p}$, since for $c\leq 1$ the $p$-Var Kolmogorov criterion breaks down; see \cite[Example A.6]{AllanPieperTeichmann2025} and \cref{example:Besov embedding}. While for $c>1$ we follow the same arguments as in the bounded case, the case $c\leq 1$ requires a different method to identify the $L^{0}$- and $L^{m}$-solutions, which we need in order to verify the stability estimate \cref{ref:Key} in the martingale fixed point argument \cref{subsection: Well-posedness via Martingale fixed points}. In particular we show this identification by introducing a weak stability property of solutions to linear RSDEs in \cref{appendix: A discrete approximation approach to consistency} (when available). Our methodology for \cref{section:On semi-linear RSDEs} is summarized in \cref{fig:proof-structure} below.
\newcommand{\LabStart}{semi-linear RSDE}
\newcommand{\CondP}{$p < m$?}        

\begin{figure}[htbp]
\centering
\begin{tikzpicture}[xscale=1,yscale=1]
  \tikzset{
    noeud/.style={minimum width=2cm,text width=4.4cm,minimum height=0.8cm,
      rectangle,rounded corners=5pt,draw,fill=yellow!50,text=black,
      font=\bfseries,text centered,text badly centered},
    core/.style={minimum width=2cm,text width=2.4cm,minimum height=1cm,
      rectangle,rounded corners=5pt,draw,fill=lime!50,text=black, text centered,text badly centered},
    main/.style={minimum width=2cm,text width=2.4cm,minimum height=1cm,
      rectangle,rounded corners=5pt,draw,fill=red!30,text=black,
      font=\bfseries,text centered,text badly centered},
    exa/.style={minimum width=2cm,text width=4cm,minimum height=0.8cm,
      rectangle,rounded corners=5pt,draw,fill=cyan!30,text=black,
      font=\bfseries,text centered,text badly centered},
    cond/.style={diamond,aspect=1.6,draw,fill=cyan!30,text=black,
      font=\bfseries,text centered,inner sep=2pt},
    startbox/.style={minimum width=2cm,minimum height=0.8cm,
      rectangle,rounded corners=5pt,draw,fill=yellow!50,text=black,
      font=\bfseries,text centered,text badly centered},
    dnode/.style={ellipse,draw,dashed,thick,black!60,minimum width=1.6cm,
      minimum height=1cm,font=\bfseries},
    regime/.style={font=\Large\bfseries},
    fleche/.style={->,>=stealth,thick},
    note/.style={->,>=stealth,thick,draw=black!50},
  }

  \node[startbox] (start) at (7,10.3) {\LabStart};
  \node[core] (b1) at (7,8.8) {linear RSDE; \cref{subsection: Well-posedness via Martingale fixed points}};
  \draw[fleche] (start) -- (b1);

  \node[cond] (cond) at (7,7.1) {\CondP};
  \draw[fleche] (b1) -- (cond);

  \node[core] (q) at (3.7,6.1) {Apply weak stability of solutions; \cref{appendix: A discrete approximation approach to consistency}};
  \draw[fleche] (cond.west) -- node[above,pos=0.3,font=\itshape] {no} (q.north);
  \node[core] (b3) at (10.3,6.1) {Apply Kolmogorov criterion; \cref{subsection:A Kolmogorov criterion and rough stochastic Doob-Meyer decomposition}};
  \draw[fleche] (cond.east) -- node[above,pos=0.3,font=\itshape] {yes} (b3.north);

  \node[main] (b4) at (5.9,2.7) {linear $L^{m}$- and $L^{0}$-solution agree; \cref{cor:good-semilinear-deterministic-jumps}};
  \draw[fleche] (q.south) -- (b4.north west);
  \draw[fleche] (b3.south) -- (b4.north east);

  \node[main] (b5) at (13,2.7) {Stability estimate \cref{ref:Key} holds};
  \node[core] (b6) at (9.5,2.7)
  {Apply $L^{m}$ estimates; \cref{proposition: exsitence of additive linear RSDEs with random coeff}};
  \draw[fleche] (b4) -- (b6);
  \draw[fleche] (b6)--(b5); 
  \node[main] (D) at (2.3,2.7) {Rough stoch. Duhamel princ.; \cref{subsection: Duhamel formula}};
  \draw[fleche,dashed,black!50] (D.south) to[out=270,in=270,looseness=0.4] (b5.south);
  \draw[fleche,dashed,black!50] (b4.west) to (D.east);

  \draw[fleche,rounded corners=8pt] (b5.east) -- ++(1.0,0) |-
        node[pos=0.28,sloped,below,font=\itshape]
        {well-posedness of good solutions} (start.east);

  \draw[fleche,dashed,rounded corners=6pt, black!50] (D.north) -- ++(0,0.6) -- ++(-1.8,0) |-
        node[pos=0.79,above,font=\itshape] {explicit represent.} (b1.west);

  \draw[fleche,dashed,rounded corners=6pt, black!50] (D.west) -- ++(-0.97,0) |-
        node[pos=0.81,above,font=\itshape] {semi-explicit represent.} (start.west);
\end{tikzpicture}
\caption{Structure of the argument for semi-linear RSDEs; see \cref{section:On semi-linear RSDEs}. The lime blocks denote methods/principles used, the diamond an "if-statement", the red block denote claims. The dashed lines are only available 
when $(f, f')$ is deterministic.}
\label{fig:proof-structure}
\end{figure}
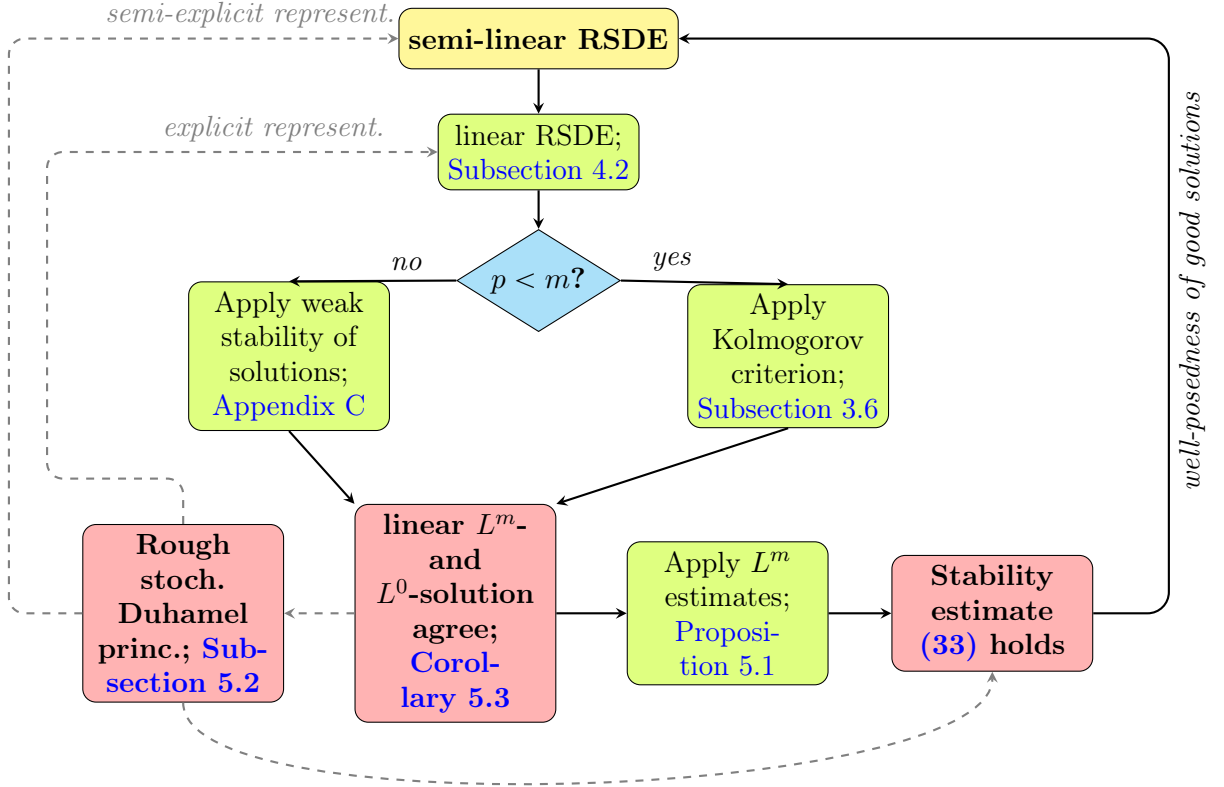
\subsection{Comparison with the literature}
We round off this introduction by discussing this works contribution to the theory of rough (stochastic) analysis.
\label{section:literature}
\subsubsection{Rough stochastic calculus:}
In \cref{section:Fundamentals of Rough Stochastic Analysis} we, for the first time, consider the mixed-moment rough stochastic calculus  of \cite{FHL21} and RSM-based rough stochastic calculus of \cite{fzk23, DauseFrizJentzenSong2026} in a unifying framework, by introducing the notion of \emph{good SCRPs}. We further extend many main results of \cite[Section 4]{DauseFrizJentzenSong2026} to the càdlàg setting. We provide a Fubini theorem (\cref{prop:rough_stochastic_fubini}), rough stochastic It\^{o}- /product formula (\cref{prop:rsIto} \& \cref{corollary:rs_product_formula}) as well as a rough stochastic Doob-Meyer decomposition (\cref{prop: Doob--Meyer}) in the càdlàg setting, which were previously only known in the continuous case \cite{FHL21, BuginiFrizStannat2024, FGHT26}. We emphasize that our framework is built in such a way that all such results formulated in the language of RSMs can be easily extended to analogous mixed-moment results by assuming suitable moments as highlighted throughout this section. \cref{thm: Kolmogorov for some remainders} extends the càdlàg Kolmogorov criterion for paths \cite[Theorem 2.10]{AllanPieperTeichmann2025} to more general two-parameter processes.
\subsubsection{Semi-linear RSDEs:}
In \cite{BCN24, horst2025pontryagin, FGHT26} the authors investigated linear/ semi-linear/linear-affine RSDEs in the context of Malliavin calculus, stochastic control and mean-field games in the continuous setting. For the first time the present work provides well-posedness of solutions to such \footnote{Note, that all of the results that we attain for linear stochastic controlled vector fields can easily be extended to the case of affine-linear stochastic controlled vector fields with deterministic affine term.} RSDEs in the càdlàg setting (\cref{subsection:Well-posedness and stability of semi-linear RSDEs}) notably also in the full regularity scale $m \in [2, \infty)$ and $p \in [2,3)$ (which means assuming $N\in V^{p} L^{m, \infty}$) . In particular we show that, as a consequence of our approach being based on the martingale fixed point ansatz from \cref{subsection: Well-posedness via Martingale fixed points}, all solutions constructed here are \emph{good}. In \cref{proposition: good solutions for Jacobians of flows} we show that the aforementioned claims are applicable to the main application of \cite[Section 3]{BCN24} being well-posedness of linear RSDEs corresponding to Jacobians of RSDE-flows. Our claims are applicable to the semi-linear affine RSDEs of \cite{horst2025pontryagin} given that both $b , \sigma$ have bounded growth constants. In \cref{subsection: Duhamel formula} we further provide a characterization of solutions to semi-linear RSDEs with deterministic rough vector fields, as a novel rough stochastic Duhamel principle. Due to the non-invertibility of solutions to càdlàg RDEs we stress that its proof is non-trivial compared to the analogous result in the continuous setting.
\subsubsection{RSDEs with bounded coefficients :}
Solutions to RSDEs with bounded coefficients  have previously been investigated in \cite{FHL21, AllanPieper2026} with the latter being the first work in the càdlàg setting (under optimal regularity conditions). \cref{section:Well-posedness of fully non-linear RSDEs} shows that, by leveraging bounds from the latter work, those solutions can be shown to be \emph{good}, which is particularly important when considering Lévy-noise. All of the results are again obtained on the full regularity scale $p \in [2,3)$ and $ m \in [2, \infty)$ (which means assuming $N 
\in V^{p}L^{m,n}$ with $n=\infty$). Extensions to the McKean-Vlasov setting of \cite{FHL25p} can be easily (at least with measure dependence in the It\^{o} data $b,\sigma$) done by performing an additional linearization procedure in the measure dependence analogous to the martingale fixed point ansatz \cref{subsection: Well-posedness via Martingale fixed points} following \cref{inspiration_McKean}.
\par 
Our main contributions to the theory of RSDEs are also summarized in \cref{table: main results} below.
\subsubsection{Càdlàg rough analysis:}
In \cref{appendix:deterministic_rough_path_theory} we extend the theory of càdlàg rough paths initiated in \cite{Friz_Zhang_2018} by introducing the notion of \emph{strongly controlled rough paths} in the càdlàg regime and investigating properties of the backward semi-flow of RDE-solutions. The rough stochastic Grönwall Lemma \cref{lemma: gronwall type result p var} of course also applies (suitably simplified) in a deterministic form and generalizes \cite[Lemma 2.12]{DeyaGubinelliHofmanovaTindel2019} for irregular controls. 
\subsection{Acknowledgements:} The authors thank Andrew L. Allan for helpful discussions. JRD is supported by Deutsche Forschungsgemeinschaft (DFG) - Project-ID 410208580 - IRTG2544 (”Stochastic Analysis in Interaction”). PKF acknowledges support by the Deutsche Forschungsgemeinschaft (DFG, German Research Foundation)
under Germany’s Excellence Strategy – The Berlin Mathematics Research Center MATH+ (EXC-2046/1, project ID: 390685689). PKF acknowledges support from DFG CRC/TRR 388 ``Rough Analysis, Stochastic Dynamics and Related Fields". 

\subsection{Tool and computational resource disclosure:}
In accordance with the Leiden Declaration on Artificial Intelligence and Mathematics, the authors declare use of LLMs in particular OpenAI's GPT-6.0 Astra and SAIR's OpenMath Model to improve some of the presentation of this manuscript. None of the claims and proofs thereof in this manuscript were created by LLMs.

\medskip

\begin{table}[H]
    \centering
    \caption{Overview of the main results.}\label{table: main results}
    \label{table:overview_RSDEs}
    \begingroup
    \footnotesize
    \setlength{\tabcolsep}{3pt}
    \renewcommand{\arraystretch}{1.2}
    \begin{tabularx}{\textwidth}{
        @{}
        >{\hsize=.50\hsize\linewidth=\hsize
          \raggedright\arraybackslash}X
        |
        >{\hsize=.38\hsize\linewidth=\hsize
          \raggedright\arraybackslash}X
        |
        >{\hsize=2.12\hsize\linewidth=\hsize
          \raggedright\arraybackslash}X
        @{}
    }
        \hline
        \textbf{Coefficient $(f, f')$}
        & \textbf{Noise regime\textsuperscript{a}}
        & \textbf{Main results}
        \\
        \hline

        \multicolumn{3}{@{}l}{
            \textbf{Bounded}
        }
        \\[3pt]

        Deterministic controlled
        &
        $p\in (2,3)$, $m\in [2,\infty)$, $n=\infty$.
        &
        Mixed-moment solutions are scRSMs and thus good. Identification with unique $L^{0}$-solution under additional pathwise
        controlledness assumptions; \cref{theorem: RSDE solution is L^q RSM,%
        thm:consistency_of_solutions I}
        \\[8pt]

        \hline
        \multicolumn{3}{@{}l}{
            \textbf{Linear\textsuperscript{b}}
        }
        \\[3pt]

        Deterministic controlled
        &
        $p\in (2,3)$, $m \in [2,\infty)$.
        &
        Well-posedness of good solutions 
        and a Duhamel principle.
        \par\smallskip
        \cref{cor:good-semilinear-deterministic-jumps} and \cref{thm:semilinearDuhamel}
        \\[5pt]

        Rough semimartingale 
        &
        $p\in (2,3)$, $m \in [2,\infty)$.
        &
        Well-posedness of good solutions under additional jump assumptions.\textsuperscript{c}
        \par\smallskip
        \cref{cor:good-semilinear-deterministic-jumps} and \ref{corollary: consistency for linear RSM with deterministic jumps}. 
        \\[5pt]

        Coefficients evaluated along
        an RSDE solution
        &
        $p \in (2,3)$,
        $m\in [2,\infty)$. 
        &
        Well-posedness of good solutions of the specified
        auxiliary linear equations; \cref{proposition: good solutions for Jacobians of flows}
        \\

        \hline
    \end{tabularx}

    \smallskip
    \begin{minipage}{\textwidth}
        \footnotesize

        \smallskip
        \textsuperscript{a}Refers to the choice of coefficients in $N\in V^{p}L^{m,n}$.

        \smallskip
        \textsuperscript{b}Whenever we talk about linear vector fields we assume that $n=m$ unless specified otherwise.
        
        \smallskip
        \textsuperscript{c}The controlled part of the coefficient,
        after separating off its martingale component, and its
        controlled derivative have deterministic jump times;
        see the referenced results for the precise assumptions.
    \end{minipage}
    \endgroup
\end{table}

\section{Notations \&  conventions}
\label{section:notation}
\subsection{General Notations:}  We write $a \lesssim b$ to express that there exists a non-negative constant $C\ge 0$ such that $a \le Cb$. If such a constant depends on some parameters $\theta_1,\dots,\theta_N$ and we want to emphasize this fact, we write $a \lesssim_{\theta_1,\dots,\theta_N} b$.
\subsection{Linear Algebra:}
In the following we assume familiarity with direct sums and tensor products of vector spaces see \cite{bayer_introduction_2026}. 
In the following let $V, W, U$ be some finite-dim.\ real Banach spaces. We denote their direct sum by $V\oplus W$ and tensor product by $V \otimes W$. We denote by $\mathcal{L}(V; W)$ the space of linear maps $A: V \to W$ and by $\operatorname{Bil}(V \times W; U)$ the space of bilinear maps $A: V \times W\to U$.  For any $v\otimes w\in V\otimes W$ denote by $(v \otimes w)^{\top}\coloneqq w \otimes v$ and for any $A\in \mathcal{L}(V \otimes W; U)$ we denote by $A^{\top}$ the unique linear map $A^{\top}\in \mathcal{L}(W\otimes V; U)$ such that $A(v \otimes w)=A^{\top}(w \otimes v)$ for any $v \in V, w \in W$. For $v_{1}\otimes v_{2} \in V\otimes V$  denote $\operatorname{Sym}(v_{1}\otimes v_{2})\coloneqq \frac{1}{2}\left(v_{1}\otimes v_{2}+ (v_{1} \otimes v_{2})^{\top}\right)$. We denote by $\mathcal{S}(V \otimes V; W)\subset \mathcal{L}(V \otimes V; W)$ the subset of symmetric linear maps.\newline
Given vector spaces $V_{1}, V_{2}$ and $A_{i,j}\in \mathcal{L}(V_{i}\otimes V_{j}; W)$ for $1\leq i,j\leq 2$ we define the block-matrix operator $A \in \mathcal{L}((V_{1} \oplus V_{2})^{\otimes 2}; W)$ by 
\begin{equation*}
    A((v_{1}\oplus v_{2})^{\otimes 2})\coloneqq \left(\begin{array}{cc}
    A_{1, 1} & A_{1,2}\\
    A_{2, 1} & A_{2,2}
    \end{array}\right): ((v_{1}\oplus v_{2})^{\otimes 2})\coloneqq \sum_{1\leq i, j\leq 2} A_{i,j}(v_{i} \otimes v_{j}),
\end{equation*}
and for $A_{i}\in \mathcal{L}(V_{i}; W)$ the block-vector operator $ A\in \mathcal{L}(V_{1}\oplus V_{2}; W)$ by
\begin{equation*}
    A(v_{1} \oplus v_{2}) \coloneqq \left(\begin{array}{c}
    A_{1}\\
    A_{2}
    \end{array}
    \right)\cdot
    (v_{1} \oplus v_{2})\coloneqq A_{1} v_{1} + A_{2} v_{2}.
\end{equation*}
One immediately checks that these notations are generalizations of the Fréchet- and scalar-product from classical linear algebra. It will however often be helpful to perform calculations in the more compact tensor notation even when working in $\mathbb{R}^{d}$. To this end recall that for $x \in \mathbb{R}^{d_{x}}, y \in \mathbb{R}^{d_{y}}$ the tensor-product is defined by $x \otimes y\coloneqq x y^{\top}$. 
\subsection{Function spaces:}
We denote by $\mathcal{L}(V;\bar{V})$ the space of linear and continuous functions from $V$ to $\bar{V}$ Banach spaces.
The latter is a Banach space, if endowed with the norm $|T|_{\mathcal{L}} := \sup_{x \in V, |x|_V\le 1} |Tx|_{\bar{V}}$. 

We identify $\mathcal{L}(\R^m;\R^n)$ and the tensor product $\R^n \otimes \R^m$ with the matrix space $\R^{n \times m}$. 
The following general identifications also hold: $\mathcal{L}(V \otimes \bar{V};\tilde{V}) \equiv \mathcal{L}(V;\mathcal{L}(\bar{V};\tilde{V})) \equiv \mathcal{L}^2(V \times \bar{V};\tilde{V}),$
where $\mathcal{L}^2(V \times \bar{V};\tilde{V})$ denotes the space of bilinear and continuous maps from $V \times \bar{V}$ to $\tilde{V}$.
We denote by $DF:V \to \mathcal{L}(V,\bar{V})$ the first Fréchet derivative of $F$. If not otherwise specified, in this paper we always refer to differentiability in Fréchet sense simply as differentiability. Let $k\geq 0$ be an integer and let $\gamma\in (k,k+1]$.
We say that $f \in \Lip^{\gamma}(V; W)$ if $f: V \to W$ such that $f \in C^{k}_{b}(V; W)$ and $D^{k}f\in \mathcal{C}^{\gamma -k}(V; \mathcal{L}(V^{\otimes k}; W))$. For $\alpha \in (0, 1], m\in [1, \infty)$  denote by $\mathcal{B}^{\alpha}_{m}(V)$ the Besov space of paths $X: [0,T]\to V$ with $
    \int_{0}^{T}\int_{0}^{T} \left(\frac{|\delta X_{u,v}|}{|v-u|^{\alpha+\frac{1}{m}}}\right)^{m} \dd u \dd v< \infty$.
\subsection{Paths:}
For any $T>0$ we define  $\Delta_{I} := \{ (s,t) \in I\times I \ | \ s < t\}$ for any subinterval $I\subseteq [0,T]$. 
For any Banach space valued path i.e. a map $[0,T]\ni t \mapsto X_t \in W$, we define its increment as the two parameter function
$\delta X_{s,t} := X_t - X_s$ for $ (s,t) \in \Delta_{[0,T]}$. Further we define (if well-defined) for any $t \in [0,T]$ the (forward-) jump $\Delta X_{t} \coloneqq \lim_{s \uparrow t} \delta X_{s, t}$ with $\Delta X_{0}\equiv 0$ by convention.
Similarly, given a two-parameter function  $A:\Delta_{[0,T]} \to W$, $(s,t) \mapsto A_{s,t}$, we define $\delta A_{s,u,t}:= A_{s,t} - A_{s,u} - A_{u,t}$, for $s \le u \le t$. Further we define (if well-defined) for any $t \in [0,T]$ the (forward-) jump $\Delta A_{t} \coloneqq \lim_{s \uparrow t} A_{s,t}$ with $\Delta A_{0}\equiv 0$ by convention. For $A: \Delta_{[0,T]}\to W$ and $p \in (0, \infty)$ we define its \textit{p-variation} over $[0,T]$ by 
\begin{equation*}
    \Vert A \Vert_{p; [0,T]}\coloneqq \sup_{\mathcal{P}} \left( \sum_{i=0}^{N-1} |A_{t_{i}, t_{i+1}}|^{p} \right)^{\frac{1}{p}}, 
\end{equation*}
where the $\sup$ is ranging over all partitions $\mathcal{P}=(t_{i})_{i=0}^{N}$ of $[0,T]$. For a path $X: [0,T]\to W$ we denote $\Vert X \Vert_{p;[0,T]}\coloneqq \Vert \delta X \Vert_{p; [0,T]}$ and $\operatorname{Osc}(X; [s, t])\coloneqq \sup_{s\leq u \leq v \leq t}|\delta X_{u,v}|$ for any $(s,t) \in \Delta_{[0,T]}$. $A: \Delta_{[0,T]}\to W$ is 
\textit{càdlàg} (continue à droite, limite à gauche) if for any $(s, t) \in \Delta_{[0,T]}$ both $[0,t]\ni r\mapsto A_{r,t}$ and $[s, T]\ni r \mapsto A_{s,r}$ are continuous from the right with existing left limits everywhere. If $A$ is càdlàg and $\Vert A \Vert_{p; [0,T]}< \infty$, we say $A\in V^{p}([0,T]; W)$. If a path $X$ is càdlàg and $\Vert X \Vert_{p; [0,T]}< \infty$, then we say  $X\in V^{p}([0,T]; W)$. We denote the space of such paths which are additionally continuous by $V^{p,c}([0,T]; W)$. For $A: \Delta_{[0,T]}\to W$ and a partition $\mathcal{P}$ of $[0,T]$ we say that the sum $\sum_{[u,v]\in \mathcal{P}} A_{u,v}$ converges to $K \in W$ in Mesh Riemann-Stieltjes (short:MRS) sense, if for any $\epsilon > 0$ there is $\delta >0$ such that for any partition $\cP$ of $[0,T]$ with $|\mathcal{P}|< \delta$ it holds that $|\sum_{[u,v]\in \mathcal{P}} A_{u,v}- K |< \epsilon$. We say that it converges in Refinement Riemann-Stieltjes (short: RRS) sense, if for any $\epsilon>0$ there is a partition $(\mathcal{P}_{\epsilon})$ such that for any refinement $\mathcal{P}_{\epsilon}\subset \mathcal{P}$ it holds that $|\sum_{[u,v]\in \mathcal{P}} A_{u,v}- K |< \epsilon$. Let $w: \Delta_{[0,T]}\to \mathbb{R}_{+}$. Then we call $w$ a \textit{mild control} if it is increasing in $t$, decreasing in $s$, vanishes on the diagonal and for any $\delta>0$ there are only finitely many $u \in [0,T]$ such that $w(u-, u+)>\delta$. We call $w$ a \textit{control} if it is super-additive i.e. $w(s, u)+ w(u,t) \leq w(s, t)$ for any $s\leq u \leq t$. Any control is also a mild control.
\subsection{Rough Paths}
We refer to \cite{Friz_Zhang_2018} for the standard reference on càdlàg rough paths. In particular for $V$ some finite-dim. vector space and $p \in [1, \infty)$ we denote the set of $V$-valued càdlàg rough paths by $\mathbf{X}=(X, \mathbb{X})\in \mathscr{V}^{p}([0,T]; V)$ and their brackets by $\delta [\mathbf{X}]_{0, t}\coloneqq (\delta X)_{0, t} \otimes (\delta X)_{0,t}- 2 \operatorname{Sym}(\mathbb{X}_{0, t})$.  
The space of   
$X$-controlled $W$-valued càdlàg rough paths is denoted by $(Y, Y')\in \mathscr{V}^{p}_{X}([0,T]; W)$ i.e. the space of càdlàg $(Y, Y'):[0,T]\to W \times \mathcal{L}(V; W)$ such that for $R^{Y, Y'}_{s,t}\coloneqq \delta Y_{s,t}- Y'_{s} \delta X_{s,t}$ it holds 
\begin{equation*}
    \Vert Y, Y'\Vert_{X; p; [0,T]}\coloneqq \Vert \delta Y' \Vert_{p; [0,T]}+ \Vert R^{Y, Y'}\Vert_{\frac{p}{2}; [0,T]}< \infty. 
\end{equation*}
Further $\mathscr{V}^{p, c}$ and $\mathscr{V}^{p, c}_{X}$ denote the subspaces of continuous rough paths resp. continuous $X$-controlled rough paths. 
\subsection{Probability theory:}
Given a probability space $(\Omega,\mathfrak{F},\mathbb{P})$ and an integrability exponent $m\in [1,\infty)$, we denote by $\|\cdot\|_{L^m;V}$ - or simply by $\|\cdot\|_{L^m}$, if clear from the context - the usual norm on the Lebesgue space $L^m(\Omega;V)$ of $V$-valued random variables. The norm $\|\cdot\|_{L^\infty}$ denotes the essential supremum norm.
If $V=\R$, we sometimes write $L^m(\Omega)$ instead of $L^m(\Omega;\R)$. For a given sub-$\sigma$-field $\rG \subseteq \rF$, the space $L^m(\Omega,\rG;V)$ is the space of $\rG$-measurable random variables in $L^m(\Omega;V)$.\\
Given a fixed time horizon $T>0$ and a filtered probability space $(\Omega, \mathfrak{F}, (\mathfrak{F}_{t})_{t \in [0,T]}, \mathbb{P})$. We also adopt the shorthand $\E_s [\, \cdot \,] := \E [\, \cdot \,|\, \rF_s]$ for the conditional expectation at time $s \in [0,T]$.
As in \cite{FHL21, AllanPieper2026}, for $m \in [1,\infty)$, $n \in [m,\infty]$, $s \in [0,T]$ and a random variable $Y$, we write
\[ \|Y\|_{m,n,s} := \big\| \E_s [|Y|^m]^{\frac{1}{m}} \big\|_{L^n} \]
and we write $L^{m,n}_s$ for the space of random variables $Y$ such that $\|Y\|_{m,n,s} < \infty$ we will sometimes call these spaces \emph{mixed-moment spaces}.

\smallskip

For convenience, we recall the following properties of the $\|\cdot\|_{m,n,s}$ norm, the proofs of which may be found in \cite[Proposition 2.3 \& Lemma 2.4]{FHL21}.

\begin{itemize}
    \item $(L^{m,n}_s,\|\cdot\|_{m,n,s})$ is a Banach space.
    \item We have that $\|Y\|_{L^m} \leq \|Y\|_{m,n,s} \leq \|Y\|_{L^n}$ and $\| \E_s[Y] \|_{L^n} \leq \|Y\|_{m,n,s}$.
    \item The norm $\|Y\|_{m,n,s}$ is non-decreasing in each of the variables $m, n$ and $s$.
    \item The following version of H\"older's inequality,
    \begin{equation}\label{eq: Holder's inequality for q,r,s norm}
        \big\|\E_s[|Y||Z|]\big\|_{L^\ell} \leq \big\|\E_s[|Y|^p]^{\frac{1}{p}}\big\|_{L^{p \ell}} \big\|\E_s[|Z|^m]^{\frac{1}{m}}\big\|_{L^{m \ell}} = \|Y\|_{p,p \ell,s} \|Z\|_{m,m \ell,s},
    \end{equation}
    holds for any $\ell \in [1,\infty]$ whenever $p, m \in (1,\infty)$ with $\frac{1}{p} + \frac{1}{m} = 1$.
\end{itemize}
\subsection{Stochastic processes:}
Consider a $V$-valued stochastic process $X: [0,T]\times \Omega \to V$. We say "$X$ is càdlàg" if it is càdlàg almost surely. $X$ is said to be $L^m$-integrable if $X_t$ belongs to $L^m(\Omega;V)$ for any $t \in [0,T]$.
For $p, m \in [1,\infty)$, $n \in [m,\infty]$, and a two-parameter stochastic process $F = (F_{s,t})_{(s,t) \in \Delta_{[0,T]}}$ we say $F$ is $(\rF_t)$-adapted, if for any $(s,t)\in \Delta_{[0,T]}$ we have $F_{s,t} \in \rF_t$, and we write
\begin{equation*}
    \|F\|_{p,m,n,[s,t]} := \bigg(\sup_{\cP \subset [s,t]} \sum_{[u,v] \in \cP} \|F_{u,v}\|_{m,n,u}^p\bigg)^{\hspace{-2pt}\frac{1}{p}} = \bigg(\sup_{\cP \subset [s,t]} \sum_{[u,v] \in \cP} \big\|\E_u[|F_{u,v}|^m]^{\frac{1}{m}}\big\|_{L^n}^p\bigg)^{\hspace{-2pt}\frac{1}{p}}
\end{equation*}
for $(s,t) \in \Delta_{[0,T]}$. For a stochastic process $Y = (Y_t)_{t \in [0,T]}$, we then let
\begin{equation*}
    \|Y\|_{p,m,n,[s,t]} := \|\delta Y\|_{p,m,n,[s,t]}.
\end{equation*}
We write $V^p L^{m,n} = V^p L^{m,n}([0,T];E)$ for the space of stochastic processes $Y \colon [0,T] \times \Omega \to W$ which are c\`adl\`ag almost surely, and are such that $Y_0 \in L^m$ and $\|Y\|_{p,m,n,[0,T]} < \infty$.
In the special case when $n = m$, we also write
\begin{equation}\label{eq: defn p,infty,[0,T] norm}
    \|F\|_{p,m,[s,t]} := \|F\|_{p,m,m,[s,t]} = \bigg(\sup_{\cP \subset [s,t]} \sum_{[u,v] \in \cP} \|F_{u,v}\|_{L^m}^p\bigg)^{\hspace{-2pt}\frac{1}{p}},
\end{equation}
with $\|Y\|_{p,m,[s,t]} := \|\delta Y\|_{p,m,[s,t]}$ and $V^p L^m := V^p L^{m,m}$. We can also define $\|\cdot\|_{p,\infty,[s,t]}$ and $V^p L^\infty$, by simply replacing the $L^m$ norm in \eqref{eq: defn p,infty,[0,T] norm} with an $L^\infty$ norm. We denote for $m \in [1, \infty)$ by $\mathcal{M}^{m}([0,T])$ the space of càdlàg martingales with finite $m$th-moments equipped with the norm $\Vert M \Vert_{\mathcal{M}^{m}([0,T])}\coloneqq \Vert \sup_{t \in [0,T]} |M_{t}|\Vert_{L^{m}}$ (we omit the dependence on $[0,T]$ if clear from context and the dependence on $m$ for $m=1$). Further, for any $p\in [1,\infty)$ and $n\in [m,\infty]$ we write $\cM^{p,m,n}:= \cM \cap V^p L^{m,n}$ and whenever $m=n$ for brevity also $\cM^{p,m}$. Let $V, W$ finite-dim. Banach spaces, $M$ be a $V$-valued (resp. $\mathcal{L}(V; W)$-valued) and $H$ be a $\mathcal{L}(V; W)$-valued (resp. $V$-valued) stochastic process.
We denote by $H \bullet \dd M$ (resp. $\dd M \bullet H$) the $W$-valued stochastic integral of a suitable integrand $H$ against an integrator $M$ in the sense of It\^{o}, if well-defined. For a sequence of stochastic processes $(X^{n})_{n \in 
\mathbb{N}}$ and a process $X$ we denote convergence \emph{uniformly on compacts in probability} (short:u.c.p.) by 
\begin{equation*}
    \operatorname{u.c.p.-}\lim_{n \to \infty} X^{n} = X :\Leftrightarrow \sup_{t \in [0,T]} |X^{n}_{t} - X_{t}|\stackrel{\mathbb{P}; n \to \infty}{\longrightarrow}0,
\end{equation*}
where the latter stands for convergence in probability w.r.t. $\mathbb{P}$.

\subsection{Classifications of RSDEs:}
In the course of this work we will be interested in RSDEs of the form 
\begin{equation}
\label{eq:classification}
    Y_{t}= \xi + \int_{0}^{t} b_{s}(Y_{s}) \dd s+ \int_{0}^{t} \sigma_{s} (Y_{s-}) \dd N_{s} + \int_{0}^{t} (f, f')_{s} (Y_{s}) \dd \mathbf{X}_{s}. 
\end{equation}
Since we think of the $\d \bX$-term as the term of lowest regularity (even though $\mathbf{X}\in \mathscr{V}^{\frac{1}{2}}$ is allowed and may be more regular than the sample paths of $N$), we will treat this as the "leading order term". Accordingly we refer to \eqref{eq:classification} as "semilinear" if $y \mapsto f_{s}(y)$ is a linear map for any $s \in [0,T]$. 

\subsection{Conventions:}
Here we collect some fundamental assumptions that are assumed to hold throughout the whole manuscript. We consider a finite time horizon $T$>0 and a fixed filtered probability space $(\Omega, \mathfrak{F}, (\mathfrak{F}_{t})_{t\in [0,T]}, \mathbb{P})$, satisfying the 'usual hypotheses'. Properties related to measurability or nullsets are always meant with respect to this filtered probability space, if not specified otherwise. $V, W, U$ are used to denote real finite-dim. (but otherwise arbitrary) Banach spaces (unless stated otherwise as in \cref{appendix: A discrete approximation approach to consistency}). 
\addtocontents{toc}{\protect\setcounter{tocdepth}{3}}
\section{Fundamentals of rough stochastic analysis}
\label{section:Fundamentals of Rough Stochastic Analysis}
This section both revisits and extends the fundamental elements of the mixed-moment theory of rough stochastic analysis from \cite{FHL21} (following the exposition from \cite{AllanPieper2026}) and the $L^{0}$-theory of \cite{fzk23} (partially following the exposition from \cite{DauseFrizJentzenSong2026}) and for the first time considers both in a unifying framework. First we however revisit and extend some of the deterministic theory of càdlàg rough paths. 
\subsection{Some notions from càdlàg rough paths theory}
\label{subsection:Some notions from deterministic rough paths theory}
We quickly revisit some notions related to càdlàg controlled rough paths and vector fields mostly spread throughout the literature \cite{Friz_Zhang_2018, BCN24, DauseFrizJentzenSong2026} (whenever necessary adapting the notion to the càdlàg $p$-Var setting). 
In the following $U, V, W$ will denote arbitrary finite-dimensional Banach spaces.\\
Inspired by \cite[Chapter 3]{DauseFrizJentzenSong2026} we introduce the following convenient 'expansion notation'.
\begin{definition}
\label{def:expansion_notation}
    Let $A: \Delta_{[0,T]}\to V$. We say that for $r>0$ , $A_{s,t} \stackrel{r}{=}0$ if there exist $N \in \mathbb{N}$, controls $w_{i, j}$ and $\alpha_{i,j}>0 $ for $i=1,2$, $j=1,\dots, N$ such that for any $(s, t) \in \Delta_{[0,T]}$ it holds that
    \begin{equation}
    \label{def:eq:expansion_notation}
        |A_{s,t}| \leq \sum_{j=1}^{N} w_{1, j}(s, t-)^{\alpha_{1, j}}w_{2, j}(s+, t)^{\alpha_{2, j}}
    \end{equation}
    and $\alpha_{1, j}+ \alpha_{2, j}\geq \nicefrac{1}{r}$ for any $j =1, \dots, N$. For $A, B:\Delta_{[0,T]}\to V$ we define $A\stackrel{r}{=}B : \Leftrightarrow A-B\stackrel{r}{=}0$. 
    \end{definition}
In the following we are mostly interested in the situation where $r< 1$, since this corresponds precisely to the bounds in the sewing lemma \cite[Theorem 2.5]{Friz_Zhang_2018}. 
\begin{lemma}
\label{lemma:expansion_notation}
    Let $r< 1$ and $A: \Delta_{[0,T]}\to V$. Then 
    \begin{enumerate}
        \item \label{lemma:expansion_notation_item1} If $A\stackrel{r}{=}0$, then $A\in V^{r}$. 
        \item \label{lemma:expansion_notation_item3} If $A\in V^{r}$ and $A_{s,t}$ is both right-continuous in  $s$ and left-continuous in $t$ then $A\stackrel{r}{=}0$.  
        \item
        \label{lemma:expansion_notation_item4} If $A\stackrel{r}{=}0$, then 
        $
        \lim_{|\mathcal{P}|\downarrow 0} \sum_{[s, t]\in \mathcal{P}} |A_{s,t}|=0$.
        in RRS-sense. 
        \end{enumerate}
In particular if $A$ is given by a path increment $A_{s,t}\coloneqq \delta X_{s,t}$ for $X: [0,T]\to V$, then 
\begin{enumerate}[resume]
        \item \label{lemma:expansion_notation_item5} $\delta X \stackrel{r}{=}0$ iff $ X$ is constant. 
    \end{enumerate}
\end{lemma}
\begin{proof}
    For \eqref{lemma:expansion_notation_item1} note that for any partition $\mathcal{P}$ of $[0,T]$ it holds 
    \begin{equation*}
        \sum_{[s, t]\in \mathcal{P}} |A_{s,t}|^{r} \lesssim_{N, r} \sum_{j=1}^{N} \sum_{[s, t]\in \mathcal{P}} w_{1, j}(s, t-)^{r\alpha_{1, j}} w_{2, j}(s+, t)^{r \alpha_{2, j}}\leq \sum_{j=1}^{N} \tilde{w}_{j}(0, T) < \infty, 
    \end{equation*}
    with the controls $\tilde{w}_{j}(s, t)\coloneqq w_{1, j}(s, t)^{r\alpha_{1, j}}w_{2, j}(s, t)^{r\alpha_{2, j}}$. For \eqref{lemma:expansion_notation_item3} this follows by noting that if $A$ is right-continuous in $s$ and left-continuous in  $t$, then $|A_{s,t}|\leq \Vert A \Vert_{r; (s, t)}\eqqcolon w_{1}(s+, t-)^{\frac{1}{r}}$. \eqref{lemma:expansion_notation_item4} follows directly by applying 
    \cite[Lemma B.1]{AllanPieper2026}. For \eqref{lemma:expansion_notation_item5} obviously $X$ being constant implies $\delta X\stackrel{r}{=}0$. Conversely given $\delta X\stackrel{r}{=}0$, one uses the triangle inequality and \eqref{lemma:expansion_notation_item4}, to see that $X$ is constant. 
\end{proof}
\begin{definition}[Proposition 2.6 in \cite{Friz_Zhang_2018}]
    Let $p \in [2,3)$ and $X \in V^{p}([0,T]; V)$. Then we call a tuple $\mathcal{Y}=(Y, Y')$ of paths $\mathcal{Y}: [0,T]\to W \times \mathcal{L}(V; W)$ an $X$-controlled rough path if for $R^{Y, Y'}\coloneqq \delta Y_{s,t}- Y'_{s} \delta X_{s,t}$ it holds
    \begin{equation*}
        \Vert Y' \Vert_{p; [0,T]}< \infty; \quad \Vert R^{Y, Y'}\Vert_{\frac{p}{2}; [0,T]} < \infty
    \end{equation*}
    We denote the space of $W$-valued $X$-controlled rough paths by $\mathscr{V}^{p}_{X}([0,T]; W)$ and further equip it with the semi-norm\footnote{
    As usual this semi-norm as well as any other appearing in this work can easily be turned into norms by a standard modification.
    }
    \begin{equation*}
        \Vert Y, Y'\Vert_{X; p}\coloneqq \Vert Y'\Vert_{p; [0,T]}+ \Vert R^{Y, Y'}\Vert_{\frac{p}{2}; [0,T]}. 
    \end{equation*}
\end{definition}
\begin{definition}
\label{def:scRP}
    Let $p \in [2,3)$ and $\mathbf{X} \in \mathscr{V}^{p}([0,T]; V)$. Then we call a triple $\mathcal{Y}=(Y, Y',  Y'')$ of paths $\mathcal{Y}: [0,T]\to W \times \mathcal{L}(V; W)\times \mathcal{L}(V^{\otimes 2}; W)$ a strongly $X$-controlled rough path if $(Y', Y'')\in \mathscr{V}^{p}_{X}([0,T]; \mathcal{L}(V; W))$ and
    \begin{equation*}
    \delta Y_{s, t}\stackrel{\frac{p}{3}}{=} Y'_{s} \delta X_{s,t} + Y''_{s} \mathbb{X}_{s,t}. 
    \end{equation*} We denote the space of such $W$-valued strongly $\mathbf{X}$-controlled rough paths by $\mathscr{V}^{p}_{\mathbf{X}}([0,T]; W)$.
\end{definition}
\begin{remark}
\label{remark:scRPandRI}
    Let $p \in (2, 3)$. Then for any $(Y', Y'')\in \mathscr{V}_{X}^{p}([0,T]: \mathcal{L}(V; W))$, the rough integral $\mathcal{I}_{t}\coloneqq \int_{0}^{t} (Y', Y'')_{r}\dd \mathbf{X}_{r}$ can be constructed using sewing \cite[Proposition 2.6]{Friz_Zhang_2018}. Then $(\mathcal{I}, Y', Y'')\in \mathscr{V}_{\mathbf{X}}^{p}$. Conversely for any $(Y, Y', Y'')\in \mathscr{V}_{\mathbf{X}}^{p}$ it follows directly $Y\equiv \mathcal{I}$ using the fact that $\mathcal{I}$ is a path and \cref{lemma:expansion_notation}. 
\end{remark}
The notion of controlled paths also naturally extends to vector fields. 
\begin{definition}
    Let $p \in [2,3)$ and $X \in V^{p}([0,T]; V)$. We consider 
    \begin{equation*}
    (f, f'): [0,T]\times W \to U \times \mathcal{L}(V; U); \quad (t, x) \mapsto (f_{t}(x), f'_{t}(x))
    \end{equation*}
    Further assume that both 
    \begin{enumerate}
        \item $f_{t}\in \Lip^{2}(W; U)$, $f'_{t}\in \Lip^{1}(W; \mathcal{L}(V; U))$ uniformly in $t\in [0,T]$. 
        \item $(f(x), f'(x))\in 
        \mathscr{V}^{p}_{X}([0,T]; U)$ and $Df(x)\in V^{p}([0,T]; \mathcal{L}(W; U))$ uniformly in $x \in W$.
    \end{enumerate}
    We denote the space of such $(f, f')$ by $\mathscr{V}^{p}_{X}\Lip^{2}(W; U)$ and further equip this space with the semi-norm
    \begin{equation*}
    \begin{aligned}
        \Vert f , f'\Vert_{X; [0,T], W;p, 2}\coloneqq & \sup_{t\in [0,T]} \left (\Vert f_{t}\Vert_{\Lip^{2}(W; U)}+ \Vert f'_{t} \Vert_{\Lip^{1}(W; U)} \right)\\
        &+ \sup_{x \in W} \left ( \Vert f(x), f'(x) \Vert_{X; p} + \Vert Df(x) \Vert_{p; [0,T]}\right)
    \end{aligned}
    \end{equation*}
\end{definition}

\begin{definition}[Definition B.2 in \cite{DauseFrizJentzenSong2026}]
\label{Def:controlled Lip3 fields}
   Let $p \in [2,3)$ and $X \in V ^{p}([0,T]; V)$. We consider 
    \begin{equation*}
    (f, f'): [0,T]\times W \to U \times \mathcal{L}(V; U); \quad (t, x) \mapsto (f_{t}(x), f'_{t}(x))
    \end{equation*}
    Further assume that for $k \in \mathbb{N}_{\geq 2}$ both 
    \begin{enumerate}
        \item $f_{t}\in \Lip^{k}(W; U)$, $f'_{t}\in \Lip^{k-1}(W; \mathcal{L}(V; U))$ uniformly in $t\in [0,T]$. 
        \item $(D^{j}f, (D^{j}f')^{\top})\in 
        \mathscr{V}^{p}_{X}\Lip^{2}(W; \mathcal{L}(W^{\otimes j}, U))$ for any $j=0, \dots, k-2$. 
    \end{enumerate}
   We denote the space of such $(f, f')$ by $\mathscr{V}^{p}_{X}\Lip^{k}(W; U)$. We further equip this space with the semi-norm 
    \begin{equation*}
    \begin{aligned}
        \Vert f , f'\Vert_{X; [0,T], W;p, k}\coloneqq & \sup_{t\in [0,T]} \left (\Vert f_{t}\Vert_{\Lip^{k}(W; U)}+ \Vert f'_{t} \Vert_{\Lip^{k-1}(W; U)} \right)\\
        &+ \sup_{x \in W} \left ( \sum_{j=0}^{k-2}\Vert D^{j}f(x), (D^{j}f')^{\top}(x) \Vert_{X; p} + \Vert D^{k-1}f(x) \Vert_{p; [0,T]}\right)
    \end{aligned}
    \end{equation*} 
\end{definition}

We will mostly work with the spaces $\mathscr{V}^{p}\Lip^{3}$ as they provide a natural space of rough vector fields driving an RDE. Furthermore we are also interested in the following linear variants. 
\begin{definition}[Definition 3.2 in \cite{BCN24}]
\label{def:rough_linear_vf}
   Let $p\in [1, \infty)$ and $X \in V^{p}([0,T]; V)$. Then we call a tuple $\mathcal{A}=(A, A')$ of paths $\mathcal{A}: [0,T]\to \mathcal{L}(W; U)\times \mathcal{L}(W; \mathcal{L}(V; U))$ an $X$-controlled linear rough vector field if for $R^{A, A'}_{s,t}x\coloneqq \delta A_{s,t}x- A'_{s}x \delta X_{s,t}$ it holds
   \begin{equation*}
   \begin{aligned}
       \Vert A'\Vert_{p; [0,T]; \mathcal{L}(W; \mathcal{L}(V; U))}&\coloneqq \sup_{\mathcal{P}}\left( \sum_{[s,t]\in \mathcal{P}} \Vert \delta A'_{s,t}\Vert^{p}_{\mathcal{L}(W; \mathcal{L}(V; U))} \right)^{\frac{1}{p}}< \infty\\
       \Vert R^{A, A'}\Vert_{\frac{p}{2}; [0,T]; \mathcal{L}(W; U)}&\coloneqq \sup_{\mathcal{P}}\left( \sum_{[s,t]\in \mathcal{P}} \Vert  R^{A, A'}_{s,t}\Vert^{\frac{p}{2}}_{\mathcal{L}(W; U)} \right)^{\frac{2}{p}}< \infty. 
    \end{aligned}
   \end{equation*}
   The space of such $\mathcal{A}$ is denoted by $\mathscr{V}^{p}_{X}\cL(W; U)$.
\end{definition}
\subsection{Rough semimartingales \& stochastic controlled rough paths} 
We start off by introducing spaces of \textit{rough semimartingales}, extending their original definition from \cite{fzk23}.
\begin{definition}
    Let $p\in [1,\infty)$ and $m\in \{0 \} \cup [1,\infty)$, and $X\in V^p([0, T]; V)$ then we say a pair of càdlàg adapted processes $(Y, Y'): [0,T] \times \Omega \to W \times \mathcal{L}(V; W)$ is an almost surely $X$-controlled process if for $\mathbb{P}$-a.e. $\omega \in \Omega$ it holds $(Y(\omega), Y'(\omega))\in \mathscr{V}^{p}_{X}([0,T]; W)$. We say $(Y, Y')$ is an \emph{$L^{m}$-X controlled process}, if 
    \begin{equation*}
        \|Y,Y'\|_{L^{m}(\sV^p_X)} \coloneqq \left \Vert \Vert Y, Y'\Vert_{X; p; [0,T]} \right \Vert_{L^m}<\infty.  
    \end{equation*}
    In this case we write $(Y,Y') \in L^{m}\sV^p_X([0,T]; W)$, where in the case $m=0$ we mean that $\Vert Y, Y'\Vert_{X; p; [0,T]}$ is a.s. finite and for simplicity we call $(Y,Y')$ an \emph{$X$-controlled process}. 
\end{definition}

\begin{definition}\label{definition: Lm RSM}
    Let $p\in [1,\infty)$ and $X\in \sV^p([0,T], V)$, then we call a $3$-tuple $(Y,Y'; M)$ an $X$-controlled, $W$-valued, \emph{$p$-rough semimartingale} (short: RSM), if $M\in \cM(W)$ and 
    \begin{equation*}
        (Y, Y') :  [0, T] \times \Omega  \rightarrow W \times \cL(V; W)
   \end{equation*}
   is an almost surely càdlàg adapted process such that $(Y-M, Y') \in L^{0}\sV^p_X([0,T]; W)$. 
   Further, for any $m \in [1,\infty)$, we say $(Y, Y'; M)$ is an \emph{($X$-controlled) $L^m$-$p$-rough semimartingale} if $M \in \cM^m$ and $(Y-M,Y') \in L^m\sV^p_X([0,T]; W)$. In this case we write $(Y,Y';M) \in \sH^{p,m}_X([0,T]; W)$ and define the semi-norm
   \begin{equation}
        \big\|(Y,Y';M)\big\|_{\sH^{p,m}_X}:= \Big\|[M]^{\frac{1}{2}}_T \Big\|_{L^m}+ \|Y-M,Y'\|_{L^{m}(\sV^p_X)}.
    \end{equation}
    Further, we set $\mathscr{H}^{p}_{X}([0,T]; W)\coloneqq \mathscr{H}^{p, 0}_{X}([0,T]; W)$. 
\end{definition}
\begin{remark}\label{remark:comaprison_FSZZ26_1}
    The space $\mathscr{H}_{X}^{p}([0,T]; W)$ is the same as the space $\operatorname{RSM}^{(p,p)-\operatorname{Var}}_{X}([0,T]; W)$ from  \cite[Definition 2.6]{friz2026roughbackwardstochasticdifferential}. For continuous RSMs $(Y, Y'; M)$ the authors of \cite{friz2026roughbackwardstochasticdifferential} proceed to equip this space with a semi-norm given by 
    \begin{equation*}
    \begin{aligned}
       \bm{\Vert} Y, Y'; M \bm{\Vert}_{X; p, p; m, \infty}\coloneqq & \sup_{(s, t) \in \Delta_{[0,T]}} \left \Vert \Vert \delta Y'\Vert_{p; [s, t]} |\mathfrak{F}_{s} \right \Vert_{m , \infty}+ \sup_{t \in [0,T]} \Vert Y'_{t} \Vert_{\infty; [0,T]} \\
       &+ \sup_{(s, t) \in \Delta_{[0,T]}} \left \Vert \Vert \delta R^{Y, Y'} \Vert_{\frac{p}{2}; [s, t]} |\mathfrak{F}_{s} \right \Vert_{m , \infty}+ \Vert \delta M \Vert_{\operatorname{BMO}; [0,T]}.
    \end{aligned}
    \end{equation*}
     We emphasize that using $\Vert \cdot \Vert_{L^m}\leq \Vert \cdot \Vert_{L^{m,n}}$ for any $n \in [m, \infty]$ and the John-Nirenberg inequality (together with Doob- and BDG-inequality as usual) to handle the BMO-term, it holds
     \begin{equation*}
         \Vert (Y, Y'; M)\Vert_{\mathscr{H}^{p, m}_{X}}+ \sup_{t \in [0,T]}\Vert Y'_{t} \Vert_{L^{\infty}} 
         \lesssim \bm{\Vert} Y, Y'; M \bm{\Vert}_{X; p, p; m, \infty}. 
     \end{equation*}
     In the càdlàg setting however working with the classical $\operatorname{BMO}$-norm is less desirable, since $M \in \operatorname{BMO}$ implies $M \in \mathcal{M}^{2}$ with bounded jumps. This rules out many interesting examples e.g. $M$ being the piece-wise constant approximation of Brownian Motion. It is therefore unclear to us as of yet how to extend the results of \cite{friz2026roughbackwardstochasticdifferential} in an appropriate manner to the setting of this work.
\end{remark}
The main technical feature of rough semimartingales is that they have a 1:1 correspondence to random controlled rough paths.

\begin{definition} \label{def:randomCRPs}
    Let $p \in [2,3)$ and $X: [0,T]\times \Omega \to V$ càdlàg, adapted with $X\in V^{p}([0,T]; V)$ almost surely. Then define the space of càdlàg, adapted controlled rough paths $L^{0}_{ad}(\Omega, \mathscr{V}^{p}_{X}([0,T]; W))$ as consisting of càdlàg adapted $(Y, Y'): [0,T]\times \Omega \to W\times \mathcal{L}(V; W)$ such that $(Y, Y')(\omega)\in \mathscr{V}^{p}_{X(\omega)}([0,T]; W)$ for $\mathbb{P}$-almost every $\omega \in \Omega$. We additionally define the space of càdlàg, adapted controlled rough vector fields $L^{0}_{ad}(\Omega, \mathscr{V}^{p}\Lip^{2}_{X}([0,T], W; U))$ as consisting of càdlàg, adapted $(f, f'): [0,T] \times \Omega \times W\to U \times \mathcal{L}(V; U)$ such that $(f, f')(\omega)\in \mathscr{V}^{p}_{X(\omega)}\Lip^{2}(W; U)$ for $\mathbb{P}$-almost surely every $\omega \in \Omega$. 
\end{definition}
\begin{lemma}[Theorem 6.5 in \cite{fzk23}]
\label{lemma:RSM<->cRP}
    Let $p \in [2,3)$, $X \in V^{p}([0,T]; V)$ and  $M \in \mathcal{M}(U)$. Then it holds that:  
    \begin{enumerate}
    \item Let $(Y, Y')\in L^{0}_{ad}(\Omega, \mathscr{V}^{p}_{(X; M)}([0,T]; W \oplus U))$ with $Y'=(\partial_{X}Y, \partial_{M}Y)^{\top}\in \cL(V; W \oplus U)$. Then $(Y, \partial_{X}Y; N)\in \mathscr{H}_{X}^{p}([0, T]; W)$, where $N_{t}\coloneqq \int_{0}^{t} \partial_{M}Y_{r-} \dd M_{r}$\footnote{Note that in \cite[Theorem 6.5]{fzk23} $N$ (called $\tilde{g}$ there) is incorrectly defined, which is why we provide the corrected statement here.}.
    \item Let $(Y, \partial_{X}Y; M)\in \mathscr{H}^p_{X}([0,T]; W)$. Then $(Y, Y')\in L^{0}_{ad}(\Omega, \mathscr{V}^{p}_{(X, M)}([0,T]; W \oplus W))$ with $Y' \coloneqq (\partial_{X} Y, \Id_{W})^{\top}$.
    \end{enumerate} 
\end{lemma}
A similar, but a priori different, notion are \textit{stochastic controlled rough paths} first introduced in \cite{FHL21}. 
\begin{definition}[Definition 4.1 in \cite{AllanPieper2026}]
Let $p \in [2,3)$, $m \in [2,\infty)$ and $n \in [m,\infty]$, and let $X \in V^p([0,T] ; V)$. We call a pair of processes $(Y,Y'): [0,T]\times \Omega \to W \times \mathcal{L}(V; W) $ a \emph{stochastic controlled (rough) path} (short: SCRP)(relative to $X$), if
\begin{enumerate}[label=(\roman*)]
\item $Y$ and $Y'$ are both progressively measurable,
\item $Y \in V^p L^{m,n}([0,T]; W)$, 
\item $Y' \in V^p L^{m,n}([0,T]; \mathcal{L}(V; W))$ and $\sup_{s \in [0,T]} \|Y'_s\|_{L^{n}} < \infty$, and
\item $\|\E_{\cdot} R^{Y, Y'}\|_{\frac{p}{2}, n ,[0,T]} < \infty$, where the two-parameter process $R^{Y, Y'}: \Delta_{[0,T]}\to W$ is defined by
\begin{equation}\label{eq: defn R^Y_s,t}
R_{s,t}^{Y, Y'}\coloneqq \delta Y_{s,t}- Y'_{s} \delta X_{s,t}; \quad (\mathbb{E}_{\cdot}[R^{Y, Y'}])_{s,t}\coloneqq \mathbb{E}_{s}[R^{Y, Y'}_{s,t}], 
\end{equation}
for any $(s,t)\in \Delta_{[0,T]}$.
\end{enumerate}
We write $\sV^{p,m,n}_X([0,T]; W)$ for the space of stochastic controlled paths relative to $X$. For brevity, in the case when $m = n$, we will simply write $\sV^{p,m}_X([0,T]; W) \coloneqq \sV^{p,m,m}_X([0,T]; W)$.
\end{definition}

We further recall the following rough stochastic counterpart to the controlled vector fields defined in \cref{Def:controlled Lip3 fields}, first introduced in \cite{FHL21} and subsequently generalised to the càdlàg $p$-Var setting in \cite{AllanPieper2026}.

\begin{definition}\label{definition: stochastic controlled vector fields}
Let $p \in [2,3)$, $m \in [2,\infty)$, $n \in [m,\infty]$, $\gamma \in \{2,3\}$, let $E, W, \bar{W}$ be finite-dimensional Banach spaces, and let $X \in V^p$ be an $E$-valued path. We call $(f,f')$ a \emph{stochastic controlled vector field} if the following hold.
\begin{enumerate}
\item \begin{equation*}
(f,f') \colon [0,T] \times \Omega \to \Lip^{\gamma}(W;\bar{W}) \times \Lip^{\gamma - 1}\big(W;\cL(E; \bar{W})\big)
\end{equation*}
are strongly\footnote{Recall, this essentially just means that $f$ and $f'$ take values in a separable subspace; see \cite[Sections~2.1 \& 3.3]{FHL21} for a discussion of the necessity of the strong measurability assumption here.} progressively measurable, càdlàg almost surely, and
\[\|f,f'\|_{n; \gamma} := \sup_{s \in [0,T]} \big\|  \|f_s\|_{\Lip^{\gamma}} \big\|_{L^n} + \sup_{s \in [0,T]} \big\| \|f'_s\|_{\Lip^{\gamma-1}} \big\|_{L^n} < \infty. \]
\item Setting
\begin{equation}\label{eq: defintion bracket of vector field}
\llbracket g \rrbracket_{p,m,n,[s,t]} := \Big\| \sup_{x \in W} |\delta g(x)| \Big\|_{p,m,n,[s,t]},
\end{equation}
we have that $\llbracket f \rrbracket_{p,m,n,[0,T]}$, $\llbracket  f' \rrbracket_{p,m,n,[0,T]}$, $\llbracket  D f \rrbracket_{p,m,n,[0,T]}$  and if $\gamma =3$ additionally $\llbracket  \D^2 f \rrbracket_{p,m,n,[0,T]}$ and $\llbracket D f' \rrbracket_{p,m,n,[0,T]}$ are finite.
\item Setting for the map $\E_{\edot} R^{g,g'}$, defined by $(s,t) \mapsto \E_s R^{g,g'}_{s,t} := \E_s [\delta g_{s,t}] - g'_s \delta X_{s,t}$, 
\begin{equation*}
\llbracket \E_{\edot} R^{g,g'} \rrbracket_{\frac{p}{2},n,[0,T]} := \bigg( \sup_{\cP \subset [0,T]} \sum_{[s,t] \in \cP} \Big\| \sup_{x \in W} \big| \E_s \big[ R_{s,t}^{g,g'}(x) \big] \big| \Big\|^{\frac{p}{2}}_{L^n} \bigg)^{\hspace{-2pt}\frac{2}{p}},
\end{equation*}
\end{enumerate}
then $\llbracket \E_{\cdot} R^{f,f'}\rrbracket_{\frac{p}{2},n, [0,T]}$ and if $\gamma = 3$ additionally $\llbracket \E_{\cdot} R^{D f, D f'}\rrbracket_{\frac{p}{2}, n, [0,T]}$ are finite. We denote the space of stochastic controlled vector fields by $\sV^{p,m,n}_X \Lip^\gamma$. 
\end{definition}

For $\gamma \in \{2,3\}$, and a stochastic controlled vector field $(f,f') \in \cV^{p,m,n}_X \Lip^\gamma$, we write for any $(s,t) \in \Delta_{[0,T]}$
\begin{align*}
    \llbracket f, f'\rrbracket_{p,m,n,[0,T]} &:= \sum_{j=0}^{\gamma-2}\Big( \llbracket \D^j f \rrbracket_{p,m,n,[s,t]}+\llbracket \D^j f'\rrbracket_{p,m,n,[s,t]}+ \big\llbracket\E_{\cdot} R^{\D^j f, \D^j f'} \big\rrbracket_{\frac{p}{2}, n, [s,t]}\Big)\\
    &\quad +\llbracket \D^{\gamma-1} f\rrbracket_{p,m,n, [s,t]}.
\end{align*}
For the (semi-)norm in \eqref{eq: defintion bracket of vector field}, we also adopt the shorthand $\| \hspace{1pt} \cdot \hspace{1pt} \|_{p,m,[s,t]} := \| \hspace{1pt} \cdot \hspace{1pt} \|_{p,m,m,[s,t]}$.
We further define for $(f,f')\in \cV^{p,m,\infty}_X \Lip^3$ and $(\tf,\tf')\in \cV^{p,m,n}_{\tX} \Lip^2$ and any $(s,t)\in \Delta_{[0,T]}$
\begin{equation}\label{eq: distance vfs}
\begin{split}
     \llbracket f,f'; \tf, \tf' \rrbracket_{p,m,n, [s,t]}&:=  \llbracket f-\tf \rrbracket_{p,m,n,[s,t]} + \llbracket f'-\tf' \rrbracket_{p,m,n,[s,t]} + \llbracket \D f- \D \tf \rrbracket_{p,m,n,[s,t]} \\
    &\quad+ \big\llbracket \E_{\edot} R^f -R^{\tf} \big \rrbracket_{\frac{p}{2},n,[s,t]}
\end{split}
\end{equation}

Similarly, we define the linear analogue to stochastic controlled vector fields as introduced in \cite{BCN24}.

\begin{definition}\label{definition: stochastic controlled linear vector field}
Let $X\in V^p([0,T];V)$ be deterministic for some $p\in[2,3)$,
$m\in[2,\infty)$ and $n\in[m,\infty]$.
We call a pair $(A,A')$ of adapted, almost surely càdlàg
processes with values in $\cL(W; \bar W)$ and
$\cL(W; \cL(V; \bar W))$, respectively, a stochastic
($X$-)controlled linear vector field from $W$ to $\bar W$
and write $(A,A')\in\sV_X^{p,m,n}\cL(W ; \bar W)$ if
\begin{equation*}
     \|A,A'\|_{\infty; \cL} :=\|A\|_\infty+\|A'\|_\infty<\infty,
    \quad \text{and} \quad (A,A') \in \sV^{p,m,n}([0,T];\cL(W ;\bar W)),
\end{equation*}
where we define $\|G\|_{\infty}:=\esssup_{\omega\in\Omega}
\sup_{t\in[0,T]}|G_t(\omega)|_{\cL}$.
\end{definition}

For any $(A,A')\in \sV^{p,m,n}_X\cL$ and $(\tilde{A}, \tilde{A}')\in \sV^{p,m,n}_{\tX}\cL$ and any $(s,t)\in \Delta_{[0,T]}$, we define $\llbracket A,A', \tilde{A}, \tilde{A}' \rrbracket_{p,m,n,[s,t]}$ analogously to \eqref{eq: distance vfs}.

It will be useful for the remainder of this work to have a notion of processes, which combine both the rough semimartingale and stochastic controlled structure. 
\begin{definition}
    Let $p\in [2,3), m \in [2, \infty), n \in [m, \infty]$. We call a triple of processes $(Y, Y'; M)$ a \emph{good $L^{m, n}$-stochastic controlled rough path} if $(Y, Y') \in \mathscr{V}^{p, m, n}_{X}([0,T]; W)$ and $(Y, Y'; M)\in \mathscr{H}^{p}_{X}([0,T]; W)$.  
\end{definition}

Most notably, good stochastic controlled rough paths admit natural composition rules for $n=\infty$. 
\begin{lemma}\label{lemma: composition for RSM}
    Let $\mathcal{Y}=(Y, Y'; M) \in \mathscr{H}_{X}^{p}([0,T]; W)$ and $(f, f')\in \mathscr{V}^{p}_{X}\Lip^{2}(W; U)$. Then for 
    \begin{equation*}
        (f, f') \circ \mathcal{Y}\coloneqq (f(Y), Df(Y)Y'+ f'(Y); Df(Y) \bullet \dd M)
    \end{equation*}
    it holds
    \begin{enumerate}[label=(\roman*)]
        \item $(f, f')\circ \mathcal{Y}\in \mathscr{H}^{p}_{X}([0,T]; U)$.
        \item Let $(Y, Y')\in \mathscr{V}^{p, m,n}_{X}([0,T]; W)$. Then $(f(Y), Df(Y)Y'+ f'(Y))\in \mathscr{V}^{p, m, \frac{n}{2}}_{X}([0,T]; W)$.  
    \end{enumerate}
   In particular the composition maps good $L^{m,n}$-SCRPs to good $L^{m, \frac{n}{2}}$-SCRPS if $n=\infty$ and so (good) $L^{m, \infty}$-SCRPs are stable under composition with $\mathscr{V}^{p}\Lip^{2}$-vector fields. 
\end{lemma}
\begin{proof}
    For $(i)$ first note that by \cref{lemma:RSM<->cRP} for $\tilde{Y}'\coloneqq (Y', \Id_{W})$ it holds \\
    $(Y, \tilde{Y}')\in L^{0}_{ad}(\Omega, \mathscr{V}^{p}_{(X, M)}([0,T]: W))$. Then we know that by applying classical composition rules for controlled rough paths for $\tilde{f}'\coloneqq (f', 0)^{\top}\in \mathcal{L}(V; W)\oplus \mathcal{L}(W; W)$, that 
    \begin{equation*}
        (f, \tilde{f}') \circ (Y, \tilde{Y}')\coloneqq (f(Y), Df(Y) \tilde{Y}'+ \tilde{f}'(Y))\in L^{0}(\Omega, \mathscr{V}^{p}_{(X, M)}([0,T]; U)).
    \end{equation*}
    By applying again \cref{lemma:RSM<->cRP} we thus infer $(f, f') \circ \mathcal{Y}\in \mathscr{H}_{X}^{p}([0,T]; U)$. For $(ii)$ see \cite[Lemma 4.11]{AllanPieper2026}. 
\end{proof}
\begin{remark}
  The above loss of integrability $n< \infty$ is completely analogous to \cite[Lemma 3.11]{FHL21} and leads to the necessary restriction $n=\infty$. We emphasize that we don't equip the vector space of good $L^{m,n}$ solutions with a norm and therefore   \eqref{lemma: composition for RSM} is a purely qualitative statement. In the continuous setting a quantitative version for $n=\infty$ of this statement can be shown using the stronger RSM-norm from \cref{remark:comaprison_FSZZ26_1}; see \cite[Proposition 3.10]{friz2026roughbackwardstochasticdifferential}.
\end{remark}
\begin{definition}
\label{def:rough_stochastic_vf}
    Let $p \in (2,3)$ and $X \in V^{p}([0, T]; V)$. Consider a triplet of càdlàg, adapted $(A, A'; G): [0,T]\times \Omega \to \mathcal{L}(W; U) \times \mathcal{L}(W; \mathcal{L}(V; U)) \times \mathcal{L}(W; U)$, where $G \in \mathcal{M}$ such that 
\begin{enumerate}
\item It holds $(A, A')\in V^{p}([0,T]; \mathcal{L}(W; U)\times \mathcal{L}(W; \mathcal{L}(V; U))$ almost surely. 
    \item \label{def:item:lin_stoch_vf} It holds
    \begin{equation*}
        \left \Vert \sup_{t \in [0,T]} | A_{t} |_{\mathcal{L}(W; U)} \right\Vert_{\infty}+\left \Vert \sup_{t \in [0,T]} | A'_{t} |_{\mathcal{L}(W; \mathcal{L}(V; U))} \right\Vert_{\infty} < \infty.  
    \end{equation*}
        \item  For $R^{A, A'}_{s,t}\coloneqq \delta A_{s,t} - \delta G_{s,t} - A'_{s} \delta X_{s,t} \in \mathcal{L}(W; U)$
        \footnote{Here $(A'_{s}\delta X_{s,t})w\coloneqq (A'_{s}w) \delta X_{s,t}$ for any $w \in W$.} it holds $R^{A, A'}\in V^{\frac{p}{2}}([0,T]; \mathcal{L}(W; U))$ almost surely. 
\end{enumerate}
The space of such linear rough stochastic vector fields $(A, A'; G)$ is denoted by $\mathscr{H}_{X}^{p}\cL(W; U)$. 
\end{definition}
\subsection{Joint martingale lifts}
The following ideas were introduced in \cite{DOR13, fzk23}. 
First we define the integral of a martingale against a rough path. 
\begin{definition}[Proof of Theorem 1.3 in \cite{fzk23}]\label{definition: integral mart against rp}
    Let $p \in [2,3), X\in V^{p}([0,T]; V)$ and $M \in \mathcal{M}([0,T
    ]; W)$. Then we define the càdlàg adapted process 
    \begin{equation*}
        \int_{0}^{t} M_{s-} \otimes \dd X_{s}\coloneqq \delta (M \otimes X)_{0,t}- \left(\int_{0}^{t} X_{s-} \otimes \dd M_{s}\right)^{\top}-[M , X]_{t},  
    \end{equation*}
    where the integral on the right-hand side is the It\^{o} integral and the bracket term is defined as
    \begin{equation*}
        [M, X]_{t} =\sum_{0\leq s\leq t} \Delta M_{s} \otimes \Delta X_{s},
    \end{equation*} which is well-posed by \cite[Theorem 5.4]{fzk23}. Further we define $[X, M]\coloneqq [M, X]^{\top}$.
\end{definition}

\begin{definition}[Theorem 1.1. in \cite{fzk23}]
    Let $p \in (2,3), X \in V^{p}([0,T]: V)$ and $M \in \mathcal{M}([0,T]; W)$. Then we define the iterated integrals for $(s, t) \in \Delta_{[0,T]}$
    \begin{equation*}
    \begin{aligned}
        \Pi(X; M)_{s,t}&\coloneqq  \int_{s}^{t} \delta X_{s,r-} \otimes \dd M_{r}\\ 
        \Pi(M; X)_{s, t}&\coloneqq \int_{s}^{t} \delta M_{s,r-} \otimes \dd X_{r}\\
        \Pi(M; M)_{s,t} &\coloneqq \int_{s}^{t} \delta M_{s,r-} \otimes \dd M_{r}
    \end{aligned}
    \end{equation*}
\end{definition}
\begin{lemma}
\label{lemma: joint lift and bracket}
    Let $p \in (2,3), \mathbf{X}=(X, \mathbb{X})\in \mathscr{V}^{p}([0,T]; V)$ and $M \in \mathcal{M}([0,T]; W)$. We define the \emph{joint lift} of $\mathbf{X}$ and $M$ by 
    \begin{equation*}
        \mathscr{J}(\mathbf{X}; M)\coloneqq \left( \begin{pmatrix}
            X \\
            M 
        \end{pmatrix}, \begin{pmatrix}
            \mathbb{X} & \Pi(X; M)\\
            \Pi(M; X) & \Pi(M; M)
        \end{pmatrix} \right).
    \end{equation*}
    Then $\mathscr{J}(\mathbf{X}; M)\in \mathscr{V}^{p}([0,T]; V \oplus W)$ almost surely and
    \begin{equation*}
        [\mathscr{J}(\mathbf{X}; M)]= \begin{pmatrix}
            [\mathbf{X}]& [X, M]\\
            [M, X] & [M]
        \end{pmatrix},
    \end{equation*}
    where $[\mathscr{J}(\mathbf{X}; M)],[\mathbf{X}]$ denote the rough path bracket, $[M]$ the classical covariation bracket of $M$ and 
    \begin{equation*}
        [M, X]_{t}\coloneqq  
        \sum_{0 \leq s \leq t} \Delta M_{s} \otimes \Delta X_{s} \eqqcolon [X, M]^{\top}.
    \end{equation*}
\end{lemma}
\begin{proof}
First $\mathscr{J}(\mathbf{X}; M)\in \mathscr{V}^{p}([0,T]; V \oplus W)$ a.s. follows by \cite[Theorem 1.1]{fzk23}. The characterisation of the bracket follows by analogous calculation to \cite[Lemma 4.8]{DauseFrizJentzenSong2026}. 
\end{proof}
\subsection{On rough stochastic integration}
We introduce the notion of rough stochastic integration for both RSM- and SCRP-integrands. 
\begin{definition} \label{def:RSI}
    Let $\mathbf{X}\in \mathscr{V}^{p}([0,T]; V)$ and $(Y, Y'): [0,T]\times \Omega \to \mathcal{L}(V; W) \times \mathcal{L}(V^{\otimes 2}; W)$ $(\mathfrak{F}_{t})$-progressive. We say the \textit{rough stochastic integral} (short:RSI) of $(Y, Y')$ with respect to $\mathbf{X}$ exists if there is a measurable process $\mathcal{I}(Y, Y'; \mathbf{X}): [0,T]\times \Omega \to W$ such that for any  sequence of deterministic partitions $\mathcal{P}^{n}$ of $[0,T]$ of locally vanishing mesh-size it holds that\begin{equation*}
        \operatorname{u.c.p.-}\lim_{n \to \infty}\left(\sum_{[u,v]\in \mathcal{P}^{n}} Y_{u} \delta X_{u \wedge \cdot ,v\wedge \cdot } + Y'_{u}\mathbb{X}_{u \wedge \cdot ,v \wedge \cdot}\right)_{\cdot} =\mathcal{I}(Y, Y'; \mathbf{X}).
    \end{equation*}
    If so this limit is denoted by
    \begin{equation*}
    \int_{0}^{t} (Y, Y')_{r} \dd \bX_{r} \coloneqq \mathcal{I}(Y, Y'; \mathbf{X})_{t}.
    \end{equation*}
\end{definition}
\begin{remark}
    By imposing the convergence in u.c.p. sense it holds that, by equipping $\sV^{p}$ with a $\sigma$-field $\mathfrak{V}$ we can under natural assumptions on $(Y, Y')$ show $\mathcal{I}$ to be jointly $\mathfrak{P}\otimes \mathfrak{V}$ measurable in $(t, \omega, \mathbf{X})$ when equipping $[0,T]\times \Omega$ with it's progressive $\sigma$-field $\mathfrak{P}$; see \cite{FLZ25} for details. In particular this allows for the randomization of $\mathcal{I}$ in $\mathbf{X}$ i.e. setting $\overline{\mathcal{I}}(\omega)\coloneqq \mathcal{I}((Y, Y')(\omega); \mathbf{X})|_{\mathbf{X}=\mathbf{Z}(\omega)}$, where $Z$ is a càdlàg semimartingale of $(Y, Y')$ and $\mathbf{Z}$ is a suitable rough path lift of $Z$. Notably when choosing $\mathbf{Z}$ to be the It\^{o}-lift and $Y$ suitably regular, it holds that  $\overline{\mathcal{I}}= \int Y_{s-} \dd Z_{s}$ due to the left-point approximation chosen above; see \cite[Proposition 3.6]{AllanPieperTeichmann2025}.
\end{remark}

\begin{lemma}
\label{lemma:RSI_cadlag_jumpstructure}
Let $(Y, Y'), \mathbf{X}$ as in \cref{def:RSI} and $(Y,Y')$ càdlàg in probability such that the RSI $\mathcal{I}=\mathcal{I}((Y, Y'); \mathbf{X})$ exists. Then $\mathcal{I}: [0,T]\times \Omega \to W$ is a càdlàg adapted process such that almost surely for any $t \in [0,T]$ 
\begin{equation}
\label{eq:RSI_jump_structure}
    \Delta \mathcal{I}_{t}= Y_{t-} \Delta X_{t}+ Y'_{t-} \Delta \mathbb{X}_{t}.
\end{equation}
\end{lemma}

\begin{proof}
    The adaptedness of $\mathcal{I}$ follows directly from the adaptedness of the discrete sums. Note that u.c.p. limits of càdlàg processes are always càdlàg themselves and thus so is $\mathcal{I}$. The jump structure \eqref{eq:RSI_jump_structure} can be identified by noting that for any $\epsilon>0$
    \begin{equation*}
    \begin{aligned}
        \mathbb{P}\left(|\Delta \mathcal{I}_{t}- Y_{t-} \Delta X_{t} - Y'_{t-} \Delta \mathbb{X}_{t}|> \epsilon \right)\leq & 2 \mathbb{P}\left(\sup_{s \in [0,T]}| \mathcal{I}_{s}-\mathcal{I}^{n}_{s}|> \frac{\epsilon}{3} \right)\\
        &+ \mathbb{P}\left(|\Delta \mathcal{I}^{n}_{t}- Y_{t-} \Delta X_{t} - Y'_{t-} \Delta \mathbb{X}_{t}|> \frac{\epsilon}{3} \right) \stackrel{n \to \infty}{\longrightarrow}0,
    \end{aligned}
    \end{equation*}
    where $\mathcal{I}^{n}$ denotes the left-point sum from \cref{def:RSI} along a locally vanishing partition sequence $(\mathcal{P}^{n})_{n \in \mathbb{N}}$.  By definition we especially also know that $\sup_{s\in [0,T]} |\Delta \cI_s- \Delta \cI^n_s| \to 0$ as $n\to \infty$ in probability, and thus almost surely,
    \[ \{t \in (0,T] \, : \, |\Delta \cI_t| > 0\} \subset \{t \in (0,T] \, : \, |\Delta X_t| \vee |\Delta \X_t| > 0\} =: J_{\bX}, \]
    where, crucially, the latter set of jump times $J_{\bX}$ of $\bX = (X,\X)$ is deterministic and countable. Since \eqref{eq:RSI_jump_structure} holds almost surely for each $t \in (0,T]$, it also holds simultaneously for every $t \in J_{\bX}$ almost surely, and the result follows.
\end{proof}

In \cite{fzk23,FHL21, AllanPieper2026} the authors give sufficient conditions for the existence of such limits based on either the concepts of \textit{rough semimartinagles} or \textit{stochastic controlled rough paths}
\begin{lemma}
\label{lemma:sufficient_conditions_RSI}
    Let $\mathbf{X}\in \mathscr{V}^{p}([0,T]; V)$ with $p \in [2,3)$ and càdlàg adapted $(Y, Y'): [0,T]\times \Omega \to \mathcal{L}(V; W)\times \mathcal{L}(V^{\otimes 2}; W)$. Then $\cI(Y, Y'; \bX)$ exists if any of the following conditions is satisfied:
    \begin{enumerate}
        \item \label{lemma:sufficient_conditions_RSI:RSM}There is a martingale $M \in \mathcal{M}$ such that $(Y, Y'; M) \in \mathscr{H}^{p}_{X}([0,T]; \mathcal{L}(V; W))$; 
        \item \label{lemma:lemma:sufficient_conditions_RSI:SCRP}  $m \in [2, \infty), n \in [m , \infty]$ and $(Y, Y') \in \sV^{p,m,n}_{X}([0,T]; \cL(V; W))$.
    \end{enumerate}
\end{lemma}
\begin{proof}
    The claim  \eqref{lemma:sufficient_conditions_RSI:RSM} is the content of \cite[Theorem 1.3]{fzk23}. The claim \eqref{lemma:lemma:sufficient_conditions_RSI:SCRP} follows by \cite[Theorem 3.4]{AllanPieper2026} and analogous arguments to \cite[Lemma 4.3]{AllanPieper2026}\footnote{Notice that \cite[Lemma 4.3]{AllanPieper2026} does not immediately provide the existence of the RSI as it does not provide the necessary convergence in u.c.p.}.
\end{proof}
\begin{remark}
Note that the sufficient conditions given in \cref{lemma:sufficient_conditions_RSI} are by no means complete. In \cite[Theorem 3.5]{FHL21} the authors construct the RSI for generalizations of SCRPs. In \cite[Theorem 4.22]{DauseFrizJentzenSong2026} the authors consider compositions between Kunita-type parameter-dependent martingales $(M_{t}(x))_{t \in [0,T], x \in \mathbb{R}^{d}}$ and a.s. $\mathbf{X}$-controlled rough paths $(Y, Y')$, which also fall into neither of the above categories.
\end{remark}

Notably working with SCRPs comes with the direct benefit of having moment-estimates on the RSI in terms of the mixed moment norms of the germ.

\begin{lemma}[Lemma 4.3 in \cite{AllanPieper2026}]\label{lemma: existence of rough stochastic integrals}
Let $p \in [2,3)$, $m \in [2,\infty)$ and $n \in [m,\infty]$. Let $\bX = (X,\X) \in \sV^p$ be a c\`adl\`ag rough path, and let $(Y,Y') \in \sV^{p,m,n}_X$ be a stochastic controlled path. Then there exists an $L^m$-integrable adapted process $(\int_0^t Y_u \dd \bX_u)_{t \in [0,T]}$, such that, for every $(s,t) \in \Delta_{[0,T]}$,
\begin{equation*}
\lim_{|\cP| \to 0} \bigg\|\int_s^t Y_u \dd \bX_u - \sum_{[u,v] \in \cP} \big(Y_u \delta X_{u,v} + Y'_u \X_{u,v}\big)\bigg\|_{m,n,s} = 0,
\end{equation*}
where the limit holds along partitions $\cP$ of the interval $[s,t]$ as the mesh size tends to zero. Moreover, we have that
\begin{equation}\label{eq: Lqr- estimate for rough stochastic integral}
\bigg\|\E_s \bigg[\int_s^t Y_u \dd \bX_u - Y_s \delta X_{s,t} - Y'_s \X_{s,t}\bigg]\bigg\|_{L^{n}} \leq C\Big(\|\E_{\edot} R^{Y,Y'}\|_{\frac{p}{2},n,[s,t)} \|X\|_{p,[s,t]} + \|Y'\|_{p,m,n,[s,t)} \|\X\|_{\frac{p}{2},[s,t]}\Big)
\end{equation}
and
\begin{equation}\label{eq: p,q,r,s- estimate for rough stochastic integral}
\begin{split}
&\bigg\|\int_s^t Y_u \dd \bX_u - Y_s \delta X_{s,t} - Y'_s \X_{s,t}\bigg\|_{m,n,s}\\
&\leq C \Big(\|Y\|_{p,m,n,[s,t)} + \|Y'\|_{p,m,n,[s,t)} + \|\E_{\edot} R^{Y,Y'}\|_{\frac{p}{2},n,[s,t)} + \|X\|_{p,[s,t)} \sup_{u \in [0,T)} \|Y'_u\|_{L^{n}}\Big) \|\bX\|_{p,[s,t]}
\end{split}
\end{equation}
for every $(s,t) \in \Delta_{[0,T]}$, where the constant $C$ depends only on $p$ and $m$.
\end{lemma}

Using the pathwise estiamtes provided in \cite{Friz_Zhang_2018} together with the corresponding martingale transform estimates from \cite{fzk23} we however find similarly for $L^m$-$X$-rough semimartingales that

\begin{lemma}
   Let $p \in [2, 3)$, $m\in \{0\}\cup [1, \infty)$ and $\bX\in \sV^{p}([0,T]; V)$. Then for any $(Y, Y'; M) \in \sH^{p,m}_X([0,T]; \cL(V; W))$ we have that 
   \begin{equation*}
       \bigg(\int_0^{\cdot} (Y,Y')_{r} \dd \bX_r, Y, 0\bigg)\in \sH^{p,m}_{X}. 
   \end{equation*}
   Further if $m\neq 0$ it holds  
   \begin{equation*}
       \bigg\|\int_0^{\cdot} (Y,Y')_{r} \dd \bX_r, Y ; 0\bigg\|_{\sH^{p,m}_X} \leq C \big(\|Y_0\|_{L^m}+\|Y'_0\|_{L^m}+ \|M_0\|_{L^m} + \|Y,Y'; M\|_{\sH^{p,m}_X} \big)\|\bX\|_{p,[0,T]},
   \end{equation*}
   where $C>0$ only depends on $p, m$.
\end{lemma}

\begin{proof}
    Notice, that $(s,t) \in \Delta_{[0,T]}$ 
    \begin{equation*}
    \begin{aligned}
        R^{\int_0^{\cdot} (Y,Y') \dd \bX_r}_{s,t} &\coloneqq \int_s^t (Y,Y')_r \dd \bX_r - Y_s \delta X_{s,t}\\ &= \int_s^t (Y-M,Y')_r \dd \bX_r - (Y_s -M_s) \delta X_{s,t} +\int_s^t \delta M_{s,r-} \dd X_r\\
        &= R^{\int_0^{\cdot }((Y-M),Y')_r \dd \bX_r,(Y-M)}_{s,t} + \Pi(M,X)_{s,t}
    \end{aligned}
    \end{equation*}
    and again by \cite[Theorem 1.3]{fzk23} together with \cite[Proposition 2.6]{Friz_Zhang_2018} we find that 
    \begin{align*}
        \|R^{\int_0^{\cdot} (Y,Y')_r \dd \bX_r,Y}\|_{L^m\big(V^{\frac{p}{2}} \big)} &\leq \|R^{\int_0^{\cdot }((Y-M),Y')_r \dd \bX_r,(Y-M)}\|_{L^m \big(V^{\frac{p}{2}}\big)} + \|\Pi(M,X)\|_{L^m\big(V^{\frac{p}{2}} \big)} \\
        &\leq C\Big(\|Y'_0\|_{L^m}+ \|R^{(Y-M),Y'}\|_{L^m\big(V^{\frac{p}{2}}\big)} \|X\|_{p,[s,t]} + \|Y'\|_{L^m(V^p)} \|\X\|_{\frac{p}{2},[0,T]}\Big)\\
        &\quad + \|X\|_{p,[0,T]}\||M|^{*}_T\|_{L^m}, 
    \end{align*}
    which immediately implies the claim.
\end{proof}

Similarly, by \cite[Lemma 3.4]{Friz_Zhang_2018} we find the following stability result for RSIs in $\sH^{p,m}_X$:

\begin{lemma}
\label{lemma:stability_rough_integration_Hpq}
Let $p\in [2,3)$, $m\in [1,\infty)$ and let $\bX,\widetilde{\bX}$  be two $p$-rough paths.
Suppose
\[ (Y,Y';M)\in \sH^{p,m}_X([0,T];\cL(V;W)),\quad
(\widetilde Y,\widetilde Y';\widetilde M)\in
\sH^{p,m}_{\widetilde X}([0,T];\cL(V;W)).\] 
For any two paths $(A, \widetilde{A}) \in \{(Y-M,\tY-\tM), (Y',\tY'), (M,\tM), (X, \tX)\}$ we write $\Delta A= A-\widetilde A$  
and similarly set $\Delta R= R^{(Y-M),Y'}-R^{(\tY-\tM), \tY'}$ as well as $\|\Delta\bX\|_{p,[0,T]} := \|\bX -\widetilde{\bX}\|_{p, [0,T]}$ .
Then 
\begin{align*}
    \Big\| R^{\int(Y,Y')\,d\mathbf X,Y} &- R^{\int(\tY,\tY') \dd \widetilde{\bX},
    \widetilde Y}
    \Big\|_{L^m(V^{p/2})} \\
    &\lesssim_{p}  (\| \bX \|_{p,[0,T]}+1)\Big(
    \big( \|\Delta Y_0\|_{L^m} +\|\Delta M_0\| +\|\Delta R\|_{L^m(V^{\frac{p}{2}})}\\
    &\quad +\|\Delta (Y-M)\|_{L^m(V^p)} +\||\Delta M|^*_T\|_{L^m} \big)\|\bX\|_{p,[0,T]}+\big(1+ \|\widetilde Y\|_{\sH^{p,m}_X} \big) \|\Delta \bX\|_{p,[0,T]} \Big).
\end{align*}
\end{lemma}

At last we provide a Fubini's theorem for the RSI with suitable integrands.

\begin{proposition}\label{prop:rough_stochastic_fubini}
    Let $\mathbf{X}\in \mathscr{V}^{p}([0,T]; V)$ with $p \in [2, 3)$ and càdlàg adapted $(Y, Y'): [0,T]\times \Omega \to \mathcal{L}(V; W)\times \mathcal{L}(V^{\otimes 2}; W)$. Suppose that any of the following holds true
    \begin{enumerate}
        \item \label{eq: rough_stochastic_Fubini_1} There is a martingale $M\in \cM^1$ such that $(Y, Y'; M) \in \mathscr{H}^{p, 1}_{X}([0,T]; \mathcal{L}(V; W))$;
        \item \label{eq: rough_stochastic_Fubini_2} For $m \in [2, \infty)$ and $n \in [m , \infty]$ $(Y, Y') \in \mathscr{V}^{p, m,n}_{X}([0,T]; \mathcal{L}(V; W))$.
    \end{enumerate}
    Then $(\mathbb{E}[Y], \mathbb{E}[Y'])\in \mathscr{V}^{p}_{X}([0,T]; \mathcal{L}(V; W))$ and for any $t \in [0,T]$ it holds 
    \begin{equation}
    \label{eq:rough_stochastic_Fubini}
        \mathbb{E}\left [ \int_{0}^{t} (Y, Y')_{r} \dd \mathbf{X}_{r} \right]= \int_{0}^{t} (\mathbb{E}[Y_{r}], \mathbb{E}[Y'_{r}]) \dd \mathbf{X}_{r},
    \end{equation}
    where the latter integral is the classical rough integral. 
\end{proposition}

\begin{proof}
    In the case of \eqref{eq: rough_stochastic_Fubini_1}, $(\mathbb{E}[Y], \mathbb{E}[Y'])\in \mathscr{V}_{X}^{p}([0,T]; \mathcal{L}(V; W))$ follows directly by the construction of $\mathscr{H}^{p, 1}_{X}([0,T]; \mathcal{L}(V; W))$, indeed for any $(s,t) \in \Delta_{[0,T]}$ it holds by Minkowski's inequality that for any partition $\cP$ of $[0,T]$ we have that
    \begin{equation*}
    \begin{aligned}
        \bigg(\sum_{[u,v]\in \cP} \big|\delta (\E [Y_{\cdot}])_{u,v}- \E [Y'_{u}]\delta X_{u,v} \big|^{\frac{p}{2}}\bigg)^{\frac{2}{p}} &\leq \bigg(\sum_{[u,v]\in \cP} (\E[|R^{Y-M,Y'}_{u,v}|]^{\frac{p}{2}}\bigg)^{\frac{2}{p}}\\
        &\leq \E\bigg[\bigg(\sum_{[u,v]\in \cP} |R^{Y-M,Y'}_{u,v}|^{\frac{p}{2}}\bigg)^{\frac{2}{p}}\bigg]
    \end{aligned}
    \end{equation*}
    and thus 
    \begin{align*}
        \big\|R^{\E[Y],\E[Y']} \big\|_{\frac{p}{2},[0,T]} \leq \big\|R^{Y-M,Y'} \big\|_{\frac{p}{2},1,[0,T]} \leq \|Y-M,Y'\|_{\sH^{p,1}_X} <\infty.
    \end{align*}
    Further, by \cref{theorem: a p-variation kolmogorov result} we also know $Y-M \in V^pL^1$ as well as $Y'\in V^p L^1$.
    Thus, we know that indeed the right-hand side of \eqref{eq:rough_stochastic_Fubini} is well-defined as the unique process $\cI$ such that there exists an $r<1$ such that
    \begin{equation}\label{eq: identification integral fubini}
        \big(\delta \cI_{s,t}- \E[Y_s] \delta X_{s,t}- \E[Y'_s] \X_{s,t}\big)_{(s,t) \in \Delta_{[0,T]}} \stackrel{r}{=} 0.
    \end{equation}
    However, as the previous argument also shows that $(Y-M, Y')\in \sV^p_X([0,T], L^1(\Omega))$ we find by the sewing lemma \cite[Proposition 2.2]{Friz_Zhang_2018} that 
    \begin{align*}
        &\bigg|\mathbb{E}\bigg [ \int_{s}^{t} (Y, Y')_{r} \dd \mathbf{X}_{r} \bigg]- \E[Y_s]\delta X_{s,t}-\E[Y'_s] \X_{s,t}\bigg|\\ &\leq \bigg \| \int_{s}^{t} (Y-M, Y')_{r} \dd \mathbf{X}_{r} -(Y_s-M_s)\delta X_{s,t}- Y'_s \X_{s,t}\bigg \|_{L^1} + \bigg\| \int_s^t \delta M_{s,r-} \otimes \dd X_r\bigg\|_{L^1}\\
        &\stackrel{\frac{p}{3}}{=} \bigg| \E\bigg[ \int_s^t \delta M_{s,r-}\otimes \dd X_r\bigg ] \bigg|.
    \end{align*}
    However, notice by \cref{definition: integral mart against rp} we have that 
    \begin{align*}
       \E\bigg[ \int_s^t \delta M_{s,r-} \otimes \dd X_r\bigg ] = \E[\delta [M,X]_{s,t}]= \E\bigg[\int_s^t \Delta X_r \dd M_r\bigg]=0
    \end{align*}
    and thus the claim follows, where we used that as $[0,T] \ni t \mapsto \Delta X_t $ is deterministic and thus predictable and bounded and thus especially $\Delta X \in L^1(M)$.
    
    For \eqref{eq: rough_stochastic_Fubini_2} we may use that we immediately have by construction, that 
    \begin{equation*}
        \E\bigg[\lim_{n \to \infty} \sum_{[u,v] \in \cP^n} Y_u \delta X_{u,v} + Y'_{u} \X_{u,v}\bigg] = \lim_{n \to \infty} \sum_{[u,v] \in \cP^n} \E[Y_u \delta X_{u,v} + Y'_{u} \X_{u,v}]
    \end{equation*}
    for any sequence of partitions $\{\cP^n\}_{n\in \N}$ of $[0,t]$ with vanishing mesh-size by \cite[Lemma 4.3]{AllanPieper2026}.
\end{proof}

\subsection{On strongly controlled processes}
We further adopt the following notion of strongly controlled RSMs from \cite{DauseFrizJentzenSong2026} into the càdlàg setting.
\begin{definition}
  \label{def:scRSM}
  Let $p \in [2,3)$ and $\mathbf{X}\in \mathscr{V}^{p}([0,T]; V)$. 
  Consider a $6$-tuple 
  \[\mathcal{Y}=(Y, \partial_{X} Y, \partial^{2}_{X}Y, V; M , N)\] 
of càdlàg, adapted processes of suitable dimensions:
\begin{equation}
\label{eq:dim_scRSM}
\mathcal{Y}: [0, T] \times \Omega \longrightarrow W \times \mathcal{L}(V; W) \times \mathcal{L}(V^{\otimes 2}; W)\times W \times W \times \mathcal{L}(V; W), 
\end{equation}
  where $(Y', Y''; N)\in \mathscr{H}_{X}^{p}([0,T]; \mathcal{L}(V; W))$. Further assume $M \in \mathcal{M}(W)$, $A \in V^{1}([0,T]; W)$ and 
  \begin{equation*}
  Y_{t}=Y_{0}+ A_{t} + M_{t} + \int_{0}^{t} (Y',  Y'')_s \dd \bX_s
  \end{equation*}
  for any $t \in [0,T]$, almost surely. 
  Then  the process $\mathcal{Y}$ is called a \textit{strongly} $\mathbf{X}$-\textit{controlled} \textit{rough semimartingale} (short: $\mathbf{X}$-scRSM). We denote the space of such processes by $\mathscr{H}_{\mathbf{X}}^{p}([0,T]; W)$. 
  \end{definition}
  
\begin{definition}
        Let $p \in [2,3)$ and $X: [0,T]\times \Omega \to V, \mathbb{X}: \Delta_{[0,T]}\times \Omega \to V^{\otimes 2}$ càdlàg, adapted with $\mathbf{X}=(X, \mathbb{X})\in \mathscr{V}^{p}([0,T]; V)$ almost surely. Then define the space of càdlàg, adapted strongly $\mathbf{X}$-controlled rough paths $L^{0}_{ad}(\Omega, \mathscr{V}^{p}_{\bX}([0,T]; W))$ as consisting of càdlàg adapted $(Y, Y', Y''): [0,T]\times \Omega \to W\times \mathcal{L}(V; W) \times \mathcal{L}(V^{\otimes 2}; W)$ such that $(Y, Y', Y'')(\omega)\in \mathscr{V}^{p}_{\mathbf{X}(\omega)}([0,T]; W)$ as defined in \cref{def:scRP} for $\mathbb{P}$-a.e. $\omega \in \Omega$.
  \end{definition}
The following is a càdlàg $p$-Var version of \cite[Proposition 4.10 \& Proposition 4.13]{DauseFrizJentzenSong2026}. 
\begin{proposition}
\label{prop:scRPtoscRSM}
   Let $p \in (2,3)$ and $\mathbf{X}\in \mathscr{V}^{p}([0,T]; V)$.  Then it holds; 
   \begin{enumerate}
       \item \label{prop:item:SCRP_to_RIP}Let $(Y, Y', Y'')\in L^{0}_{ad}(\Omega, \mathscr{V}^{p}_{(\mathbf{X}; M)}([0,T]; V\oplus W))$ with 
   \begin{equation*}
  \begin{aligned}
      Y'_{t}&\eqqcolon \left(\left( 
    \begin{array}{c}
    \partial_1 Y_{t}\\
    \partial_2 Y_{t}
    \end{array} \right)\cdot \right) \in \mathcal{L}(V \oplus W; W) ;\\
    Y''_{t}&\eqqcolon \left(\left( \begin{array}{ll}
    \partial_{1,1} Y_{t} & \partial_{1,2} Y_{t}\\
    \partial_{2,1} Y_{t} & \partial_{2,2} Y_{t}
  \end{array} \right) : \right) \in \mathcal{L}((V \oplus W)^{\otimes 2}; W).
  \end{aligned}
  \end{equation*}
   Then for $i = 1, 2$, $(\partial_i Y,
\partial_{1, i} Y; N_{i})\in \mathscr{H}_{X}^{p}$ with $N_i
:= \int_0 (\partial_{2, i} Y) d M \in \mathcal{M}$. If further $p \in [2,3)$, then almost surely
\begin{equation}
  \label{prop:eq:SCRPtoRIP} Y = N + \int (\partial_{X} Y, \partial_{X}^{2}Y) d \mathbf{X}, \quad \text{ where } N  :=
  \int_0 \partial_{M}Y d M \in \mathcal{M}^{c,1},  
\end{equation}
where $\partial_{X}Y\coloneqq\partial_{1}Y, \partial^{2}_{X}Y \coloneqq\partial_{1,1} Y$ and $\partial_{M}Y\coloneqq \partial_{2}Y$. 
\item  \label{prop:item:RIP_to_SCRP}Let $(Y, \partial_{X} Y, \partial^{2}_{X} Y, A; M , N)\in \mathscr{H}_{\mathbf{X}}^{p}([0,T]; W)$. 
  Then there exist processes $Y', Y''$ of the form
  \begin{equation}
  \label{prop:eq:RIP_to_SCRP}
  \begin{aligned}
      Y'_{t}&\coloneqq \left( \left( \begin{array}{c}
      \partial_{X} Y_{t}\\
      \operatorname{Id}\\
      0
      \end{array}\right) \cdot \right)\in \mathcal{L}(V\oplus W \oplus W; W) \\
      Y''_{t} &\coloneqq \left(\left(\begin{array}{ccc}
          \partial_{X}^{2} Y_{t} & 0 &0\\
          0&0&0 \\
          \operatorname{Id} &0&0
      \end{array}\right):\right)\in \mathcal{L}((V\oplus W \oplus W)^{\otimes 2}; W)
  \end{aligned}    
  \end{equation}
  such that $(Y - A, Y', Y'')\in L^{0}_{ad}(\Omega, \mathscr{V}^{p}_{(\mathbf{X}; M; N)}([0,T]; W))$. 
  \end{enumerate}
\end{proposition}
\begin{proof}
    For \eqref{prop:item:SCRP_to_RIP} the proof is similar to \cite[Proposition 4.10]{DauseFrizJentzenSong2026}. Notably showing that $(\partial_{i} Y, \partial_{1, i}Y; N_{i})\in \mathscr{H}_{X}^{p}$ follows analogously using \cref{lemma:RSM<->cRP}. Now by definition it holds almost surely, denoting by $\mathscr{J}^{(2)}(\mathbf{X}; M)$ the second level of the rough path $\mathscr{J}(\mathbf{X}; M)$,
    \begin{equation*}
        \begin{aligned}
            \delta Y_{s,t} \stackrel{\frac{p}{3}}{=} & Y'_{s} \delta (X; M)_{s, t}+ Y''_{s} \mathscr{J}^{(2)}(\mathbf{X}; M)_{s,t} \\
            =& \partial_{1} Y_{s} \delta X_{s,t}+ \partial_{2} Y_{s} \delta M_{s,t} + \partial_{1,1}Y_{s} \mathbb{X}_{s,t} + \partial_{1,2} Y_{s} \Pi(X; M)_{s,t}\\
            &+ \partial_{2, 1}Y_{s} \Pi(M ; X)_{s,t} + \partial_{2,2}Y_{s} \Pi(M; M)_{s,t}\\
            \stackrel{\frac{p}{3}}{=} & \int_{s}^{t} (\partial_{1}Y, \partial_{1,1}Y)_{r}\dd \mathbf{X}_{r}+ \partial_{2} Y_{s} \delta M_{s,t}+ \partial_{1,2} Y_{s} \Pi(X; M)_{s,t}\\
            &+ \partial_{2, 1}Y_{s} \Pi(M ; X)_{s,t} + \partial_{2,2}Y_{s} \Pi(M; M)_{s,t}- \Pi(N_{1}; X)_{s,t}.\\
        \end{aligned}
    \end{equation*}
    Now we are left to show that almost surely 
    \begin{equation}
    \label{eq:scRPtoscRSMHilfsaussage}
    \begin{aligned}
    A^{1}_{s,t}  &\coloneqq \delta N_{s,t} - \partial_{2}Y_{s} \delta M_{s,t} - \partial_{1,2} Y_{s} \Pi(X; M)_{s,t}- \partial_{2,2} Y_{s} \Pi(M; M)_{s,t}\stackrel{\frac{p}{3}}{=}0, \\
    A^{2}_{s,t} & \coloneqq \Pi(N_{1}; X)_{s,t} - \partial_{2,1} Y_{s} \Pi(M ; X)_{s,t} \stackrel{\frac{p}{3}}{=}0.
    \end{aligned}
    \end{equation}
    Using \cite[Theorem 1.1]{fzk23} and a suitable localization argument see \cite[Lemma 4.4]{fzk23}, one obtains $A^{i}\in V^{\frac{p}{3}}_{2}$ almost surely for $i=1,2$. Now define $\tilde{A}_{s,t}^{i}\coloneqq A^{i}_{s,t}- A^{i}_{s, t-}$. Then using the jump-structure of the It\^{o} integral it follows
    \begin{equation*}
    \begin{aligned}
        |\tilde{A}^{1}_{s, t}|&= \big| \left(\delta (\partial_{2} Y_{\cdot})_{s,t-}- \partial_{1,2} Y_{s} \delta X_{s, t-}- \partial_{2,2} Y_{s} \delta M_{s,t-}\right) \Delta M_{t}|\\
        &\leq w_{1}(s, t-)^{\frac{2}{p}}w_{2}(t-,t)^{\frac{1}{p}}\leq w_{1}(s, t-)^{\frac{2}{p}}w_{2}(s+,t)^{\frac{1}{p}}, 
    \end{aligned}   
    \end{equation*}
    for suitable controls $w_{1}, w_{2}$. Thus, $\tilde{A}^{1}_{s,t} \stackrel{\frac{p}{3}}{=}0$ and so $\tilde{A}^{1}\in V^{\frac{p}{3}}$ almost surely. Since $(s, t) \mapsto A_{s,t-}$ is right-continuous in $s$ and left-continuous in $t$ it follows $A^{1}_{s, t-}\stackrel{\frac{p}{3}}{=}0$ and therefore $A^{1}_{s,t}\stackrel{\frac{p}{3}}{=}0$. The claim for $A^{2}$ follows analogously using the jump-structure of the RSI \cref{lemma:RSI_cadlag_jumpstructure}. Thus, \eqref{eq:scRPtoscRSMHilfsaussage} is shown and so is the claim \eqref{prop:item:SCRP_to_RIP}. The claim \eqref{prop:item:RIP_to_SCRP} follows analogously to \cite[Proposition 4.13]{DauseFrizJentzenSong2026}. 
\end{proof}
\begin{remark}
\label{remark:commutation_diagramm}
    To see how \cref{prop:scRPtoscRSM} may be applied we define the bundles
\begin{equation*}
    \begin{aligned}
        \mathscr{V}^{p}_{\mathcal{M}}([0,T
        ]; W)&\coloneqq \bigcup_{\mathbf{X} \in \mathscr{V}^{p}(V), M \in \mathcal{M}(W)} \{\mathscr{J}(\mathbf{X}; M)\} \times L^{0}_{ad}(\Omega, \mathscr{V}^{p}_{\mathscr{J}(\mathbf{X}; M)}([0,T]; W))\\
        \mathscr{H}^{p}([0,T]; W)&\coloneqq \bigcup_{\mathbf{X} \in \mathscr{V}^{p}([0,T]; V)} \{\mathbf{X}\} \times \mathscr{H}_{\mathbf{X}}^{p}([0,T]; W). 
    \end{aligned}
\end{equation*}
Now \cref{prop:scRPtoscRSM} tells us that there is an explicit bijection $T: \mathscr{H}^{p}\to \mathscr{V}^{p}_{\mathcal{M}}$. assume now that we want to consider a pathwise map $\Phi: \mathscr{H}^{p} \to \mathscr{H}^{p}$ e.g. composition with smooth functions. Then trivially $\Phi = T^{-1} \circ \tilde{\Phi} \circ T$ with $\tilde{\Phi}\coloneqq T \circ \Phi \circ T^{-1}$. However, note that $\tilde{\Phi}$ is just a map on controlled rough paths, which are already well-established. Thus, in some situations it may be advantageous to follow the commutative diagramm  
\[
\begin{tikzcd}[
    column sep=7em,
    row sep=5em,
    arrows={line width=1pt}
]
\mathscr{H}^{p}
  \arrow[r, shift left=1ex, "T"]
  \arrow[d, "\Phi"']
&
\mathscr{V}^{p}_{\mathcal{M}}
  \arrow[l, shift left=1ex, "T^{-1}"]
  \arrow[d,  "\widetilde{\Phi}"]
\\
\mathscr{H}^{p}
  \arrow[r, shift left=1ex, "T"]
&
\mathscr{V}^{p}_{\mathcal{M}}
  \arrow[l, shift left=1ex, "T^{-1}"]
\end{tikzcd}
\]
over considering $\Phi$ directly. This is precisely the idea behind the rough stochastic It\^{o}-Wentzell formula \cite[Theorem 4.16]{DauseFrizJentzenSong2026} and the following rough stochastic It\^{o} formula. 
\end{remark}

\begin{proposition} \label{prop:rsIto}
    Let $p \in (2, 3)$ and $\mathbf{X}=(X, \mathbb{X})\in \mathscr{V}^{p}_{\mathbf{X}}$ as well as $\mathcal{Y}=(Y, \partial_{X} Y, \partial_{X}^{2} Y, A; M, N)\in \mathscr{H}_{\mathbf{X}}^{p}([0,T]; W)$. Further let $F \in \Lip^{3}(W; U)$. Then it holds for any $t \in [0,T]$ almost surely
    \begin{equation*}
    \begin{aligned}
        F&(Y_{t})-F(Y_{0}) \\
        =& \int_{0}^{t} (DF(Y_{s})\partial_{X}Y_{s} , DF(Y_{s-}) \partial_{X}^{2} Y_{s} + D^{2} F(Y_{s}) (\partial_{X} Y_{s}, \partial_{X} Y_{s}) )\dd \mathbf{X}_{s} + \int_{0}^{t} DF(Y_{s})\dd M_{s}\\
        &+ \int_{0}^{t} DF(Y_{s}) \dd A_{s} + \frac{1}{2} \int_{0}^{t} D^{2} F(Y_{s}) (\partial_{X} Y_{s}, \partial_{X} Y_{s}) \dd [\mathbf{X}]_{s}+ \frac{1}{2} \int_{0}^{t} D^{2} F(Y_{s}) \dd [M]_{s}\\
        &+ \sum_{0< s\leq t} F(Y_{s})- F(Y_{s-}) -DF(Y_{s-}) \Delta A_{s} - DF(Y_{s-}) (\partial_{X} Y_{s-} \Delta X_{s} + \Delta M_{s}) \\
        & \quad \quad -DF(Y_{s-}) \partial_{X}^{2} Y_{s-}\Delta \mathbb{X}_{s}  -\frac{1}{2} D^{2} F(Y_{s-}) \left( (\partial_{X}Y_{s-} \Delta X_{s}, \partial_{X} Y_{s-}\Delta X_{s}) +(\Delta M_{s}, \Delta M_{s})\right),
    \end{aligned}
    \end{equation*}
    where all of the term are either RSIs, Young-integrals or sums converging almost surely in RRS sense. 
\end{proposition}
\begin{proof}
    By applying \cref{prop:scRPtoscRSM} to $\mathcal{Y}$ we obtain explicit $(Y', Y'')$ such that $\tilde{Y}\coloneqq(Y-A, Y', Y'') \in L^{0}_{ad}(\Omega, \mathscr{V}^{p}_{(\mathbf{X}; M; N)}([0,T]; W))$. We then apply first \cref{theorem:roughIto} to $\tilde{Y}$ and $F$ and afterwards \cref{prop:scRPtoscRSM} to obtain the above formula. The details of this calculation are analogous to \cite[Theorem 4.16]{DauseFrizJentzenSong2026} and left for the readers. We would like to however put attention to the following two details appearing in the calculation, which are different from the continuous case. First by an easy integration by parts calculation one verifies, that 
    \begin{equation*}
[\sJ(\mathbf{X}; M; N)] =
\begin{pmatrix}
[\mathbf{X}] & [X, M] & [X, N] \\
[M, X] & [M] & [M, N] \\
[N, X] & [N, M] & [N]
\end{pmatrix}, 
\end{equation*}
where $[\mathbf{X}]$ denotes the rough path bracket, $[X, M], [X, N]$ and transposes thereof the brackets from \cite{fzk23} and all other terms the classical QV-brackets. Secondly we note that 
\begin{equation*}
    \Delta \Pi(N; X)_{t}= \lim_{s \uparrow t}\Pi(N; X)_{s, t}= \lim_{s \uparrow t}\left ( \delta X_{s,t} \delta M_{s,t} - \int_{s}^{t} \delta X_{s, r-} \dd M_{r} - \delta [X, M]_{s,t} \right)=0. 
\end{equation*}
Therefore the jumps of the joint lifts are not appearing in the above expression.\end{proof}
Compared to the It\^{o} formulae \cite[Theorem 2.12]{Friz_Zhang_2018} and \cite[Theorem 6.7]{AllanPieper2026}, \cref{prop:rsIto} is formulated on the level of strongly controlled processes and thus allows for more explicit decompositions. This implies e.g. the following rough stochastic product rule, which is a càdlàg, moment-free version of \cite[Theorem A.3]{FGHT26}.  
\begin{corollary}
\label{corollary:rs_product_formula}
    Let $p \in (2,3)$ and $\mathbf{X}=(X, \mathbb{X})\in \mathscr{V}^{p}([0,T]; V)$. Consider \\
    $\mathcal{Y}^{i}=(Y^{i}, \partial_{X} Y^{i}, \partial_{X}^{2} Y^{i}, A^{i}; M^{i}, N^{i})\in \mathscr{H}_{\mathbf{X}}^{p}([0,T]; \mathbb{R})$ for $i=1,2$. Let $\phi: \{1,2\} \to \{1, 2\}$ with $\phi(1)=2, \phi(2)=1$. Then it holds almost surely
    \begin{equation*}
    \begin{aligned}
        Y^{1}_{t} Y^{2}_{t} - Y^{1}_{0}Y^{2}_{0}=& \sum_{i=1,2}  \int_{0}^{t} (Y^{i}_{s} \partial_{X}Y^{\phi(i)}_{s}, Y^{i}_{s} \partial_{X}^{2}Y^{\phi(i)}_{s}+ \partial_{X}Y^{i}_{s} \partial_{X}Y^{\phi(i)}_{s})\dd \mathbf{X}_{s} \\
        &+ \int_{0}^{t} Y^{i}_{s-}\dd M^{\phi(i)}_{s} + \int_{0}^{t} Y^{i}_{s} \dd A^{\phi(i)}_{s} + \frac{1}{2}\int_{0}^{t} (\partial_{X} Y^{i}_{s}, \partial_{X}Y^{\phi(i)}_{s}) \dd [\mathbf{X}]_{s}+\frac{1}{2} [M^{1}, M^{2}]_{t} \\
        &+ \sum_{0< s \leq t} \frac{1}{2} Y^{i}_{s} Y^{\phi(i)}_{s}- \frac{1}{2} Y^{i}_{s-} Y^{\phi(i)}_{s-}- \frac{1}{2} (\partial_{X} Y^{i}_{s-}\Delta X_{s}, \partial_{X} Y^{\phi(i)}_{s-} \Delta X_{s}) - \frac{1}{2} \Delta M^{1}_{s} \Delta M^{2}_{s}\\
        & \quad - \partial_{X}Y^{i}_{s-} (\Delta A^{\phi(i)}_{s} + \Delta M^{\phi(i)}_{s} + \partial_{X}Y^{\phi(i)}_{s-} \Delta X_{s}+ \partial_{X}^{2}Y^{\phi(i)}_{s-} \Delta \mathbb{X}_{s}).
    \end{aligned}
    \end{equation*}
\end{corollary}

\begin{proof}
    Consider $\tilde{\mathcal{Y}}\coloneqq (\tilde{Y}, \partial_{X} \tilde{Y}, \partial_{X}^{2} \tilde{Y}, \tilde{A}; \tilde{M}; \tilde{N})\in \mathscr{H}_{\mathbf{X}}^{p}([0,T]; \mathbb{R}^{2})$, where $\tilde{G}\coloneqq (G^{1}, G^{2})^{\top}$ for $G^{i} \in \{ Y^{i}, M^{i}, A^{i} \}$ i=1,2, $\partial_{X} \tilde{Y}=( \partial_{X}Y^{1}, \partial_{X} Y^{2})^{\top}\cdot)$ and $\partial_{X}^{2} \tilde{Y}\coloneqq ((\partial_{X}^{2} Y^{1}, \partial_{X}^{2} Y^{2})^{\top}\cdot )$. Then apply \cref{prop:rsIto} to $\tilde{\mathcal{Y}}$ and $F(x,y)\coloneqq xy$. 
\end{proof}
  \begin{definition}
  \label{def:rIp}
    Consider a $5$-tuple $\mathcal{Y}=(Y, Y', Y'', A; M )$ of càdlàg, adapted processes of suitable dimensions as in \eqref{eq:dim_scRSM}. Further assume that $(Y', Y'')\in \sV^{p,m, n}_X$, 
    $M \in \mathcal{M}^{p, m, n}$ and $V \in \mathscr{V}^1$ such that 
    \begin{equation*}
    Y_{t}=Y_{0}+ A_{t} + M_{t} + \int_{0}^{t} (Y', Y'')_s \dd  \mathbf{X}_{s}
    \end{equation*}
    for any $t \in [0,T]$, almost surely. 
  Then  the process $\mathcal{Y}$ is called an $L^{m,n}$ \emph{rough Itô process} w.r.t. $\mathbf{X}$ (short: $L^{m, n}$-$\mathbf{X}$-rIp). We denote the space of such $\mathcal{Y}$'s by $\mathscr{V}_{\mathbf{X}}^{p, m, n}([0,T]; W)$.
  \end{definition}
  
\begin{definition}
      Consider a $6$-tuple $\mathcal{Y}=(Y, \partial_{X} Y, \partial^{2}_{X}Y, A; M , N)$. If $\mathcal{Y}$ is an $\mathbf{X}$-scRSM and $(Y, \partial_{X} Y, \partial^{2}_{X}Y, A; M)$ is an $L^{m,n}$-$\mathbf{X}$-rIp, we say that $\mathcal{Y}$ is a \emph{good $L^{m,n}$rough Itô process w.r.t. $\mathbf{X}$} (short: good $L^{m,n}$-$\mathbf{X}$-rIp or just 'good rough Itô process' if $m,n, \mathbf{X}$ are clear from context/irrelevant).
  \end{definition}

\subsection{A Kolmogorov criterion and rough stochastic Doob-Meyer decomposition}
\label{subsection:A Kolmogorov criterion and rough stochastic Doob-Meyer decomposition}

We first recall the Kolmogorov-type criterion for non-continuous $p$-Var settings as established in \cite{PisierXu1987} in discrete-time and \cite[Theorem 2.10]{AllanPieperTeichmann2025} in continuous-time.

\begin{theorem}\label{theorem: a p-variation kolmogorov result}
Let $Y \in V^p L^m$ for some $1 \leq p < m < \infty$. Then there exists a modification of $Y$ with sample paths which are almost surely c\`adl\`ag and such that for every $p<\tp$ we have that 
\begin{equation*}
    \Big\| \|Y\|_{\tp,[0,T]} \Big\|_{L^m} \leq C \|Y\|_{p,m,[0,T]}
\end{equation*}
where the constant $C>0$ only depends on $\tp, p, m$ and $T$. 
Conversely, let $Y \in L^m(V^p)$ with almost surely càdlàg sample paths by which we mean that $\|\|Y\|_{p,[0,T]}\|_{L^m} <\infty$\footnote{Note here that by the assumption that $Y$ is almost surely càdlàg we indeed ensure that $\|Y\|_{p,[0,T]}$ is a well defined random variable.} for some $1 \leq m \leq p < \infty$. Then also
\begin{equation*}
   \|Y\|_{p,m,[0,T]} \leq C  \Big\| \|Y\|_{p,[0,T]} \Big\|_{L^m}, 
\end{equation*}
where again the constant $C>0$ only depends on $p,m$ and $T$. 
\end{theorem}

In \cite[Example A.6]{AllanPieperTeichmann2025} the authors provide counterexamples, to show that the conditions in the first part of \cref{theorem: a p-variation kolmogorov result} on the coefficients $m$ and $p$ are indeed sharp. Since this break-down of the Kolmogorov-criterion for $m \leq p$ is crucial to the structure of \cref{section:On semi-linear RSDEs} we further motivate this condition by the following Besov-embedding heuristic. 

\begin{example} \label{example:Besov embedding}
Let $W$ be a finite-dim. Banach space and suppose $Y\in V^{p, c}L^{m}([0,T]; W)$. By \cite[Proposition 5.14]{Friz_Victoir_2010} we see that there is a continuous, non-decreasing, function $\varphi: [0,T]\to [0,1]$ with $\varphi(0)=0$ and $\varphi(T)=1$ and a continuous process $X\in \mathcal{C}^{\alpha}([0,1]; L^{m}(W; \mathbb{P}))$ with $\alpha=\nicefrac{1}{p}$ such that $Y= X\circ \varphi$ up to indistinguishability. 
One easily verifies that for any $\tilde{p}\in (p, \infty)$ with $\tilde{\alpha}\coloneqq \frac{1}{\tilde{p}}$ it holds $X\in \mathcal{B}^{\tilde{\alpha}}_{m}(L^{m}(W; \mathbb{P})))$. Now if $\tilde{\alpha} m > \operatorname{dim}([0,T])=1$ we have 
\begin{equation*}
     \mathcal{B}_{ m}^{\tilde{\alpha}}(L^{m}(W; \mathbb{P}))= L^{m}(\mathcal{B}^{\tilde{\alpha}}_{m}(W); \mathbb{P})\subset L^{m} ( V^{\tilde{p}}(W); \mathbb{P}),
\end{equation*}
where the last inclusion follows by the Besov-embedding \cite[Corollary A.3]{Friz_Victoir_2010}. Therefore $X\in L^{m}(V^{\tilde{p}}([0,T]; W))$. By the invariance of $\Vert \cdot \Vert_{\tilde{p}}$ under reparameterization we see $Y\in L^{m}(V^{\tilde{p}}([0,T]; W))$. Note that for $\tilde{\alpha}m\leq 1$ the Besov-embedding is known to break down and thus the above argument fails for $m \leq p$. 
\end{example}

In order to establish a first connection between the notions of stochastic controlled paths and rough semimartingales based on Kolmogorov-type criteria, we will need the following two-parameter generalisation of the previous result, similar to the $\alpha$-Hölder counter part (see \cite[Theorem A.1]{DauseFrizJentzenSong2026} or \cite[Lemma 5.8]{BCN24}) to also treat the pathwise regularity of the remainder of a stochastic controlled path.

\begin{theorem}\label{thm: Kolmogorov for some remainders}
   Let $p \in [1,\infty)$, $\tp \in (p,\infty)$, $m\in [2, \infty)$ and $m_2\in (p, \infty]$ as well as $(R_{s,t})_{(s,t) \in \Delta_{[0,T]}}$ a right-continuous in probability two parameter processes and assume that there exist two processes $R^{1} \in V^{\tilde{p}}$ almost surely as well as $R^{2} \in V^{p} L^{m_2}$ such that 
    \begin{equation}\label{eq: Chen for remainders}
        \delta R_{s,u,t} = \delta R^{(1)}_{s,u} \otimes \delta R^{(2)}_{u,t}.
    \end{equation}
    If $R \in V^{\frac{p}{2}} L^m$, then there exists a modification $\tilde{R}$ of $R$ with $\tilde{R}\in V^{\frac{\tilde{p}}{2}}$ almost surely as a two-parameter process. \\
    Moreover, if additionally $R^{(1)} \in L^{m_1}(V^{\tp})$ for $m_1 \in [2,\infty)$ such that $\frac{1}{m_1}+ \frac{1}{m_2}=\frac{1}{m}$ then also 
    \begin{equation*}
        \Big\|\|R\|_{\frac{\tp}{2},[0,T]}\Big\|_{L^{m}}\lesssim \|R\|_{\frac{p}{2},m,[0,T]}+ \big\|\|R^{(1)}\|_{\tp, [0,T]}\big\|_{L^{m_1}} \|R^{(2)}\|_{p,m_2,[0,T]}.
    \end{equation*}
\end{theorem}

The following result may be interpreted as a form of Doob--Meyer decomposition which is an extension of the corresponding $\alpha$-Hölder version \cite[Theorem 3.3]{FHL21} and in combination with the previously discussed Kolmogorov-criteria especially shows a first consistency between SCRPs and RSMs. 

\begin{proposition}\label{prop: Doob--Meyer}
         Let $p\in [2,3)$, $\tp \in (p, \infty)$, $m \in [2,\infty)$ and $n \in [m,\infty]$ and $\bX=:(X, \X)$ a deterministic $p$-rough path. If $(Y,Y') \in \sV^{p,m,n}_X$ is such that for any $s\leq u \leq t$ we have that 
         \begin{equation}\label{eq: jump bound for remainder Doob--Meyer}
             \Big\|\big(\E_s-\E_u\big)\big[\Delta R^{Y,Y'}_{t}\big] \Big\|_{L^n} \lesssim \sum_{i=1}^{N} w_{1,i}(s,t-)^{\alpha_{1,i}}w_{2,i}(s+,t)^{\alpha_{2,i}}
         \end{equation} 
         for some $N\in \N$, some controls $w_{1,i}, w_{2,i}$ for every $i\in \{1, \ldots N\}$ and $\alpha_{1,i}, \alpha_{2,i}>0$ such that $\alpha_{1,i}+\alpha_{2,i}>\frac{1}{2}$, then there exists a unique adapted càdlàg $M\in \cM^{p,m,n}$ with $M_0=0$ such that $(Y-M, Y')$ is an $L^{m,n}$-valued $X$-controlled path in the sense that 
        \begin{equation}\label{eq: bound Doob--Meyer on remainder}
            \big\|R^{(Y-M, Y')} \big\|_{m,n,s} \lesssim \|Z'\|_{p,m,n,[s,t)} \|X\|_{p,[s,t]} + \|\E_{\cdot}R^{Y,Y'}\|_{\frac{p}{2},n, [s,t]}.
        \end{equation} 
         Thus, especially if $p<m$, then there exists an almost surely càdlàg modification $(\tY, \tY')$ of $(Y,Y')$ such that $(\tY, \tY'; M) \in \mathscr{H}^{\tp,m}_X$ with  
         \begin{equation}\label{eq: bound Doob--Meyer RSM p<m}
             \|\tY,\tY; M\|_{\mathscr{H}^{\tp,m}_X} \lesssim \|Y,Y'\|_{p,m,[0,T]}.
         \end{equation}
         Further, if $n=\infty$ and $m\leq p$, then there exists an almost surely càdlàg modification $\tY$ of $Y$ such that $\tY \in V^{\tp}$ almost surely. If additionally there exists an almost surely càdlàg, modification $\tY'$ of $Y'$ with  $\tY'\in V^{\tp}$ almost surely as well, then again $(\tY,\tY'; M) \in \mathscr{H}^{\tp}_X$. 
    \end{proposition}

    \begin{proof}
        We follow the same line of argument as in the proof of \cite[Theorem 3.3]{FHL21} and define $J_{s,t}:= \E_{s}[\delta Y_{s,t}]$, then one immediately checks, that $\delta J_{s,u,t}=(\E_{s}-\E_{u})[\delta Y_{u,t}]$ and by Chen's relation also
        \begin{equation*}
        \begin{split}
            \big(\E_{s}-\E_{u}\big)[\delta Y_{u,t}] &= \big(\E_{s}-\E_{u}\big)[\delta Y'_{s,u} \delta X_{u,t}]+ \big(\E_{s}-\E_{u}\big)[R^{Y,Y'}_{u,t}]\\
            &= \big(\E_{s}-\E_{u}\big)[\delta Y'_{s,u} \delta X_{u,t}]+ \big(\E_{s}-\E_{u}\big)[R^{Y,Y'}_{u,t-}] \\
            &\quad + \big(\E_{s}-\E_{u}\big)[\delta Y'_{u,t-}  \Delta X_{t}] +  \big(\E_{s}-\E_{u}\big)[\Delta R^{Y,Y'}_{t}]. 
        \end{split}
        \end{equation*}
       Therefore, we especially have that
        \begin{align*}
            \Big\|\Big(\E_{s}-\E_{u}\Big)(\delta Y_{u,t})\Big\|_{m,n,s}
            &\lesssim \|Y'\|_{p,m,n,[s,t)}\|X\|_{p,[s,t]} + \|\E_{\cdot} R^{Y,Y'}\|_{\frac{p}{2}, n, [s,t)}\\
            &\quad + \sum_{i=1}^{N} w_{1,i}(s,t-)^{\alpha_{1,i}}w_{2,i}(s+,t)^{\alpha_{2,i}}.
        \end{align*}
        Thus, we find that $J$ fulfills the assumptions of the stochastic sewing lemma (see \cite[Theorem 3.2]{AllanPieper2026})\footnote{\label{footnote: generalisation of stochastic sewing}The stochastic sewing results of \cite[Theorems 3.2]{AllanPieper2026} extend to the present setting without any substantive change to their proofs: it can be shown that the RRS convergence only requires bounds of the form
        $\sum_{i=1}^{N}w_{1,i}(s,t-)^{\alpha_{1,i}}w_{2,i}(s+,t)^{\alpha_{2,i}}$. The corresponding general statement is included in J.~Pieper's forthcoming doctoral thesis \cite{Pieper2026}.}  and the existence of $M$ and the bound \eqref{eq: bound Doob--Meyer on remainder} follow immediately by an application thereof. 
        Especially, we find $\cJ= Y-M$ fulfills 
        \begin{align*}
            \|\delta \cJ_{s,t}- J_{s,t}\|_{m,n,s} &\lesssim \|Y'\|_{p,m,n,[s,t)}\|X\|_{p,[s,t]} + \|\E_{\cdot} R^{Y,Y'}\|_{\frac{p}{2}, n, [s,t)}+ \sum_{i=1}^{N} w_{1,i}(s,t-)^{\alpha_{1,i}}w_{2,i}(s+,t)^{\alpha_{2,i}}\\
            &=: w(s,t)^{\frac{1}{2}}
        \end{align*}
         and $\E_{s}[\delta \cJ_{s,t}- J_{s,t}]=0$, where we defined the control 
        \begin{align*}
             w(s,t):= \|Y'\|^{2}_{p,m,n,[s,t)}\|X\|^{2}_{p,[s,t]} + \|\E_{\cdot} R^{Y,Y'}\|^{2}_{\frac{p}{2}, n, [s,t)} + \sum_{i=1}^{N} w_{1,i}(s,t-)^{2\alpha_{1,i}}w_{2,i}(s+,t)^{2\alpha_{2,i}}.
        \end{align*}
        An application of \cite[Lemma B.2]{AllanPieper2026} in the same spirit as \cite[Theorem 3.3 (vi)]{Le2023} then shows the uniqueness. 
        Whenever $p<m$ then a direct application of \cref{theorem: a p-variation kolmogorov result} and \cref{thm: Kolmogorov for some remainders} implies \eqref{eq: bound Doob--Meyer RSM p<m} directly. Here we point out that $R^{\tY, \tY'}$ is version of $R^{Y,Y'}$ and thus $\|\E_{\cdot}R^{\tY,\tY'}\|_{\frac{p}{2}, n , [0,T]}\vee \|R^{\tY,\tY'}\|_{\frac{p}{2},m, n, [0,T]} <\infty$. Further, as $(s,t) \in \Delta_{[0,T]} \mapsto R^{\tY, \tY'}$ is jointly right-continuous almost surely, the version of almost surely finite $\frac{\tp}{2}$-variation, found through \cref{thm: Kolmogorov for some remainders} coincides with $R^{\tY, \tY'}$ up to indistinguishability.
        Finally, consider, $m\leq p$ and $n=\infty$. Define the discrete approximation of $J$ for a partition $\cP= \{t_i\}_{i=0}^N$ by
        \begin{equation*}
            J^{\cP}_t := \sum_{i=0}^{N}  J_{t_i,t_{i+1}} \1_{[t_{i+1},T]}(t).
        \end{equation*}

        By definition, we have that $J^{\cP}$ is piecewise-constant and especially almost surely càdlàg.
        Write $\lfloor s \rfloor_{\cP}:= \sup \{u\in \cP \ : \ u\leq s\}$ and $\llceil s \rrceil_{\cP }:= \inf \{u\in \cP \ : \ u> s\} \wedge T$.
        Then, also by definition of $J^{\cP}$ we find that for any $(s,t) \in \Delta_{[0,T]}$ we have that 
        \begin{equation*}
            \delta J^{\cP}_{s,t}= \sum_{\substack{i\leq N \\ [t_i,t_{i+1}]\subset [s,t]}}  J_{t_i,t_{i+1}} + J_{\lfloor s \rfloor_{\cP} , \llceil s \rrceil_{\cP}\wedge \lfloor t\rfloor_{\cP}},
        \end{equation*}
        where we use the convention that a sum over the empty set equals $0$.
        Thus, especially by a standard BDG argument (see \cite[Lemma 7.1]{AllanPieper2026}) applied to $z_i=\cJ_{t_i,t_{i+1}} - J_{t_i,t_{i+1}}$ we find that 
        \begin{align*}
           &\| \delta J^{\cP}_{s,t}\|_{m,\infty,s}\\
           &\leq \bigg\| \sum_{\substack{i\leq N \\ [t_i,t_{i+1}]\subset [s,t]}}  \delta \cJ_{t_i,t_{i+1}} - J_{t_i,t_{i+1}}\bigg\|_{m,n,s} + \|J_{\lfloor s \rfloor_{\cP}, \llceil s \rrceil_{\cP} \wedge \lfloor t\rfloor_{\cP}}\|_{m,n,s} + \|\delta \cJ_{\llceil s \rrceil_{\cP }\wedge \lfloor t\rfloor_{\cP},\lfloor t \rfloor_{\cP}}\|_{m,n,s} \\
           &\leq \bigg(\sum_{\substack{i\leq N \\ [t_i,t_{i+1}]\subset [s,t]}} \|\delta \cJ_{t_i,t_{i+1}} -J_{t_i,t_{i+1}}\|^2_{m,n,t_i}\bigg)^{\frac{1}{2}}+ \|\delta Y_{\lfloor s \rfloor_{\cP}, \llceil s \rrceil_{\cP}\wedge \lfloor t\rfloor_{\cP}}\|_{m,\infty, \lfloor s \rfloor_{\cP}}+ \|\delta\cJ_{\llceil s \rrceil_{\cP }\wedge \lfloor t\rfloor_{\cP},\lfloor t \rfloor_{\cP}}\|_{m,n,s}\\
           &\leq w(s,t)^{\frac{1}{2}}+\|\cJ\|_{p,m,\infty, [\lfloor s\rfloor_{\cP}, \lfloor t\rfloor_{\cP}]}+ \|Y\|_{p,m,\infty, [\lfloor s\rfloor_{\cP}, \lfloor t\rfloor_{\cP}]},
        \end{align*}
        where we used that 
        \begin{align*}
            \|J_{\lfloor s \rfloor_{\cP}, \llceil s \rrceil_{\cP}\wedge \lfloor t\rfloor_{\cP}}\|_{m,\infty,s} &= \Big\| \E_{s} \big[\big| \E_{\lfloor s \rfloor_{\cP}}\big[\delta Y_{\lfloor s \rfloor_{\cP}, \llceil s \rrceil_{\cP}\wedge \lfloor t\rfloor_{\cP}}\big]\big|^m\big]^{\frac{1}{m}} \Big\|_{L^n}\\
            &= \Big\| \big| \E_{\lfloor s \rfloor_{\cP}}\big[\delta Y_{\lfloor s \rfloor_{\cP}, \llceil s \rrceil_{\cP} \wedge \lfloor t\rfloor_{\cP}}\big]\big| \Big\|_{L^n}\leq \|\delta Y_{\lfloor s \rfloor_{\cP}, \llceil s \rrceil_{\cP} \wedge \lfloor t\rfloor_{\cP}}\|_{m,n,\lfloor s \rfloor_{\cP}}
        \end{align*}
        
        As for any other partition $\tilde \cP$ the set $\{\lfloor u \rfloor_{\cP} \ : \ u\in \tilde \cP\} =: \bar{\cP}$ is yet another partition of $[0,T]$, we find that 
        \begin{equation*}
        \begin{split}
            \sum_{[s,t] \in \tilde \cP}\|\cJ\|^p_{p,m,\infty, [\lfloor s\rfloor_{\cP}, \lfloor t\rfloor_{\cP}]}+ \|Y\|^p_{p,m,\infty, [\lfloor s\rfloor_{\cP}, \lfloor t\rfloor_{\cP}]} &= \sum_{[u,v] \in \bar{\cP}}\|\cJ\|^p_{p,m,\infty, [u,v]}+ \|Y\|^p_{p,m,\infty,[u,v]}\\
            &\leq \|\cJ\|^p_{p,m,\infty, [0,T]}+ \|Y\|^p_{p,m,\infty, [0,T]}.
        \end{split}
        \end{equation*}
      Thus, especially
        \begin{equation*}
            \| J^{\cP}\|_{p, m,\infty,[0,T]} \leq  w(0,T)^{\frac{1}{2}}+ \|\cJ\|_{p,m,\infty, [0,T]}+ \|Y\|_{p,m,\infty, [0,T]}
        \end{equation*}
        for every partition $\cP$. Notice, that also, by definition we know that 
        \begin{equation*}
            \Gamma_{s,t}^{\cP} := \sup_{v\in [s,t]} |\Delta J^{\cP}_v| \leq  \sup_{\substack{[u,v] \in \cP; v\in [s,t]}} \E_{u}[|\delta Y_{u,v}|]
        \end{equation*}
        and thus one easily checks, that .
        \[\big\|\Gamma^{\cP} \big\|_{p,\infty, [0,T]} \lesssim \|Y\|_{p,m,\infty, [0,T]}.\]
        Therefore, by an application of \cite[Proposition 4.3]{AllanPieperTeichmann2025} we know that 
        \begin{align*}
            \sup_{\cP} \|J^{\cP}\|_{\textup{BMO}^{p\text{-var}}, [0,T]} \lesssim  w(0,T)+ \|\cJ\|_{p,m,\infty, [0,T]}+ \|Y\|_{p,m,\infty, [0,T]}.
        \end{align*}
        Moreover, as $J^{\cP}_0=0$, we have for every $\tilde{m}\in [2,\infty)$, again by \cite[Proposition 4.3]{AllanPieperTeichmann2025}, also that 
        \begin{equation*}
            \sup_{\cP} \|J^{\cP}\|_{p,\tilde{m}, \infty, [0,T]} \lesssim w(0,T)+ \|\cJ\|_{p,m,\infty, [0,T]}+ \|Y\|^p_{p,m,\infty, [0,T]},
        \end{equation*}
        where the constant may depend on $\tilde m$.
        Fix a partition $\cP^{0}$ of $[0, T]$, then, by the RRS convergence established above we may choose a nested sequence of refinemenets of $\cP^0$ such that for every $(s,t)\in \cP^{0}$ we have that 
        \begin{equation*}
            \|\delta \cJ_{s,t}- \delta J^{\cP^n}_{s,t}\|_{L^m}=\bigg\|\delta \cJ_{s,t}- \sum_{(u,v) \in \cP^n|_{[s,t]}} J_{u,v}\bigg \|_{L^m} \longrightarrow 0
        \end{equation*}
        as $n\to \infty$. As there are only finitely many $(s,t)\in \cP^{0}$, we may choose a subsequence $(n_k)_{k\in \N}$ such that 
        \begin{equation*}
            |\delta \cJ_{s,t} - J^{\cP_{n_k}}_{s,t}| \longrightarrow 0
        \end{equation*}
        $k\to \infty$ as a $\P$-almost sure limit. Thus, by the Fatou's lemma,  \cite[$(v)$~in~Proposition 2.2]{AllanPieper2026}, we find, for any $\tilde{m} \in [2,\infty)$ that 
        \begin{equation*}
            \sum_{(s,t) \in \cP^{0}} \|\delta \cJ_{s,t}\|^p_{\tilde{m} ,\infty,s}\leq \sum_{(s,t) \in \cP^{0}} \liminf_{k\to \infty} \|\delta J^{\cP^{n_k}}_{s,t}\|^p_{\tilde{m},\infty,s} \lesssim w(0,T)^{\frac{p}{2}}+ \|\cJ\|^p_{p,m,\infty,[0,T]}+\|Y\|^p_{p,m,\infty,[0,T]}
        \end{equation*}
        with the constant being independent of $\cP^{0}$. An application of \cref{theorem: a p-variation kolmogorov result} implies the existence of the wanted modification. Whenever there is a almost surely càdlàg version of $Y'$ we may conclude the final claim as before by \cref{thm: Kolmogorov for some remainders}. 
\end{proof}

\begin{example}
One can easily convince oneself of the following situations where \eqref{eq: jump bound for remainder Doob--Meyer} is fulfilled:
\begin{enumerate}
    \item it is trivially fulfilled if $(Y,Y')$ is a continuous stochastic controlled path à \cite{FHL21} and more generally if $Y$ only jumps at totally inaccessible times.
    \item unsurprisingly, if we let $\cY=(Y,Y',Y'',A;M) \in \sV^{p,m,n}_{\bX}$ such that $\delta A \stackrel{\frac{p}{2}}{=} 0$ in $L^{m,n}$, then one may use \eqref{eq: Lqr- estimate for rough stochastic integral} to establish \eqref{eq: jump bound for remainder Doob--Meyer} once again. Thus, especially whenever $p<m$ or $n=\infty$, and if $M,A, \bX, (Y',Y'')$ are all continuous almost surely, then the $L^{m,n}$-rough Itô process $\cY = (Y,Y',Y'',A; M)$ fulfills even $(Y,Y',Y'',A; M, M^{Y'}) \in \sH^{p}_{\bX}$, where $M^{Y'}$ is the martingale in the Doob--Meyer decomposition of $(Y',Y'')$.
    \item Let $Y_t= X_t + \1_{[t_0,T]}(t) \xi$ for $\xi\sim \mathfrak{F}_{s_0}$, $\xi \perp\!\!\!\perp \mathfrak{F}_{s_0-}$ with $\P[\xi = \pm 1]=\frac{1}{2}$, for some $s_0<t_0$, then one easily checks that \eqref{eq: jump bound for remainder Doob--Meyer} is fulfilled wit $w_1(s,t) := \delta_{s_0}((s,t])$ and $w_2(s,t) \coloneqq \delta_{t_0}((s,t])$.
\end{enumerate}
\end{example}

As a direct consequence of \cref{prop: Doob--Meyer} we also find the following rough stochastic analogue of \cref{remark:scRPandRI}:
\begin{corollary}\label{corollary: Doob-Meyer for strongly controlled}
   Let $p\in [2,3)$, $m\in [2,\infty)$ and $n\in [m,\infty]$ as well as  $(Y,Y'), (Y',Y'') \in \sV^{p,m,n}_{\bX}$ be a pair of stochastic controlled paths such that
   \begin{equation}\label{eq: remainder assumption on strongly stochastically controlled}
       \big\|\E_{s}[\delta Y_{s,t}- Y'_s \delta X_{s,t}- Y''_s\X_{s,t}]\big\|_{L^n} \stackrel{r}{=}0
   \end{equation}
   for some $r\in (0,\frac{p}{2}]$. Then there exists a martingale $M^{Y}$ and an $A\in V^{r}L^{m,n}$ with $\delta A \stackrel{r}{=} 0$ in $L^n$ such that for every $t\in [0,T]$ we have that 
   \begin{equation*}
       Y_t= Y_0+ A_t+ M^{Y}_t +\int_0^t (Y',Y'')_s \dd \bX_s \qquad \P\text{-a.s.}
   \end{equation*}
   Especially, $A\equiv 0$ whenever $r<1$.
\end{corollary}
\begin{remark}
\label{remark:continuous_case}
    In the continuous $\alpha$-Hölder setting, the Doob--Meyer
    decomposition for stochastic controlled rough paths was already
    established in \cite{FHL21,Le2023}. Combining this decomposition
    with the classical Kolmogorov criterion and its extension to remainder processes \cite[Lemma 5.8]{BCN24}, one obtains the
    following: if $m>\frac{1}{\alpha}$ and
    $(Y,Y')\in\mathbf{D}^{2\alpha}_X L^{m,n}$, then, for every
    $\beta\in(0,\alpha-\frac{1}{m})$, there exists a martingale $M$
    such that $(Y,Y';M)$ is a $\beta$-Hölder $L^m$-rough
    semimartingale.\footnote{By which we mean the straightforward
    analogue of \cref{definition: Lm RSM} in the Hölder setting,
    with the corresponding pathwise controlled Hölder seminorms
    in $L^m$; see also \cite{DauseFrizJentzenSong2026}.}
    For $n=\infty$, the connection between
    $\textup{VMO}^{\alpha}$-processes and
    $C^{\alpha}L^{m,\infty}$-processes developed in
    \cite{Le2022a} suggests that the full range
    $\beta\in(0,\alpha)$ could likewise be recovered, after suitably treating the remainder process.

    For $n=m$, the general passage from moment bounds to pathwise
    Hölder bounds entails a loss of Hölder regularity depending
    on $m$. In contrast, the $p$-Var Kolmogorov criterion
    above allows an arbitrarily small increase of the variation
    exponent, subject to the condition $p<m$. In the continuous
    setting, a related approach could also be based on
    Besov-space estimates \cite{FRIZ2022152}, after a suitable extension to controlled remainder processes; see also the discussion in \cref{example:Besov embedding}.

    A converse implication holds in the Hölder setting when the
    corresponding pathwise Hölder norms are $L^m$-integrable.
    More precisely, let $(Y,Y';M)$ be an $\alpha$-Hölder
    $L^m$-rough semimartingale, and
    assume that \footnote{Notice that we here include the $\alpha$-Hölder norm of the martingale. Upon a suitable localization one can verify this assumption for (local) $L^1$-martingales whose quadratic variation is Lipschitz continuous in time; see \cite[Lemma 4.2]{DauseFrizJentzenSong2026}. We further point out, however, that in order to establish finite $p$-variation as an element of $L^m$ no such restriction is necessary due to Lepingle's inequality.}
    \[
        \big\|\|Y\|_{\alpha,[0,T]}\big\|_{L^m}
        +\big\|\|Y'\|_{\alpha,[0,T]}\big\|_{L^m}
        +\big\|\|R^{Y-M,Y'}\|_{2\alpha,[0,T]}\big\|_{L^m}
        <\infty.
    \]
    Together with  corresponding integrability assumptions on the initial values,
    these bounds imply $(Y,Y')\in\mathbf{D}^{2\alpha}_X L^m$.
    Indeed,
    \begin{align*}
        \sup_{0\leq s<t\leq T}
        \frac{\|\delta Y_{s,t}\|_{L^m}}{|t-s|^\alpha}
        &\leq
        \left\|
            \sup_{0\leq s<t\leq T}
            \frac{|\delta Y_{s,t}|}{|t-s|^\alpha}
        \right\|_{L^m},
    \end{align*}
    and the same estimate applies to $Y'$. Moreover, since
    $\E_s[\delta M_{s,t}]=0$, conditional Jensen's inequality gives
    \begin{align*}
        \sup_{0\leq s<t\leq T}
        \frac{\|\E_s[R^{Y,Y'}_{s,t}]\|_{L^m}}
             {|t-s|^{2\alpha}}
        &\leq
        \left\|
            \sup_{0\leq s<t\leq T}
            \frac{|R^{Y-M,Y'}_{s,t}|}{|t-s|^{2\alpha}}
        \right\|_{L^m}.
    \end{align*}

    In the càdlàg setting, however, integrable pathwise $p$-Var does not in general imply $p$-Var in mixed moments. To see this, let $P$ be a
    unit-rate Poisson process and set $L_t=P_t-t$. Then
    $L\in\cM^m$ for every $m\in[2,\infty)$ and  $\|L\|_{p,[0,T]}\in L^m$ for every $p\geq1$ and
    every $m\in [2,\infty)$ on the other hand $L\notin V^p L^{m}$ for any $p<m$ as already described in \cite[Example 2.16]{AllanPieper2026}.
\end{remark}

\section{A new approach to rough stochastic differential equations}
\label{Section:main_theorem}
For the remainder of this work we will consider rough stochastic differential equations (short:RSDEs) of the form 
\begin{equation}
\label{main_ref:RSDE}
Y_{t}= \xi + \int_{0}^{t} b_{s}(Y_{s}) \dd s+ \int_{0}^{t} \sigma_{s} (Y_{s-}) \dd N_{s} + \int_{0}^{t} (f, f')_{s} (Y_{s}) \dd \mathbf{X}_{s}
\end{equation}
 for any $t \in [0,T]$, where $\xi \in L^{0}(\mathfrak{F}_{0}; \mathbb{R}^{d_{Y}}), N\in \mathcal{M}(\mathbb{R}^{d_{N}}), \mathbf{X}\in \mathscr{V}^{p}$ for $p \in (2,3)$ and 
 \begin{equation*}
     \begin{aligned}
         (f, f'): [0, T]\times \Omega &\to C^{1}\left(\mathbb{R}^{d_{Y}}; \mathcal{L}(\mathbb{R}^{d_{X}}; \mathbb{R}^{d_{Y}}) \times \mathcal{L}((\mathbb{R}^{d_{X}})^{\otimes 2}; \mathbb{R}^{d_{Y}})\right)\\
         \sigma: [0,T] \times \Omega  &\to C\left( \mathbb{R}^{d_{Y}}; \mathcal{L}(\mathbb{R}^{d_{N}}; \mathbb{R}^{d_{Y}}) \right),\\
         b: [0,T]\times \Omega &\to  C\big( \mathbb{R}^{d_{Y}}; \R^{d_{Y}}\big)
     \end{aligned}
 \end{equation*}
 where $
 \sigma$ is predictable and $b,(f, f')$ are progressively measurable. Here we use the short-hand notation $(f, f')_{s}(Y_{s})\coloneqq (f_{s}(Y_{s}), Df(Y_{s}) f(Y_{s})+ f'_{s}(Y_{s}))$. We will be focusing on the following two choices of $(f, f'):$
 \begin{assumption}
 \label{assumption:vector_fields}
     Let $(f, f')$ be as above. Further assume that one of the following is satisfied:
      \begin{enumerate}[label=\Roman*.]
     \item \label{assumption:item:vector_fields_1} It holds $(f, f')\in L^{0}_{\operatorname{prog}}(\Omega, \mathscr{V}^{p}_{X}\Lip^{3}(\mathcal{L}(\mathbb{R}^{d_{X}}; \mathbb{R}^{d_{Y}})))$.
     \item \label{assumption:item:vector_fields_2} It holds $(f, f'; G)\in \mathscr{H}_{X}^{p}\cL(\mathbb{R}^{d_{Y}}; \cL(\mathbb{R}^{d_{X}}; \mathbb{R}^{d_{Y}}))$ for some $G \in \mathcal{M}$.
 \end{enumerate}
 \end{assumption}

\begin{definition}
\label{def:solution_RSDE}
    We say an $(\mathfrak{F}_{t})$-adapted process $(Y_{t})_{t \in [0,T]}$ solves the RSDE \eqref{main_ref:RSDE} if $\int (f, f')(Y) \dd \mathbf{X}$ is a well-defined RSI in the sense of \cref{def:RSI} and 
    \begin{equation*}
        Y =\xi + \int_{0}^{\cdot} b_{s}(Y_{s}) \dd s+\int_{0}^{\cdot} \sigma_{s}(Y_{s-}) \dd N_{s} + \int_{0}^{\cdot}(f, f')_{s}(Y_{s}) \dd \bX_{s}
    \end{equation*}
    on $[0,T]$ up to indistinguishability. We further say that such $Y$ is 
    \begin{enumerate}
        \item an $L^{0}$\emph{-solution} if for 
        \begin{equation*}
            \mathcal{Y}\coloneqq (Y, f(Y), Df(Y)f(Y)+ f'(Y), b(Y); \sigma(Y) \bullet \d N, Z)
        \end{equation*} 
        it holds $\mathcal{Y}\in \mathcal{H}^{\tilde{p}}_{\mathbf{X}}([0,T]; \mathbb{R}^{d_{Y}})$ for any $\tilde{p}\in (p, \infty]$ \footnote{To simplify notation we abbreviate $\int_{0}^{\cdot} b_{s}(Y_{s})\dd s$ as '$b(Y)$' for the drift component of the RSDE written as an element of $\mathscr{H}_{\mathbf{X}}$}, where 
        \begin{equation*}
            Z = \begin{cases} 
Df(Y) \sigma(Y) \bullet \d N  & \text{if } (f, f') \text{  satisfy \eqref{assumption:item:vector_fields_1}}  \\ 
 Df(Y) \sigma(Y) \bullet \d N +  \d G \bullet Y& \text{if } 
 (f, f') \text{  satisfy \eqref{assumption:item:vector_fields_2}} 
\end{cases}
\end{equation*}
        \item an $L^{m,n}$\emph{-solution} for $2 \leq m < \infty$ and $n \in [m, \infty]$ if for 
        \begin{equation*}
        \tilde{\mathcal{Y}}\coloneqq (Y, f(Y), Df(Y)f(Y)+ f'(Y), b(Y); \sigma(Y)\bullet \d N)
        \end{equation*}
        it holds $\tilde{\mathcal{Y}}\in \mathscr{V}^{p, m,n}_{\mathbf{X}}([0,T]; \mathbb{R}^{d_{Y}})$.
        \item a \emph{good} $L^{m,n}$\emph{-solution} if it is both an $L^{0}$- and $L^{m,n}$-solution,.
    \end{enumerate}
\end{definition}

\begin{remark}\label{rem: difficulties with stochastic controlled vfs}
In order to consider a general non-linear stochastic controlled vector field $(f, f')\in \sV^{p, m, \infty}_{X}\Lip^{3}$ à \cite{FHL21, AllanPieper2026} in \eqref{main_ref:RSDE} in the context of this section, 
one would need to define $Z$ in  \cref{def:solution_RSDE} of the following form: 
\begin{equation*}
    Z_{t}\coloneqq\int_{0}^{t} Df_{s-}(Y_{s-}) \sigma_{s}(Y_{s-})\dd N_{s} + \int_{0}^{t} G(\d s, Y_{s-}).  
\end{equation*}
Here $G$ is a non-linear martingale field and the latter term is a Kunita-type non-linear stochastic integral i.e. 
\begin{equation*}
    \int_{0}^{t} G(\d s, Y_{s}) \sim \lim \sum_{i} G(t_{i+1}, Y_{t_{i}})- G(t_{i}, Y_{t_{i}}),
\end{equation*}
see \cite{kunita_1997} for more details. Although non-linear rough integrals have been developed \cite{CATELLIER20162323, NualartXia2020, Galeati2023} (in the continuous regime) we restrict to $(f, f')$ as in \cref{assumption:vector_fields} to avoid additional technicalities coming from the non-linear integration theory. 
\end{remark}

Throughout the remainder of this work, motivated by the considerations of \cref{section:Fundamentals of Rough Stochastic Analysis}, we will be concerned with the well-posedness of $L^{0}$- and good $L^{m,n}$-solutions. Given the existence of a unique $L^{m,n}$-solutions and a unique $L^{0}$-solution to \eqref{main_ref:RSDE} we will see that they necessarily need to agree and thus obtain well-posedness of good $L^{m,n}$-solutions. 

\begin{theorem}\label{theorem: RSDE solution is L^q RSM}
    Let $p\in (2,3)$, $m\in [2,\infty)$, $n\in [m,\infty]$, $\xi \in L^m(\rF_0)$, $(f, f')\in \mathscr{V}^{p}_{X}\Lip^{2}$ and $(Y,Y') \in \sV^{p,m,n}_X$ with $Y'\equiv f(Y)$ an $L^{m,n}$-solution to the rough SDE
    \begin{equation}\label{eq: mixed moment RSDE}
        Y_t = \xi+ \int_0^t b_s(Y_s) \dd s + \int_0^t \sigma_s(Y_{s-}) \dd N_s + \int_0^t (f,f')_s(Y_s) \dd \bX_s
    \end{equation}
    then 
    \begin{enumerate}
        \item if $p<m$, we have for any $p<\tp$ that 
         \[ (Y, f(Y), Df(Y)f(Y)+ f'(Y), b(Y); \sigma(Y) \bullet \d N, Df(Y) \sigma(Y) \bullet \d N)\in \sH^{\tp,m }_{\bX},\]
         with 
        \begin{equation*}
             \Big\|(Y,f(Y), \sigma(Y)\bullet \d N)\Big\|_{\sH^{ \tp, m}_{X}} \lesssim \|Y,f(Y)\|_{p,m, [0,T]};
        \end{equation*}
        \item if $n=\infty$, then again for any $p<\tp$ we have that 
        \[ (Y, f(Y), Df(Y)f(Y)+ f'(Y), b(Y); \sigma(Y) \bullet \d N, Df(Y) \sigma(Y) \bullet \d N)\in \sH^{\tp}_{\bX}.\]
    \end{enumerate}
\end{theorem}

\begin{proof}
   In the case $p<m$, \cref{prop: Doob--Meyer} yields $(Y,f(Y);\sigma(Y)\bullet \d N)\in\sH^{\tp}_X$ for any $\tp>p$, together with the corresponding moment bounds. In the case $n=\infty$ and $m\leq p$, $Y$ is càdlàg almost surely by \cite[Lemma 4.6]{AllanPieper2026}, and hence $Y'=f(Y)\in V^{\tp}$ almost surely, so \cref{prop: Doob--Meyer} again yields $(Y,f(Y);\sigma(Y)\bullet \d N)\in\sH^{\tp}_X$.
   By the stability of compositions of rough semimartingales with $(f, f') \in \mathscr{V}^{p} \Lip^{2}$ shown in \cref{lemma: composition for RSM} it also follows that $(f(Y), D f(Y)f(Y)+ f'(Y); D f(Y)\sigma(Y) \bullet \d N)$ is again a $\tilde{p}$-$X$-RSM, which thus yields the claim.
\end{proof}

\begin{remark}
    In light of \cref{prop: Doob--Meyer} one may notice that in the case where $(f,f')\in \sV^{p,m,\infty}_X\Lip^3$ as considered in \cite{AllanPieper2026} one may also show that any $L^{m,\infty}$-solution to \eqref{eq: mixed moment RSDE} indeed fulfils $(Y,f(Y); \sigma(Y)\bullet \d N) \in \sH^{\tp}_X$ for any $p<\tp$ and similarly we can also show that the solution is an element of $\sH^{\tp}_{\bX}$. However, if we want to specify the martingale component $Z$ we are faced with the difficulties already described in \cref{rem: difficulties with stochastic controlled vfs}.
\end{remark}

A straightforward consequence of \cref{theorem: RSDE solution is L^q RSM} is the following consistency result.

\begin{theorem}\label{thm:consistency_of_solutions I} 
    Let $p\in (2,3)$, $m\in [2,\infty)$, $n\in [m,\infty]$ such that $p<m$ or $n=\infty$. Let $\mathcal{Y}^{1}= (Y^{1}, f(Y^{1}), D F(Y^{1})f(Y^{1})+ f'(Y^{1}), b(Y^{1}); \sigma(Y^{1}) \bullet \d N, D f(Y^{1}) \bullet \d N)$ be a {\bf{unique}} $L^{0}$-solution and $\mathcal{Y}^{2}= (Y^{2}, f(Y^{2}), D F(Y^{2})f(Y^{2})+ f'(Y^{2}), b(Y^{2}); \sigma(Y^{2}) \bullet \d N)$ an $L^{m,n}$-solution of \eqref{main_ref:RSDE}. Then $Y^{1}$ and $Y^{2}$ are indistinguishable and thus
    \begin{equation*}
        \mathcal{Y}^{1}=\left( \mathcal{Y}^{2}, D f(Y^{1}) \sigma(Y^{1}) \bullet \d N\right)
    \end{equation*}
    up to indistinguishability. 
\end{theorem}

\begin{remark}\label{remark: restrictions of Kolmogorov consistency}
    The attentive reader will notice that in \cref{prop: Doob--Meyer}, \cref{theorem: RSDE solution is L^q RSM} and therefore also in the consistency result \cref{thm:consistency_of_solutions I} the case $n<\infty$ and $m \leq p$ is excluded. We, however, emphasise that all situations that have previously been considered in the literature, especially in \cite{FHL21, AllanPieper2026, BCN24, BuginiFrizStannat2024}, allow to either choose $n=\infty$ or $p<m$.
    Nonetheless, the restriction is far from being just a cosmetic one as general Lévy type noise only allows for sufficient moment bounds for $m\leq p$ but not for $p<m$ (see \cite[Example 2.16]{AllanPieper2026}). In the case of bounded coefficients (as considered in \cite{AllanPieper2026}) this does not cause issues as one may choose $n=\infty$; see \cite[Theorem 5.2]{AllanPieper2026}. If one however wants to consider linear coefficients one can only expect the solution to exist in $V^p L^m$ (i.e. with $n=m$ see \cite{BCN24}) and thus especially linear rough SDEs with general Lévy type noise are out of scope for \cref{thm:consistency_of_solutions I}.
\end{remark}
\subsection{The joint lift ansatz for RSDEs}
\label{subsection:The joint lift ansatz for RSDEs}Let
$\mathbf{X}=(X,\mathbb{X}) \in \mathscr{V}^{p}([0,T]; \mathbb{R}^{d_{X}})$,
$p \in (2,3)$, and $N$ an $\mathbb{R}^{N}$-valued càdlàg martingale. The idea behind the \textit{joint lift ansatz} to RSDEs is to rewrite an equation of the form 
\begin{equation*}
    \d Y_{t} = b_{t}(Y_{t}) \dd t+ \sigma_{t}(Y_{t}) \dd N_{t}+ f_{t}(Y_{t}) \dd \mathbf{X}_{t}; \quad Y_{0}=\xi
\end{equation*}
as an RDE driven by a random rough path $\mathscr{J}(\mathbf{X}; N)\in \mathscr{V}^{p}([0,T]; \mathbb{R}^{d_{X}} \oplus \mathbb{R}^{d_{N}})$ almost surely i.e.
\begin{equation*}
    \d Y_{t}= b_{t}(Y_{t}) \dd t+ (\sigma_{t} \oplus f_{t})(Y_{t}) \dd \mathscr{J}(\mathbf{X}; N)_{t}; \quad Y_{0}= \xi
\end{equation*}
and then to solve the latter $\omega$-wise. This ansatz again essentially follows the ideas developed in \cref{remark:commutation_diagramm}, where now we are looking for a fixed point of the map
\begin{equation*}
\begin{aligned}
    \Phi: \mathscr{H}_{\mathbf{X}}^{p} &\to \mathscr{H}_{\mathbf{X}}^{p},\\
    (Y, \partial_{X}Y, \partial_{X}^{2} Y, V; M, Z) &\mapsto (Y, f(Y), Df(Y)f(Y)+f'(Y), b(Y); \sigma(Y) \bullet \d N, Df(Y) \sigma(Y) \bullet \d N) .
\end{aligned}
\end{equation*}

\begin{definition}
\label{def:tensorized vf}
    Let $p\in (2, 3)$, $\mathbf{X}\in \mathscr{V}^{p}([0,T]; \mathbb{R}^{d_{X}})$ $N \in \cM$ and \\
    $(\sigma, \sigma ')\in L^{0}_{pred}(\Omega, \mathscr{V}^{p}_{\mathbf{N}}\Lip^{3}(\mathbb{R}^{d_{Y}}; \cL(\mathbb{R}^{d_{N}}; \mathbb{R}^{d_{Y}})))$, where $\mathbf{N}\coloneqq (N, \Pi(N; N))$ denotes the It\^{o} rough path lift of $N$. For $(f,f')$ a pair of vector fields which
    \begin{enumerate}
        \item fulfils \eqref{assumption:item:vector_fields_1} in \cref{assumption:vector_fields} we define  
        \begin{equation*}
                ((f \oplus \sigma), (f\oplus \sigma)')\coloneqq  \left( \left (\begin{array}{c} f \\ \sigma \end{array} \right), \left(\begin{array}{cc}
            f' & 0 \\
            0 & \sigma ' \end{array} \right)\right) \in L^{0}_{ad}\big(\Omega, \mathscr{V}^{p}_{\mathscr{J}(\mathbf{X}; N)} \operatorname{Lip}^{3}( \mathbb{R}^{d_{Y}}; \cL(\mathbb{R}^{d_{X}} \oplus \mathbb{R}^{d_{N}}; \mathbb{R}^{d_{Y}}) \big),
        \end{equation*}
        
        \item fulfils \eqref{assumption:item:vector_fields_2} in \cref{assumption:vector_fields} with $G$ the corresponding $\mathbb{R}^{d_{G}}\coloneqq \cL(\mathbb{R}^{d_{Y}} \otimes \mathbb{R}^{d_{X}} ;\mathbb{R}^{d_{Y}})$-valued martingale, we define\footnote{This is a deliberate abuse of notation, since $\Lip$-functions are bounded by convention in this work and $f\oplus \sigma$ is not bounded but essentially 'bounded + linear'. This however poses no issues for the remainder of the arguments.} 
        \begin{equation*}
        \begin{aligned}
                ((f \oplus \sigma \oplus 0), (f\oplus \sigma \oplus 0)')&\coloneqq   \left( \left (\begin{array}{c} f \\ \sigma \\ 0 \end{array} \right), \left(\begin{array}{ccc}
            f' & 0 & 0 \\
            0 & \sigma ' & 0\\
            \Id &0 &0 
            \end{array} \right)\right)\\
            &\in L^{0}_{ad}\big(\Omega, \mathscr{V}^{p}_{\mathscr{J}(\mathbf{X}; (N,G))} \operatorname{Lip}^{3}(\mathbb{R}^{d_{Y}};\mathcal{L}(\mathbb{R}^{d_{X}} \oplus \mathbb{R}^{d_{N}} \oplus \mathbb{R}^{d_{G}}; \mathbb{R}^{d_{Y}})\big).
        \end{aligned}\end{equation*}
    \end{enumerate}
\end{definition}

Well-posedness of $L^{0}$-solutions then just boils down to analysing fixed points of the map $\Phi$. Following \cref{remark:commutation_diagramm} it is equivalent to analyse fixed points of a transformed map $\tilde{\Phi}$ defined on the space of random strongly controlled rough paths. Concretely we obtain the following: 
\begin{theorem}
\label{thm:FZK_RSDE}
Suppose $(f, f')$ satisfies \cref{assumption:vector_fields}, $b$ is progressive bounded\footnote{This can easily be relaxed by non-explosion criteria for RDEs at least in the continuous regime; see \cite{RIEDEL2017283, LiYing2025}} Lipschitz and 
$(\sigma, \sigma ')\in L^{0}_{pred}(\Omega, \mathscr{V}^{p}_{\mathbf{N}}\Lip^{3}( \mathbb{R}^{d_{Y}}; \cL(\mathbb{R}^{d_{N}}; \mathbb{R}^{d_{Y}}))$, where $\mathbf{N}\coloneqq (N, \Pi(N; N))$ denotes the It\^{o} rough path lift of $N$. Then there exists a unique $L^{0}$-solution on $[0,T]$ of 
\begin{equation}\label{eq: random RDE}
    Y_{t}= \xi + \int_{0}^{t} b_{s} (Y_{s}) \dd s+\int_{0}^{t} \sigma_{s}(Y_{s-}) \dd N_{s} + \int_{0}^{t} (f_{s}, f'_{s}) (Y_{s}) \dd \mathbf{X}_{s}.
\end{equation}
In particular $Y$ is given as the almost surely defined solution to the (random) RDE 
\begin{enumerate}
    \item $ \d Y_{t}= b_{t}(Y_{t}) \dd t+ (f_{t} \oplus \sigma_{t} )(Y_{t}) \dd \mathscr{J}(\bX; N)_{t}; \quad Y_{0}= \xi$ if $(f,f')$ fulfils \eqref{assumption:item:vector_fields_1}
    \item $\d Y_{t}= b_{t}(Y_{t}) \dd t+ (f_{t} \oplus \sigma_{t} \oplus 0)(Y_{t}) \dd \mathscr{J}(\bX; (N,G))_{t}; \quad Y_{0}= \xi$ else.
\end{enumerate}
\end{theorem}

\begin{proof}We handle the cases \eqref{assumption:item:vector_fields_1} and \eqref{assumption:item:vector_fields_2} of \cref{assumption:vector_fields} separately. First, assume that $(f, f')$ satisfies \eqref{assumption:item:vector_fields_1}.
Then by standard RDE (see e.g. \cite[Theorem 3.9]{Friz_Zhang_2018}) theory there almost surely exists a solution to the corresponding (random) RDE
\begin{equation*}
    Y_{t} = \xi+ \int_{0}^{t} b_{s} (Y_{s}) \dd s + \int_{0}^{t} ((f \oplus \sigma), (f\oplus \sigma)')(Y_s) \dd \mathscr{J}(\mathbf{X}; N)_{s}
\end{equation*} i.e. there is  $(Y- \int b_{s}(Y_{s})\dd s, Y', Y'')\in L^{0}_{ad}(\Omega, \mathscr{V}^{p}_{\mathscr{J}(\mathbf{X}; N)}([0,T]; \mathbb{R}^{d_{Y}}))$ with 
\begin{equation*}
    Y'\coloneqq \left (\begin{array}{c} f(Y) \\ \sigma(Y) \end{array} \right); \quad Y''\coloneqq \left(\begin{array}{cc}
    Df(Y)f(Y)+ f'(Y)  & D\sigma (Y) f(Y) \\
    Df(Y)\sigma(Y) & D\sigma (Y) \sigma(Y) + \sigma'(Y) \end{array} \right).
\end{equation*}
By applying \eqref{prop:item:SCRP_to_RIP} of \cref{prop:scRPtoscRSM}, we obtain the existence of an $L^{0}$-solution. Conversely given an $L^{0}$-solution 
\begin{equation*}
    \mathcal{Y}= (Y, f(Y), Df(Y)f(Y)+ f'(Y); \sigma(Y) \bullet \d N , D f(Y) \sigma(Y) \bullet \d N), 
\end{equation*}
we note that almost surely
\begin{equation*}
\begin{aligned}
   \left( \left(\begin{array}{c} 0 \\ \sigma(Y)\end{array}  \right) \cdot, \left( \begin{array}{cc} 0 & D\sigma (Y) f(Y)  \\
   0&D\sigma(Y) \sigma(Y)+ \sigma'(Y) \end{array}\right) : \right)&\in \mathscr{V}^{p}_{(X, N)}([0,T]; \mathcal{L}(\mathbb{R}^{d_{X}} \oplus \mathbb{R}^{d_{N}}; \mathbb{R}^{d_{Y}}))\\
   \left( \left(\begin{array}{c} f(Y) \\ 0\end{array} \right) \cdot , \left( \begin{array}{cc} Df(Y)f(Y)+ f'(Y) &  0\\
   Df(Y) \sigma (Y)& 0  \end{array}\right) :\right) &\in \mathscr{V}^{p}_{(X, N)}([0,T]; \mathcal{L}(\mathbb{R}^{d_{X}} \oplus \mathbb{R}^{d_{N}}; \mathbb{R}^{d_{Y}})). 
\end{aligned}
\end{equation*}
Therefore we may rewrite $Y$ as 
\begin{equation*}
\begin{aligned}
    Y_{t} &= \xi + \int_{0}^{t} b_{s} (Y_{s})\dd s + \int_{0}^{t} \left(\left(\begin{array}{c} 0 \\ \sigma_{s}(Y_{s})\end{array} \right), \dots \right)\dd \mathscr{J}(\mathbf{X}; N)_{s} + \int_{0}^{t} \left( \left(\begin{array}{c} f_{s}(Y_{s}) \\ 0
    \end{array} \right), \dots \right)\dd\mathscr{J}(\mathbf{X}; N)_{s}\\
    &= \xi + \int_{0}^{t} b_{s}(Y_{s})\dd s + \int_{0}^{t} (Y'_{s}, Y''_{s}) \dd \mathscr{J}(\mathbf{X}; N)_{s},
\end{aligned}
\end{equation*}
with $Y', Y''$ as above. Thus the $L^{0}$-solution is unique, by the uniqueness of the corresponding RDE. \\
Secondly, assume that $(f, f')$ satisfy \eqref{assumption:item:vector_fields_2} i.e. $(f, f'; G) \in \mathscr{H}_{X}^{p}\cL(\mathbb{R}^{d_{Y}}; \cL(\mathbb{R}^{d_{X}};\mathbb{R}^{d_{Y}}))$ for some $\mathbb{R}^{d_{G}}$-valued martingale $G$. Analogously to the above there exists $(Y- \int b_{s}(Y_{s}) \dd s, \tilde{Y}', \tilde{Y}'')\in L^{0}_{ad}(\Omega, \mathscr{V}^{p}_{\mathscr{J}(\mathbf{X}; (N, G))}([0,T]; \mathbb{R}^{d_{Y}}))$ with
\begin{equation}\label{eq: Y', Y'' for random RDE case II}
\begin{aligned}
    \tilde{Y}'_{t}&\coloneqq \left(\left (\begin{array}{c}
    f_{t}(Y_{t}) \\
    \sigma_{t}(Y_{t})\\ 
    0
    \end{array} \right)\cdot \right)\in \mathcal{L}(\mathbb{R}^{d_{X}} \oplus \mathbb{R}^{d_{N}} \oplus \mathbb{R}^{d_{G}}; \mathbb{R}^{d_{Y}})\\
    \tilde{Y}''_{t}&\coloneqq\left(\left(\begin{array}{ccc}
    Df_{t}(Y_{t})f_{t}(Y_{t})+ f'_{t}(Y_{t})  & D\sigma_{t} (Y_{t}) f_{t}(Y_{t}) & 0\\
    Df_{t}(Y_{t})\sigma_{t}(Y_{t}) & D\sigma_{t} (Y_{t}) \sigma_{t}(Y_{t}) + \sigma'_{t}(Y_{t}) & 0 \\
    Y_{t} & 0 & 0\end{array} \right) :\right)\\
    & \quad \quad \in \mathcal{L}\left( (\mathbb{R}^{d_{X}} \oplus \mathbb{R}^{d_{N}} \oplus \mathbb{R}^{d_{G}})^{\otimes 2}; \mathbb{R}^{d{Y}} \right).\end{aligned}
\end{equation}
Using \eqref{prop:item:SCRP_to_RIP} of \cref{prop:scRPtoscRSM}, we obtain the existence of an $L^{0}$-solution of the above form. Conversely, given an $L^{0}$-solution 
\begin{equation*}
    \mathcal{Y}= (Y, f(Y), Df(Y)f(Y)+ f'(Y); \sigma(Y) \bullet \d N , D f(Y) \sigma(Y) \bullet \d N+ \d G \bullet Y), 
\end{equation*}
we note that almost surely 
\begin{equation*}
\begin{aligned}
    &\left( \left(\begin{array}{c} f(Y) \\ 0\\ 0 \end{array} \right) \cdot ,    \left( \begin{array}{ccc} Df(Y)f(Y)+ f'(Y) &  0 & 0\\
   Df(Y) \sigma (Y)& 0 & 0\\
   Y & 0 & 0 \end{array}\right) :\right) \\
   &\in \mathscr{V}^{p}_{(X, N, G)}([0,T]; \mathcal{L}(\mathbb{R}^{d_{X}} \oplus \mathbb{R}^{d_{N}} \oplus \mathbb{R}^{d_{G}}; \mathbb{R}^{d_{Y}})). 
\end{aligned}
\end{equation*}
Similarly one rewrites $\sigma(Y)$ as a controlled rough path in terms of $(X, N, G)$. Analogously as above the existence of a unique $L^{0}$ solution follows. 
\end{proof}
\begin{remark}
We emphasize that \cref{thm:FZK_RSDE} comes with the major drawback of increased regularity requirements on $\sigma$ both in space and time. Most notably we require controlledness of $\sigma$ w.r.t  the martingale driver $N$. This immediately rules out many applications e.g. to stochastic control \cite{FLZ24x, BDDF26} and mean-field games \cite{FGHT26, bayraktar2026meanfieldgamesroughcommon}. 
\end{remark}

\subsection{Well-posedness via martingale fixed points}\label{subsection: Well-posedness via Martingale fixed points}
\addtocontents{toc}
{\protect\setcounter{tocdepth}{2}}In \cref{subsection:The joint lift ansatz for RSDEs} we observed that one can construct solutions to RSDEs by solving an RDE driven by a rough path obtained as the joint lift of the rough and martingale noise. While this leads to increased regularity requirements for general $\sigma$, for trivial coefficients e.g. $\sigma=\Id$ the above ansatz is sound. 
Let us recall the following well-known example from the theory of McKean-Vlasov SDEs, to see how the findings of \cref{subsection:The joint lift ansatz for RSDEs} can be used for the general case.
\subsubsection{Inspiration from McKean-Vlasov SDEs}
\label{inspiration_McKean}
Consider the McKean--Vlasov SDE (short: MKV-SDE)
\begin{equation}
\label{eq:McKean}
Y_t
= \xi +
\int_0^t \sigma\bigl(Y_s, \mathfrak{L} (Y_s)\bigr)\, \dd B_s
\end{equation}
where $B$ is a Brownian motion, $\xi \in L^{2}(\mathfrak{F}_{0}), \sigma: \mathbb{R}^{d_{Y}}\times \mathcal{P}_{2}(\mathbb{R}^{d_{X}})\to \mathcal{L}(\mathbb{R}^{d_{B}}; \mathbb{R}^{d_{Y}})$, $\mathfrak{L}(Y_{s})$ denotes the law of ${Y}_{s}$ and $\mathcal{P}_{2}(\mathbb{R}^{d_{Y}})$ is the space of probability measures on $\mathbb{R}^{d_{Y}}$ with finite second moments. A classical approach to showing well-posedness of MKV-SDEs is by recasting it
as a fixed point problem on $\mathcal{P}_{2}$: For a given measure-curve $\mu \in C([0,T]; \mathcal{P}_2(\mathbb{R}^{d_{Y}}))$, define
\begin{equation*}
Y^\mu_t=\xi + \int_{0}^{t} \sigma(Y^{\mu}_s, \mu_{s}), \dd B_s
\end{equation*}
which are naturally well-posed under standard assumptions and 
\begin{equation*}
\begin{aligned}
\Phi: C([0,T]; \mathcal{P}_{2}(\mathbb{R}^{d_{Y}})) &\to C([0,T]; \mathcal{P}_{2}(\mathbb{R}^{d_{Y}}))\\
(\mu_{t})_{t \in [0,T]} & \mapsto (\mathfrak{L}(Y^{\mu}_{t}))_{t \in [0,T]}
\end{aligned}
\end{equation*}
Well-posedness of the MKV-SDE then amounts to having a unique fixed point to $\Phi$, in the space  $C([0,T]; \mathcal{P}_{2}(\mathbb{R}^{d_{Y}}))$. 
\begin{theorem}[Theorem 4.21 in \cite{CarmonaDelarue2018_I}]
  Assume $\sigma$ is Lipschitz on $\mathbb{R}^{d_{Y}}\times \mathcal{P}_{2}(\mathbb{R}^{d_{Y}})$ (where we equip $\mathcal{P}_{2}$ with the Wasserstein distance $\mathcal{W}_{2}$). Then $\Phi$ admits a unique fixed point for small $T$.
  It follows (by iteration) that the MKV-SDE is well-posed. 
\end{theorem}

\begin{proof}
  The argument is classical. We have for any $\mu, \tilde{\mu} \in C([0,T]; \mathcal{P}_{2}(\mathbb{R}^{d_{Y}}))$
\begin{equation*}
\begin{aligned}
\mathcal{W}_{2}(\Phi(\mu)_{t}, \Phi(\tilde{\mu})_{t})^{2}&\leq \Vert X^{\mu}_t - X^{\tilde{\mu}}_t \Vert^2_{L^{2}} =  \int_0^t \mathbb{E} (|
  \sigma (X^\mu_s, \mu_s) - \sigma (X^{\tilde{\mu}}_s, \tilde{\mu}_s) |^2) \dd
  s\\
  & \lesssim  \int_0^t \Vert X^\mu_s - X^{\tilde{\mu}}_s \Vert^2_{L^{2}} \dd s + \int_0^t \mathcal{W}_{2}(\mu_{s}, \tilde{\mu}_{s})^{2} \dd s
\end{aligned}
\end{equation*}
uniformly for $t\in [0,T]$. Using Grönwall's inequality we thus see 
\begin{equation*}
\sup_{t \in [0,T]} \mathcal{W}_{2}(\Phi(\mu)_{t}, \Phi(\tilde{\mu})_{t}) \lesssim  e^{\frac{c}{2}T} \sqrt{T} \sup_{t \in [0,T]} \mathcal{W}_{2}(\mu_{t}, \tilde{\mu}_{t}) 
\end{equation*}
which contracts for sufficiently small $T$. Existence and uniqueness follow by applying Banach's fixed point theorem on small intervals and extending the solution successively to $[0,T]$, using the independence of the above bounds from the initial condition.
\end{proof}

Equivalently, when working with the non-linear Fokker-Planck equation (FPE) corresponding to \eqref{eq:McKean} the above principle is known as the \emph{linearization} of the FPE. Analogous arguments can of course also be performed utilizing the Lions-lift of $\mu$ and then performing the fixed point argument on a suitable space of $L^{2}$-processes.

\subsubsection{Main idea:}
Motivated by the considerations of \cref{subsection:The joint lift ansatz for RSDEs} and on McKean-Vlasov-SDEs in \cref{inspiration_McKean}, we derive our main result on existence and uniqueness of RSDEs of the form \eqref{main_ref:RSDE}
with $(f, f')$ satisfying \cref{assumption:vector_fields} and $b, \sigma$ of progressive/predictable It\^{o}-regularity, in this section. \\
The central idea is to consider the RSDE
as a fixed point problem on a suitable martingale space. For a given $M \in
\mathcal{M}^2$, define
\begin{equation} \label{eq: additive rsde}
Y^{M}_{t}= \xi+ \int_{0}^{t} b_{s}(Y^{M}_{s})\dd s+M_{t}+ \int_{0}^{t} (f, f')_{s}(Y^{M}_{s})\dd \mathbf{X}_{s}
\end{equation}
as the unique $L^{0}$-solution given by \cref{thm:FZK_RSDE} ($\sigma=\Id$ here). We further define for $N\in  \mathcal{M}^{2}([0,T])$ as in \eqref{main_ref:RSDE} the map (if well-defined) 
\begin{equation}\label{eq: defn Phi}
    \Phi : \mathcal{M}^{2}([0,T])\to \mathcal{M}^{2}([0,T]); \quad M \mapsto \int_{0}^{\cdot} \sigma_{t}(Y^{M}_{t-}) \dd N_{t}.  
\end{equation}
Given a fixed point $M^{\ast}= \Phi(M^{\ast})$ one immediately verifies that $Y^{M^{\ast}}$ solves \eqref{main_ref:RSDE} in the $L^{0}$-sense of \cref{def:solution_RSDE}.\\
The question is now of course how to obtain such a unique fixed point of $\Phi$. To this end we provide the following result based on Banach's fixed point theorem.

\begin{theorem} \label{thm:fixed_point_Banach}
 Let $\xi \in L^{2}(\mathfrak{F}_{0})$, $b$ be progressive bounded Lipschitz, $(f, f')$ satisfy \cref{assumption:vector_fields} and  $\sigma = \sigma_t (\omega, x)$ be a predictable
  coefficient field, Lipschitz and of linear growth uniformly in $(t, \omega)$. Further, assume that there exists a mild control $\rho : \Delta_{[0,T]} \to \R_{+}$ such that $\|\delta \langle N\rangle_{s,t}\|_{ L^{\infty}} \leq \rho(s+,t)$ for every $(s,t) \in \Delta_{[0,T]}$. Suppose that $\Phi: \mathcal{M}^{2}\to \mathcal{M}^{2}$ is well-defined and that there is $L>0$ such that, for any $(s,t) \in \Delta_{[0,T]}$ and $M, \bar{M}\in \mathcal{M}^{2}([0,T])$ it holds that
 \begin{equation}\label{ref:Key}
    \left \Vert \sup_{v \in [s,t]}|Y^{M}_v - Y^{\bar{M}}_v| \right  \Vert_{L^{2}} \leq L \left \Vert \sup_{v \in [s,t]} |M_v - \bar{M}_v| \right \Vert_{L^{2}}.
 \end{equation}
 Then, there is a unique $M^{\ast} \in \mathcal{M}^{2}$ with $\Phi(M^{\ast})=M^{\ast}$. In particular, there exists a unique $L^{0}$-solution of \eqref{main_ref:RSDE} on $[0,T]$ with martingale part $M^{\ast} \in \cM^{2}$. 
\end{theorem}

\begin{proof}
For brevity we will set $b\equiv 0$. Assume further, that we have a deterministic partition $\cP=\{t_k\}_{k=0}^N$ and $\xi_k \in L^2(\rF_{t_k})$,  for any $k\in \{0, \ldots, N\}$ given, which we specify later in the proof.

For $k \in \{0,\ldots, N-1\}$ we further define the maps 
\begin{equation*}
    \bar{\Phi}^{k}: \cM^2([t_k,t_{k+1}]) \to \cM^2([t_k,t_{k+1}]); \quad M \mapsto \int_{t_k}^{\cdot} \sigma_r(Y^{M,k}_{r-}) \dd \bar{N}^{k}_r,
\end{equation*}
where we denote by $Y^{M, k}$ the solution to 
\begin{equation*}
    Y^{M,k} = \xi^k + \int_{t_k}^{\cdot} (f,f')_r(Y^{M,k}_r) \dd \bX_r + M
\end{equation*}
for any $M\in \cM^2([t_k,t_{k+1}])$ on $[t_k,t_{k+1}]$, and $\bar{N}^{k} := \1_{\{t_{k+1}\}} N_{t_{k+1}-} + N \1_{[t_k,t_{k+1})} \in \cM^{2}([t_k,t_{k+1}])$. Especially, $\Delta \bar{N}^{k}_{t_{k+1}}=0$ and thus also $\Delta \langle \bar{N}^{k}\rangle_{t_{k+1}}=0$. By the It\^{o}-isometry, Doob's $L^2$-inequality, the Lipschitz property of $\sigma$ and \eqref{ref:Key}  we find for any $k\in \{0,\ldots, N-1\}$ and any $M, \bar{M} \in \cM^2([t_k,t_{k+1}])$, that
  \begin{equation*}
  \begin{split}
   \Big\| \sup_{v\in [t_k,t_{k+1}]} |\delta \bar{\Phi}^{k} (M)_{t_k, v}& - \delta \bar{\Phi}^{k} (\bar{M})_{t_k, v}| \Big\|^2_{L^{2}} \\
   &\leq 4  \E \bigg[ \int_{t_k}^{t_{k+1}} | \sigma_r (Y_{r-}^{M,k}) - \sigma_r (Y_{r-}^{\bar{M},k})|^{2} \dd \langle \bar{N}^{k} \rangle_{r} \bigg ] \\
    & \leq 4\|\sigma\|^2_{\Lip} \|\delta \langle \bar{N}^{k} \rangle_{t_k,t_{k+1}}\|_{L^{\infty}} \Big\| \sup_{r\in [t_k,t_{k+1}]} |Y^{M,k}_r- Y^{\bar{M},k}_r|  \Big\|_{L^{2}}^{2}\\
    &\leq 4\|\sigma\|^2_{\Lip}L^2 \rho(t_k+,t_{k+1}-) \Big\|M-\bar{M}  \Big\|_{\cM^2([t_k,t_{k+1}])}^{2}.
\end{split}
\end{equation*}

By \cite[Lemma 1.5]{Friz_Zhang_2018} we know that we may choose the partition $\cP := \{t_k\}_{k=0}^{N}$ of $[0,T]$ to be such that 
\begin{equation*}
    \rho(t_k+, t_{k+1}-) < \frac{1}{1+4\|\sigma\|^2_{\Lip}L^2}.
\end{equation*}

Thus, we especially find that
  \begin{equation*}
    \Vert \bar{\Phi}^{k}(M)- \bar{\Phi}^{k}(\bar{M}) \Vert_{\mathcal{M}^{2}([t_k,t_{k+1}])} \leq  C \|M-\bar{M}  \|_{\cM^2([t_k,t_{k+1}])}
  \end{equation*}
  for some $0<C<1$ and, thus, as $\cM^{2}([t_k,t_{k+1}])$ is a Banach space, there is a unique fixed point $\bar{\Phi}^k(M^{\ast,k})=M^{\ast,k} \in \mathcal{M}^{2}([t_k,t_{k+1}])$ by Banach's fixed point theorem. For brevity we will write $Y^{M^{\ast,k}}\coloneqq Y^{M^{\ast,k},k}$ for every $k\in \{0, \ldots, N-1\}$.
  Set $\xi_0\coloneqq \xi$ and iteratively for $k\in \{1, \ldots, N\}$ define
  \begin{equation}
  \label{eq:initial condition iteration}
      \xi_{k}\coloneqq Y^{M^{\ast, k-1},}_{t_{k}}+ \sigma_{t_{k}}(Y^{M^{\ast, k-1}}_{t_{k}-}) \Delta N_{t_k}.
  \end{equation}
  Then, we may define $M^{\ast} \in \cM^{2}([0,T])$, by 
  \begin{equation*}
      M^{\ast}_t \coloneqq \sum_{k=0}^{N-1} \int_{t_k \wedge t}^{t_{k+1}\wedge t} \sigma_r(Y^{M^{\ast,k}}_{r-}) \dd N_r = \sum_{k=0}^{N-1} \Big(\delta M^{*,k}_{t_{k}\wedge t, t_{k+1}\wedge t}+\sigma_{t_{k+1}}(Y^{M^{\ast, k}}_{t_{k+1}-}) \Delta N_{t_{k+1}}\1_{[t_{k+1},T]}(t)\Big)
  \end{equation*}
  for any $t\in [0,T]$. Thus, by induction on $k$ and a standard iterative argument, we have for any $t\in [0,T]$, that 
  \begin{equation}\label{eq: Y^M^* is the sum of small interval solutions plus jump}
      Y^{M^{\ast}}_t = \sum_{k=0}^{N-1} Y^{M^{\ast,k}}_t \1_{[t_k,t_{k+1})}(t) + \xi_N \1_{\{T\}}.
  \end{equation}
 Thus, we verify that
\begin{align*}
    \Phi(M^{\ast}) &=\sum_{k=0}^{N-1} \int_{t_k \wedge \cdot}^{t_{k+1} \wedge \cdot} \sigma_r(Y^{M^{\ast}}_{r-}) \dd N_r =\sum_{k=0}^{N-1} \int_{t_k \wedge \cdot}^{t_{k+1}\wedge \cdot} \sigma_r(Y^{M^{\ast,k}}_{r-}) \dd N_r = M^{\ast}.
\end{align*}
The uniqueness of $M^{\ast}$ follows inductively from the
uniqueness of the local fixed points and from the jump structure of It\^{o} integrals, observing that for any fixed point $\tM$ of $\Phi$ on $[0,T]$ we have that $(\delta\tM_{t_k,t}-\Delta \tM_{t_{k+1}}
\1_{\{t_{k+1}\}}(t))_{t\in[t_k,t_{k+1}]}$
is a fixed point of $\bar{\Phi}^k$ on $[t_{k}, t_{k+1}]$.
 By \cref{thm:FZK_RSDE} one verifies that for $Z^{M^{\ast}}\in \{ Df(Y^{M^{\ast}}) \sigma(Y^{M^{\ast}}) \bullet \d N, Df(Y^{M^{\ast}}) \sigma(Y^{M^{\ast}}) \bullet \d N+ \dd G \bullet Y^{M^{\ast}}\} $ chosen according to \cref{def:solution_RSDE} and 
 \begin{equation*}
  \mathcal{Y}^{M^{\ast}}\coloneqq (Y^{M^{\ast}}, f(Y^{M^{\ast}}), Df(Y^{M^{\ast}})f(Y^{M^\ast})+ f'(Y^{M^{\ast}}), b(Y^{M^\ast}); \Phi(M^{\ast}), Z^{M^{\ast}})
  \end{equation*} 
  it holds $\mathcal{Y}^{M^{\ast}}\in \mathscr{H}_{\mathbf{X}}^{p}([0,T], \mathbb{R}^{d_{Y}})$
  and thus $Y^{M^{\ast}}$ solves the RSDE \eqref{main_ref:RSDE} on $[0,T]$ in $L^{0}$-sense.
\end{proof}

\begin{example}
    Let $N$ be an $\mathbb{R}^{d}$-valued Lévy-martingale with finite second moments (i.e. a martingale with finite second moments, which is also a Lévy process) starting at $0$ and denote its Lévy-Khintchine triplet by $(b, \Sigma, \nu)$. Further, assume that it generates the underlying filtration $(\mathfrak{F}_{t})_{t \in [0,T]}$. Using the Lévy-Khintchine formula we calculate its second moment to be 
    \begin{equation*}
    \mathbb{E}[|N_{t}|^{2}]= t \left(|b|^{2}+ \operatorname{tr}(\Sigma) + \int_{\mathbb{R}^{d}}|x|^{2}\nu (\dd x)\right)\eqqcolon tC. 
    \end{equation*}
    Using the martingale property, independence and stationarity of increments we see that
    \begin{equation*}
        \mathbb{E}_{s}[|N_{t}|^{2}]= |N_{s}|^{2}+ \mathbb{E}_{s}[|N_{t}-N_{s}|^{2}]=|N_{s}|^{2}+ \mathbb{E}[|N_{t-s}|^{2}] = |N_{s}|^{2}+ C(t-s).
    \end{equation*}
    Thus, we see that the predictable covariation bracket is given by $\langle N \rangle_{t}= Ct$
    and that $N$ clearly satisfies the conditions of \cref{thm:fixed_point_Banach}. 
\end{example}
\begin{remark}
    We emphasize that the entire argument in \cref{thm:fixed_point_Banach} hinges on the stability estimate \eqref{ref:Key}. In particular we note that 
    \begin{enumerate}
        \item Given \eqref{ref:Key} we don't impose any further assumptions on $\sigma$ besides uniform Lipschitzness. In particular $\sigma$ is a-priori not assumed to be bounded.
        \item One can easily adjust the above proof to any closed subset of $\mathcal{M}^{2}([0,T])$ as long as it is invariant under $\Phi$; see \cref{corollary: the fixed point argument is applicable for mixed moment assumptions} for an application of this idea. 
        \item One can always also add a Lipschitz drift of linear growth by considering a corresponding fixed point argument on the semimartingale space $\mathscr{H}^{2}$, see \cite[Proposition A.4]{friz2026roughbackwardstochasticdifferential}. Therefore we set $b \equiv 0$ from here on.
        \item One can also show the existence of a fixed point of $\Phi$ using Schauder's fixed point theorem. On the requirements' side this has the benefit of only needing a-priori bounds on the solution to \cref{eq: additive rsde} instead of the stability property \eqref{ref:Key}. However the uniqueness then needs to be obtained in another way for example through a comparison theorem \cite[Theorem 4.3]{friz2026roughbackwardstochasticdifferential} or suitable smoothness assumptions on $\sigma$. It is however unclear how to obtain uniqueness in general.
        \item Under additional integrability assumptions on the solutions $Y^M$, namely that $Y^M\in\sH^{p,1}_X$ exists for every $M\in\cM^2$,\footnote{Or for every $M$ in a suitable subspace of $\cM^2$. Such an assumption essentially reduces to sufficient a priori bounds for RDEs.} and assuming stability with respect to $M$ in $V^pL^1$ in addition to the uniform-norm stability required above, the proof of the previous theorem can be adapted to weaken the assumptions on $N$. More precisely, it suffices to assume the existence of a mild control $\rho:\Delta_{[0,T]}\to\R_+$ such that $\|\E_s[\delta\langle N\rangle_{s,t}]\|_{L^\infty}\leq\rho(s+,t)$ for every $(s,t)\in\Delta_{[0,T]}$.
        This follows from $\langle N\rangle\in V^1 L^1$ and arguments analogous to those in the proof of \cite[Lemma B.3]{AllanPieper2026}. The arguments in the preceding theorem could therefore be adapted to cover the cases considered in \cite{AllanPieper2026} and further examples. In particular, the weakened bracket assumption is satisfied by every VMO martingale.
    \end{enumerate}
\end{remark}

We emphasize that the solutions constructed in \cref{thm:fixed_point_Banach} are a priori different from those in \cite{FHL21, AllanPieper2026} and especially are directly constructed in the space of rough semimartingales. However, almost sure uniqueness of solutions to RDEs from standard RDE theory (see \cite[Theorem 3.9]{Friz_Zhang_2018} and \cite[Theorem 4.1]{AllanPieper2026}) together with \cref{thm:consistency_of_solutions I} imply that every $L^{m,n}$-solution to \eqref{eq: additive rsde} coincides with the $L^0$-solution, whenever $p<m$ or $n=\infty$ and thus also every $L^{m,n}$-solution to \eqref{main_ref:RSDE} agrees with the solution found through \cref{thm:fixed_point_Banach} in these cases. As noted in \cref{remark: restrictions of Kolmogorov consistency}, this still excludes important cases, including semi-linear RSDEs with general Lévy noise. The next result recovers good $L^m$-solutions in certain cases with $m\leq p$, assuming suitable stability and applicability of \cref{thm:fixed_point_Banach}.

\begin{theorem}\label{thm:consistency_of_solutions II}
     Let $p\in (2,3)$, $m\in [2,\infty)$ and $n\in [m,\infty]$ as well as $\mathbf{X}\in \mathscr{V}^{p}$, $N\in \cM^2$, $(f, f')$, $b$ and $\sigma$ such that the assumptions of \cref{thm:fixed_point_Banach} are fulfilled.
     Then, if additionally for every $M\in \cM^{2}$ there exists an $L^2$-solution\footnote{Or for any $M$ in an appropriate closed subset which is invariant under $\Phi$.} to \eqref{eq: additive rsde} as well a sequence of partitions $\{\cP^{n,M}\}_{n\in \N}$ of $[0,T]$ such that there exists a both $L^0$- as well as $L^2$-admissible sequence $((f^{n,M}, (f^{n,M})'))_{n\in \N}$ i.e. a sequence of pairs of vector fields such that both the $L^2$-solution as well as the $L^0$-solution\footnote{Which exists due to \cref{thm:FZK_RSDE}.} are weakly stable along $(\cP^{n,M};(f^{n,M}, (f^{n,M})'))_{n \in \N}$ (as defined in \cref{def: A Davie-type solution}) then the two solutions coincide up to indistinguishability.
\end{theorem}

\begin{proof}
    For simplicity we assume $b\equiv 0$ here but the following can be easily extended to the fully general case. By \cref{thm:fixed_point_Banach} we know that the unique $L^0$-solution to \eqref{main_ref:RSDE} is given by
    \begin{equation}\label{eq: fixed point additive RSDE}
        \d Y_{t} = \dd M^{*}_{t} + (f, f')_{t} (Y_{t}) \dd \mathbf{X}_{t};\quad Y_{0}=\xi \in L^{0}(\mathfrak{F}_{0}).
    \end{equation}
    where $M^{*} \in \cM^2$ is the unique fixed point of $\Phi$. However, by assumption there exists also an $L^2$-solution to \eqref{eq: fixed point additive RSDE} as well as a sequence of partition $\{\cP^{n, M^{*}}\}_{n\in \N}$ as well as a simultaneously $L^0$- and $L^2$-admissible sequence of vector fields $(f^{n, M^{*}}, (f^{n,M^{*}})')_{n\in \N}$. Thus, we may apply \cref{lemma: discrete consistency result for random RDEs} with $(\sigma, \sigma')= (\Id, 0) \in \mathscr{V}^{p}_{M}\mathcal{L}( \mathbb{R}^{d_{Y}}; \mathbb{R}^{d_{Y}})$ and thus indeed the unique $L^0$-solution coincides with the $L^2$-solution to \eqref{eq: fixed point additive RSDE}.
\end{proof}

\section{On semi-linear RSDEs}
\label{section:On semi-linear RSDEs}
In this section we investigate the well-posedness of semi-linear RSDEs of the form 
\begin{equation}
\label{section:ref:semilin_RSDE}
    Y_{t} = \xi + \int_{0}^{t} \sigma_{s}(Y_{s-}) \dd N_{s} + \int_{0}^{t} (A, A')_{s} Y_{s} \dd \bX_{s},
\end{equation}
where $A$ is linear and in particular satisfies one of the following:
\begin{enumerate}[label=\Roman*.]
    \item\label{section:item:semilin_start1} $(A, A') \in \mathscr{V}^{p}_{X}\cL(\mathbb{R}^{d_{Y}};  \cL(\mathbb{R}^d_{X}; \mathbb{R}^{d_{Y}}))$ as introduced in \cref{def:rough_linear_vf} (in particular $(A, A')$ is deterministic).
    \item\label{section:item:semilin_start2} $(A, A')\in \mathscr{H}^{p}_{X}\cL(\mathbb{R}^{d_{Y}}; \cL(\mathbb{R}^{d_{X}}; \mathbb{R}^{d_{Y}}))$ as introduced in \cref{def:rough_stochastic_vf}. 
\end{enumerate}
Following \cref{Section:main_theorem} this is done by deriving stability estimates of the form \eqref{ref:Key} for the \textit{linear RSDE} given by
\begin{equation}
\label{section:ref:linear_RSDE}
    Y^{M}_{t}= \xi + M_{t} +\int_{0}^{t} (A, A')_{s} Y^{M}_{s} \dd \bX_{s.} 
\end{equation}
In both cases, this is done by deriving analogous bounds to \cite[Theorem 3.11]{BCN24} using the rough stochastic Grönwall lemma \cref{lemma: gronwall type result p var}. 
In the case \eqref{section:item:semilin_start1} we additionally derive a suitable \textit{Duhamel's principle} for such equations \eqref{section:ref:linear_RSDE}, which provides an explicit representation for the solution. 

\subsection{Well-posedness \& stability of semi-linear RSDEs}
\label{subsection:Well-posedness and stability of semi-linear RSDEs}
As previously discussed in \cref{Section:main_theorem} establishing the Lipschitz bound \eqref{ref:Key} is far from trivial. The following result establishes the existence, uniqueness and Lipschitz bound \eqref{ref:Key} for the mixed moment case (already contained in the continuous setting in \cite[Theorem 3.11]{BCN24}) by an application of the rough stochastic Grönwall Lemma \cref{lemma: gronwall type result p var}. 
\begin{proposition}\label{proposition: exsitence of additive linear RSDEs with random coeff}
    Let $p \in [2,3)$, $m \in [2,\infty)$, $\bX \in \mathscr{V}^{p}([0,T]; \mathbb{R}^{d_{X}})$ and $M\in \cM^{p,m}$ as well as $(A,A') \in \sV^{p,m,\infty}_{X} \cL(\mathbb{R}^{d_{Y}}; \cL(\R^{d_{X}}, \mathbb{R}^{d_{Y}}))$.
   Then, there exists a unique $L^m$-solution to the equation \eqref{section:ref:linear_RSDE}.
   Moreover, if we have two martingales $M,\tM \in \cM^{p, m}$ and $K>0$ is such that 
    \begin{equation*}
        \|\xi \|_{L^m}+\|A,A'\|_{p,m,\infty, [0,T]}+ \|\bX\|_{p,[0,T]}  \leq K
    \end{equation*}
    then 
    \begin{equation}\label{eq: Lipschitz bound for linear RSDE}
    \begin{split}
        &\Big\| \sup_{t\in[0,T]} |Y^M_t-Y^{\tM}_t| \Big\|_{L^m} +\|Y^M- Y^{\tM}\|_{p,m,[0,T]}+\|\E_\cdot R^{Y^M-Y^{\tM}, (Y^M-Y^{\tM})'}\|_{\frac{p}{2},m,[0,T]}\\
          & \lesssim_K   \|M - \tM\|_{p, m, [0,T]}.
    \end{split} 
    \end{equation}
\end{proposition}

\begin{proof}
    The existence and uniqueness can be shown by a standard fixed point argument akin to the case of bounded coefficients (see e.g. the proof of \cite[Theorem 5.2]{AllanPieper2026}) under the additional usage of the a-priori estimates provided in \cref{proposition: a priori estimates additive linear RSDE} and is thus omitted here. For the convenience of the reader we, however, show how \eqref{eq: Lipschitz bound for linear RSDE} can be easily verified with the help of \cref{lemma: gronwall type result p var}. 
    More precisely, by \cite[Lemma 4.6 \& Lemma 4.9]{AllanPieper2026} and analogous bounds as the ones found in \cite[Proposition 3.3]{BCN24} we have that 
    \begin{align*}
         &\Big\| \sup_{v\in[s,t]} |\delta Y^{M}_{s,v}-\delta Y^{\tM}_{s,v}| \Big\|_{L^m} +\|Y^{M}- Y^{\tM}\|_{p , m ,[s,t]}+\|\E_\cdot (R^{Y^M, (Y^M)'}-R^{Y^{\tM, (Y^{\tM})'})})\|_{\frac{p}{2},m,[s,t]}\\
         &\lesssim \Big(\|A_s (Y^{M} _s- Y^{\tM}_s)\|_{L^{m}}+ \|A (Y^{M}- Y^{\tM})\|_{p ,m, [s,t)} +\|A' (Y^M-Y^{\tM})\|_{p, m ,[s,t)} \\ 
         &\qquad+ \|\E_{\edot}(R^{A(Y^M-Y^{\tM}}) \|_{\frac{p}{2},m ,[s,t)} +\|A^{\otimes 2} (Y^M-Y^{\tM})\|_{\tp, m ,[s,t)}+ \| A'_s (Y^M_s-Y^{\tM}_s)\|_{L^m} \\
         &\qquad +\|A^{\otimes 2}_s  (Y^M_s-Y^{\tM}_s)\|_{L^m}\Big) \|\bX\|_{p,[s,t]} + \|M- \tM\|_{p,m,[s,t]}\\
         &\lesssim \Big(\Big\| \sup_{v\in[s,t)} | Y^{M}_{v}- Y^{\tM}_{v}| \Big\|_{L^m} +\|Y^{M}- Y^{\tM}\|_{p , m ,[s,t)}+\|\E_\cdot (R^{Y^M, (Y^M)'}-R^{Y^{\tM, (Y^{\tM})'})})\|_{\frac{p}{2},m,[s,t)} \Big)\\
         &\quad \times (\|\bX\|_{p,[s,t]}+ \|A\|_{p,m, \infty,[s,t]})  + \|M - \tM\|_{p, m, [0,T]}
    \end{align*} 
    Therefore, using that also $\|A(Y^m-Y^{\tM})\|_{p,m,[s,t]} \lesssim_{K} \|Y^M-Y^{\tM}\|_{p, m, [s,t]}+ \|A\|_{p,m,\infty, [s,t]} \|Y^M_s- \tY^M_s\|_{L^m}$ we may indeed apply the rough stochastic Grönwall Lemma \cref{lemma: gronwall type result p var} with $R^1= A(Y^M-Y^{\tM})$, $R= R^{(Y^M-Y^{\tM})}$, $w_1(s,t)=\|\bX\|^p_{p,[s,t]}+\|A\|^p_{p,m, \infty,[s,t]}$ as well as $\eta  =\|M-\tM\|^p_{p,m,[0,T]}$, where we omit the multiplicative constants in $\eta$, which indeed implies \eqref{eq: Lipschitz bound for linear RSDE}.
\end{proof}

\begin{remark}
    The attentive reader will notice that \eqref{eq: Lipschitz bound for linear RSDE} is not the same as \eqref{ref:Key} for $m=2$. However, leveraging the fact that $K>0$ in \cref{proposition: exsitence of additive linear RSDEs with random coeff} was chosen independently of $M$ and $\tM$ as well as the fact that $\|M\|_{p,2,[s,t]} \lesssim \||M|^{*}_T\|_{L^2}$ for any martingale $M$ as seen in \cite[Remark 2.14]{AllanPieper2026} one can easily deduce \eqref{ref:Key} from \eqref{eq: Lipschitz bound for linear RSDE}. See also \cref{corollary: the fixed point argument is applicable for mixed moment assumptions} for the details on the analogous results in the case of bounded coefficients.
\end{remark}
We emphasize that \cref{proposition: exsitence of additive linear RSDEs with random coeff} provides stability estimates for the mixed-moment solution to \eqref{section:ref:linear_RSDE}. To be able to apply this in the context of \cref{thm:fixed_point_Banach} we must first also verify that the mixed-moment and $L^{0}$-solution of \eqref{section:ref:linear_RSDE} agree using \cref{corollary: consistency for linear RSM with deterministic jumps}. In total we may infer the existence of good $L^2$-solutions to \eqref{section:ref:semilin_RSDE} whenever $(A-G, A')$ only jumps at deterministic times:

\begin{corollary}\label{cor:good-semilinear-deterministic-jumps}
    Let $p\in (2,3)$, $\mathbf{X}\in \mathscr{V}^{p}$, $N \in \cM^{p,2, \infty}$ which additionally fulfills the assumptions of \cref{thm:fixed_point_Banach} as well as $\sigma =\sigma_t(\omega,x)$ a predictable coefficient field, Lipschitz and of linear growth uniformly in $(t,\omega)$. Further, assume that $(f, f') \in \sV^{p,2,\infty}_{X} \cL$ fulfils additionally \eqref{assumption:item:vector_fields_2} in \cref{assumption:vector_fields} such that $\{t\in [0,T] \ : \ |\Delta (f-G)_t| +|\Delta f'_t| \}$ is a set of deterministic times.
    Then, there exists a unique good $L^2$-solution to \eqref{main_ref:RSDE}.
\end{corollary}
Unfortunately, this does not allow for the composition of a sufficiently smooth functional taking values in the linear maps (think $D f$ for $f\in C^3_b$) with the solution of an RSDE with deterministic bounded coefficients as wanted for the study of derivatives of solution-flows and the Malliavin differentiability of solutions of RSDEs (see \cite{BCN24, BuginiFrizStannat2024}).
Specifically such concatenation will in general not factor into a controlled part with deterministic jump times and a martingale (as e.g. seen in the Itô formula in \cref{prop:rsIto}). However, using, that the latter fulfils well-known stability estimates (see e.g. \cref{theorem: existence and estimates for solutions to RSDEs}) and thus is \textit{weakly stable} (see \cref{def: A Davie-type solution} for details) along any appropriate choice of $(\cP^n, (f^{\cP^n}, (f')^{\cP^n}))_{n\in \N}$ by \cref{lemma: randomised RDE is Davie} we can, nonetheless show the consistency due to the assumed regularity of the functional:

\begin{proposition}\label{proposition: good solutions for Jacobians of flows}
    Let $p\in (2,3)$, $m\in [2,\infty)$, $N\in V^pL^{m,\infty}$ and $b, \sigma$ be, respectively, progressive and predictable bounded Lipschitz functions and $(f,f')\in \sV^p\Lip^3$ Further, let $M\in \cM^m$ and 
    \begin{equation*}
        (\Psi,\Psi') \in  \sV^p_X \Lip^{2}(\mathbb{R}^{d_{Y}}; \mathcal{L}(\mathbb{R}^{d_{Y}} \otimes \mathbb{R}^{d_{X}}; \mathbb{R}^{d_{Y}})).
    \end{equation*}
    Assume that $\cY$ is the unique $L^{m,\infty}$-solution\footnote{provided by \cref{theorem: existence and estimates for solutions to RSDEs}} to 
    \begin{equation}\label{eq: RSDE with det. coeff.}
        Y_{t}=\xi + \int_{0}^{t} b_s(Y_{s}) \dd s+\int_{0}^{t} \sigma_s(Y_{s-}) \dd N_{s} + \int_{0}^{t}(f,f')_s(Y_{s}) \dd \bX_{s}
    \end{equation}
    then there exists a unique good $L^m$-solution to 
    \begin{equation}\label{eq: Psi RSDE}
        \tY_{t}= \tilde{\xi} + M_t + \int_{0}^{t} (\Psi(Y), D \Psi(Y) f(Y) +\Psi'(Y))_s \tY_s \dd \bX_{s}.
    \end{equation}
\end{proposition}
Once again due to the technical nature of this result we postpone the proof and provide it instead in \cref{appendix: A discrete approximation approach to consistency} for the convenience of the reader. The prime example of such vector fields is $(\Psi, \Psi')=(Df, (Df)')$ for $(f, f')\in \mathscr{V}^{p}_{X}\Lip^{3}$ and $(Df)'\coloneqq (Df')^{\top}$.

\begin{corollary} 
Let $p \in (2, 3), m \in [2, \infty),  \xi,\tilde{\xi}\in L^{m}(\mathfrak{F}_{0})$, $\bX = (X, \X)\in \sV^p$ and $N \in V^{p}L^{m, \infty}$ with $\Vert \delta \langle N \rangle_{s,t} \Vert_{L^{\infty}}\leq \rho(s+, t)$ for any $(s,t)\in \Delta_{[0,T]}$ and some mild control $\rho$. Let $b$ progressive bounded Lipschitz and $\sigma$ predictable with finite $C^{1}_{b}$-norm uniformly in $(t, \omega)$ as well as $(f, f')\in \mathscr{V}^{p}_{X}\Lip^{3}(\mathbb{R}^{d_{Y}}; \mathcal{L}(\mathbb{R}^{d_{X}}; \mathbb{R}^{d_{Y}}))$. Let $Y$ be the unique $L^{m, \infty}$-solution of \eqref{eq: RSDE with det. coeff.}. Then there exists a unique good $L^{m}$-solution $\tY$ of 
\begin{equation}
\label{eq:formalJacobianequation}
    \tilde{Y}_{t} = \tilde{\xi}+ \int_{0}^{t} D\sigma_{s}(Y_{s-})\tilde{Y}_{s-} \dd N_{s} + \int_{0}^{t} (Df(Y), D^{2} f(Y) f(Y)+ (Df(Y))')_{s} \tilde{Y}_{s} \dd \mathbf{X}_{s}.
\end{equation}
\end{corollary}

\begin{proof}
    By an analogous calculation to \cite[Lemma 3.6]{BCN24}, we have that 
    $(\Psi(Y), D\Psi(Y) f(Y)+ \Psi'(Y))\in \mathscr{V}_{X}^{p, m, \infty}\mathcal{L}(\mathbb{R}^{d_{Y}}; \cL(\mathbb{R}^{d_{X}}; \mathbb{R}^{d_{Y}}))$.
    Thus, by \cref{proposition: good solutions for Jacobians of flows}, we find for $(\Psi, \Psi')\coloneqq (Df, (Df)')$ that \cref{eq: Psi RSDE} has a unique good $L^{m}$-solution, which we denote by $\tilde{Y}^{M}$. Therefore, by \cref{proposition: good solutions for Jacobians of flows} we know that $\tilde{Y}^{M}$ satisfies the stability estimate \eqref{ref:Key}. Further, by It\^{o} isometry and \cref{proposition: a priori estimates additive linear RSDE}, we have that
    \begin{equation*}
    \Phi: \mathcal{M}^{2}\to \mathcal{M}^{2}; \quad  M \mapsto \int_{0}^{\cdot} D\sigma_{s}(Y_{s-}) \tilde{Y}^{M}_{s-} \dd N_{s}
    \end{equation*}
     is well defined and thus by \cref{thm:fixed_point_Banach} there exists a unique $L^{0}$-solution to \eqref{eq:formalJacobianequation}. Applying \cref{proposition: good solutions for Jacobians of flows} with the martingale fixed point we see that the solution is a good $L^{m}$-solution. 
\end{proof}

\begin{remark}
    For $\xi= y\in \mathbb{R}^{d_{Y}}$ deterministic initial condition, denote by $Y^{y}$ the solution to \eqref{eq: RSDE with det. coeff.} for $b\equiv 0$. Then, whenever it exists, the Jacobian $D_{y}Y^{y}$ satisfies the RSDE
    \begin{equation*}
        D_{y}Y^{y}_{t}= \Id + \int_{0}^{t} D\sigma_{s}(Y_{s-})D_{y}Y^{y}_{s-} \dd N_{s} + \int_{0}^{t} (Df(Y), D^{2} f(Y) f(Y)+ (Df(Y))')_{s} D_{y} Y_{s} \dd \mathbf{X}_{s}, 
    \end{equation*}
    which is of the form \eqref{eq:formalJacobianequation}. However we emphasize that in general the well-posedness of this RSDE does not imply the differentiability of $y \mapsto Y^{y}$. Similar RSDEs of the form \eqref{eq:formalJacobianequation} are also crucial for Hörmander theory of RSDEs; see \cite[Section 5]{BCN24}. 
\end{remark}

\subsection{A càdlàg rough stochastic Duhamel's principle}\label{subsection: Duhamel formula}
In the following we investigate linear RSDEs of the form 
\begin{equation}
\label{eq:inhomogeneous_linear_RSDE}
    Y_{t} = \xi+   M_{t}+ \int_{0}^{t} (A, A')_{s} Y_{s}\dd \mathbf{X}_{s}
\end{equation}
with deterministic linear vector fields $(A, A')\in \mathscr{V}^{p}_{X}\cL(\mathbb{R}^{d_{Y}}; \cL(\mathbb{R}^{d_{X}}; \mathbb{R}^{d_{Y}}))$ and $M\in \mathcal{M}^{2}$. In particular by viewing \eqref{eq:inhomogeneous_linear_RSDE} as a linear RDE  subject to some inhomogeneous random forcing $M$, we want to derive some variational formulation or \textit{Duhamel's principle} for \eqref{eq:inhomogeneous_linear_RSDE}. To this end let us introduce the corresponding RDE semi-flow
\begin{equation}
\label{eq:RDE_flow_Duhamel}
    \phi_{s, t} \coloneqq \Id + \int_{s}^{t} (A_{r} \phi_{s,r}, (A'_{r}+ A^{2}_{r}) \phi_{s,r} )\dd \mathbf{X}_{r} \in \mathcal{L}(\mathbb{R}^{d_{Y}}; \mathbb{R}^{d_{Y}}). 
\end{equation}
In the continuous regime the derivation of a variational formula to \eqref{eq:inhomogeneous_linear_RSDE} is immediate.  
\begin{example}
\label{example:smooth_solution_Duhamel}
    Suppose $\mathbf{X}$, $M$ and $(A, A')$ are as above and continuous in $t$. Then by \cref{lemma:finite_jumps_flow_inversion}, $\phi$ as in \eqref{eq:RDE_flow_Duhamel} is a flow of homeomorphisms i.e for any $(s,t) \in \Delta_{[0,T]}$, $\phi_{s,t} \in \mathbb{R}^{d_{Y}\times d_{Y}}$ is invertible. Then for the process 
    \begin{equation*}
        Y_{t}\coloneqq \phi_{0,t} \xi + \int_{0}^{t} \phi_{r, t}\dd M_{r}= \phi_{0, t} \left( \xi + \int_{0}^{t} \phi_{0,s}^{-1}\dd M_{s}\right)\eqqcolon \phi_{0,t} \tilde{Y}_{t}
    \end{equation*}
    it follows by applying e.g. the rough stochastic  It\^{o}-Wentzell formula \cite[Theorem 4.16]{DauseFrizJentzenSong2026} with respect to the functional $F_{t}(y)\coloneqq \phi_{0,t}y$\footnote{Or rather the canonical controlled field generated by the functional; see \cite[Example 3.16]{DauseFrizJentzenSong2026}.} that 
    \begin{equation*}
        \delta Y_{s,t} = \int_{s}^{t} (F'_{r}(\tilde{Y}_{r}), F''_{r}(\tilde{Y}_{r}))\dd \mathbf{X}_{r}+ \int_{s}^{t} D F_{r}(\tilde{Y}_{r}) \dd \tilde{Y}_{r}= \int_{s}^{t} (A_{r} Y_{r}, (A'_{r} + A^{2}_{r} ) Y_{r})\dd \mathbf{X}_{r} + \delta M_{s,t}. 
    \end{equation*}
\end{example}
Following the exposition in \cref{appendix:section_flow_transform} the main difficulty of establishing such a result in the càdlàg-setting is that in general $\phi$ is not invertible. In order to get a grasp on what a generalisation of \cref{example:smooth_solution_Duhamel} to the càdlàg regime should look like, we first examine the case of finitely many jumps. 
\begin{example}
\label{example:Duhamel}
Suppose $\mathbf{X}$ has finitely many jumps denoted by $(t_{i})_{i=1}^{N}\subset (0, T)$. Thus by \cref{lemma:finite_jumps_flow_inversion} we know that the RDE flow $\phi$ is invertible on any half-open intervall $[t_{i}, t_{i+1})$ i.e. $\forall i=0, \dots, N$ and $\forall t \in [t_{i}, t_{i+1})$, $\phi_{t_{i}, t}: \mathbb{R}^{d_{Y}}\to \mathbb{R}^{d_{Y}}$ is invertible, where $t_{0}=0, t_{N+1}=T$. In particular this yields by classical variation of constants arguments that the process 
\begin{equation}
\label{eq:example_Duhamel_pre_jump}
\begin{aligned}
Y_{t}&= \phi_{0, t} \xi + \int_{0}^{t} \phi_{s, t}\dd M_{s}- \sum_{0 \leq s \leq t} 
\Delta(\phi_{0, \cdot})_{s} \phi^{-1}_{0,s-} \Delta M_{s}\\
&= \phi_{0, t} \xi + \int_{0}^{t} \phi_{s, t}\dd M_{s}
\end{aligned}
\end{equation}
is the unique solution of \eqref{eq:inhomogeneous_linear_RSDE} on $[0,t_{1})$. Now at $t_{1}$ we make the ansatz \eqref{eq:example_Duhamel_pre_jump} and then suitably adjust the jump following the dynamic \eqref{eq:inhomogeneous_linear_RSDE}. In particular we see that 
\begin{equation}
\label{eq:example_Duhamel_at_jump}
\begin{aligned}
    Y_{t_{1}}=& \phi_{0, t_{1}} \xi + \int_{0}^{t_{1}} \phi_{s,t_{1}}\dd M_{s} + \Delta M_{t_{1}} + A_{t_{1}-} Y_{t_{1}-} \Delta X_{t_{1}}+ (A'_{t_{1}-}+ A^{2}_{t_{1}-})Y_{t_{1}-} \Delta \mathbb{X}_{t_{1}} \\
    &- \Delta (\phi_{0, \cdot})_{t_{1}}\xi- \Delta \left(\int_{0}^{\cdot} \phi_{s,\cdot}\dd M_{s}\right)_{t_{1}}
\end{aligned}
\end{equation}
solves \eqref{eq:inhomogeneous_linear_RSDE} 'at $t_{1}$'. Now note that for any $u \in [0,T]$ it holds 
\begin{equation}
\label{eq:jump_Volterra}
    \begin{aligned}
        \Delta \left(\int_{0}^{\cdot} \phi_{s,\cdot}\dd M_{s}\right)_{t_{1}}=& \lim_{u \uparrow t_{1}} \int_{u}^{t_{1}} \phi_{s, t_{1}}\dd M_{s} + \int_{0}^{u} (\phi_{s, t_{1}}- \phi_{s, u})\dd M_{s}\\
        =&\phi_{t_{1}-, t_{1}} \Delta M_{t_{1}} + \lim_{u \uparrow t_{1}}(\phi_{u, t_{1}}- \operatorname{Id}) \int_{0}^{u} \phi_{s,u}\dd M_{s}\\
        =&\phi_{t_{1}-, t_{1}} \Delta M_{t_{1}} + \lim_{u \uparrow t_{1}}(\phi_{u, t_{1}}- \operatorname{Id}) (Y_{u}- \phi_{0, u} \xi)\\
        =&\phi_{t_{1}-, t_{1}} \Delta M_{t_{1}} + A_{t_{1-}} Y_{t_{1}-}\Delta X_{t_{1}}+  (A'_{t_{1}-}+ A^{2}_{t_{1}-})Y_{t_{1}-} \Delta \mathbb{X}_{t_{1}}\\
        &-\Delta (\phi_{0,\cdot})_{t_{1}}\xi,
    \end{aligned}
\end{equation}
using \cref{flow:RDE} in the last line. Inserting this into \eqref{eq:example_Duhamel_at_jump} yields 
\begin{equation}
    \label{eq:example_Duhamel_after_jump}
    Y_{t} = \phi_{0, t} \xi + \int_{0}^{t} \phi_{s,t}\dd M_{s} + \sum_{0 \leq r\leq t} (\phi_{r, t}- \phi_{r-, t}) \Delta M_{r}
\end{equation}
solves \eqref{eq:inhomogeneous_linear_RSDE} on $[0, t_{1}]$. Now on $[t_{1}, t_{2})$ the flow is again invertible and thus by Markovianity of \eqref{eq:inhomogeneous_linear_RSDE} and the semi-flow property of $\phi$ we see that the unique solution on $[t_{1}, t_{2})$ needs to be of the form
\begin{equation*}
\begin{aligned}
    Y_{t}= \phi_{t_{1}, t}Y_{t_{1}} + \int_{t_{1}}^{t} \phi_{s, t}\dd M_{s}=
    \phi_{0, t} \xi + \int_{0}^{t} \phi_{s, t}\dd M_{s} + \sum_{0\leq r \leq t_{1}} (\phi_{r, t} - \phi_{r-, t}) \Delta M_{r}.
\end{aligned}
\end{equation*}
Adding an analogous jump correction as above at $t_{2}$, yields that the unique solution on $[0,t_{2}]$ is again of the form \eqref{eq:example_Duhamel_after_jump}. Iterating this argument over all (finitely many) jumps yields that \eqref{eq:example_Duhamel_after_jump} is the unique solution of \eqref{eq:inhomogeneous_linear_RSDE} on $[0,T]$. 
\end{example}

In general the rough driver $\mathbf{X}$ can have countably many jumps that possibly accumulate. Therefore verifying the formula \eqref{eq:example_Duhamel_after_jump} is not possible through a naive limiting argument built on \cref{example:Duhamel}. Instead the proof will follow by our Doob-Meyer decomposition \cref{corollary: Doob-Meyer for strongly controlled} in combination with variational estimates on the covariation bracket à \cite[Theorem 5.1]{fzk23}.
\begin{theorem} \label{lemma: Duhamel Y is strongly stochastically controlled}
Let $p \in [2, 3)$, $m\in [2,\infty)$, $\bX:= (X,\X) \in \sV^{p}([0,T]; \R^{d_{X}})$ and $(A, A')\in \sV^{p}_{X}\cL(\R^{d_{Y}}; \cL(\mathbb{R}^{d_{X}};\R^{d_{Y}}))$ deterministic. Further, let $M \in \cM^{p,m}$ and $\xi \in L^{m}(\mathfrak{F}_{0})$. Then, for any $t \in [0,T]$ it holds $\sum_{0\leq s\leq t} \Delta (\phi_{\cdot, t})_{s}\Delta M_{s} \in L^{m}$ and the process 
    \begin{equation}
    \label{eq: definition Y for Duhamel}
        Y_{t} := \phi_{0, t}\xi + \int_{0}^{t} \phi_{s, t} \dd M_{s} + \sum_{0 \leq s \leq t} \Delta(\phi_{\cdot, t})_{s} \Delta M_{s}. 
    \end{equation}
    then $(Y, AY, (A^2+A') Y) \in \sV^{p,m}_{\bX}$ is the unique $L^m$-solution to \eqref{eq:inhomogeneous_linear_RSDE}.
\end{theorem}

\begin{proof}
By definition of the semi-flow and a priori bounds on (deterministic) linear RDEs (as e.g. implicitly included in \cref{corollary: a priori control bound linear additive RSDE})  we have that 
\begin{equation}\label{eq: uniform p-var bound on semi-flow and remainder}
  \|(\phi_{\cdot, \cdot}-\Id)\|_{p,[0,T]}  \leq C \quad \text{and} \quad \|R\|_{\frac{p}{2},[0,T]} \leq C,
\end{equation}
where we write for $R_{s,t}:= R^{\phi_{s, \cdot}-\Id}_{s,t})$ for any $(s,t)\in \Delta_{[0,T]}$ and $0\leq C < \infty$ is a constant which only depends on $p, T, \|A,A'\|_{\infty; \cL},\llbracket A,A'\rrbracket_{p, [0,T]}$ and $\|\bX\|_{p,[0,T]}$. Indeed, we notice that for any partition $\cP$ of $[0,T]$ we have by \cref{corollary: a priori control bound linear additive RSDE} that 
\begin{equation*}
    \sum_{[u,v]\in \cP} |(\phi_{u,v}-\Id)|^p \lesssim_{p,m,K} \sum_{[u,v]\in \cP} \|\bX\|_{p,[u,v]}^p\leq \|\bX\|^p_{p,[0, T]}.
\end{equation*}
Similarly, again by \cref{corollary: a priori control bound linear additive RSDE} we further have that 
\begin{equation*}
    \sum_{[u,v]\in \cP}  \big|R^{(\phi_{u,\cdot}-\Id)}_{u,v}\big|^{\frac{p}{2}} \lesssim_{p,m,K} \sum_{[u,v]\in \cP} \big(\|\bX\|_{p,[u,v]}+ \|A,A'\|_{p,[u,v]} \big)^p \lesssim \|\bX\|^p_{p,[0, T]}+ \|A,A'\|^p_{p,[0,T]}.
\end{equation*}
Therefore, also
\begin{align}\label{eq:Duhamel_proof_boundphi}
    \sup_{(s,t) \in [0,T]} |\phi_{s,t}|\leq 1+C \quad \text{and} \quad \bigg(\sum_{s\leq r\leq t} |\Delta (\phi_{\cdot,t })_{r}|^p\bigg)^{\frac{1}{p}} \lesssim C .  
\end{align}
By \cite[Theorem 5.1]{fzk23} with $r=\infty$ and $p_1=p$, we thus may verify that 
\begin{equation}\label{eq: bound jumps semi-flow times martingale}
\begin{split}
     \bigg\|\sup_{(s,t) \in \Delta_{[0,T]}} \bigg|\sum_{s \leq r \leq t} \Delta( \phi_{\cdot, t})_{r} \Delta M_{r} \bigg| \bigg\|_{L^m} &\lesssim \bigg(\sum_{s\leq u\leq t} |\Delta (\phi_{\cdot,t })_{s}|^p\bigg)^{\frac{1}{p}} \bigg\|\bigg(\sum_{0\leq s\leq T} |\Delta M_s|^2\bigg)^{\frac{1}{2}} \bigg\|_{L^m}\\
    &\lesssim \|\bX\|_{p,[0,T]} \|[M]_T^{\frac{1}{2}}\|_{L^m}   <\infty
\end{split}
\end{equation}
and therefore also by another application of the BDG inequality that 
\begin{equation}\label{eq: Duhamel Lm bound}
\begin{split}
    \sup_{t \in [0,T]} \| Y_{t} \|_{L^{m}} &\lesssim
    \sup_{t \in [0,T]}|\phi_{0,t}| \|\xi\|_{L^{m}} + \sup_{(s, t) \in \Delta_{[0,T]}} |\phi_{s,t}| \|M_T \|_{L^m}  \\
    &+ \sup_{t\in [0,T]}\bigg\|\sum_{0 \leq r \leq t} \Delta( \phi_{\cdot, t})_{r} \Delta M_{r}  \bigg\|_{L^m} < \infty.
\end{split}
\end{equation}
By the semi-flow property of $\phi$ one verifies directly, that 
\begin{equation}\label{eq: Markov property Duhamel}
    Y_{t} = \phi_{s,t} Y_{s} + \int_{s}^{t} \phi_{r, t} \dd M_{r} + \sum_{s< r\leq t}\Delta(\phi_{\cdot, t})_{r} \Delta M_{r}
\end{equation}
and thus especially that 
\begin{equation} \label{eq: Duhamel proof id_1}
\begin{split}
    \delta (Y-M)_{s,t}= \int_{s}^{t} (\phi_{r,t}-\Id) \dd M_{r} + (\phi_{s,t} - \Id) Y_{s} + \sum_{s < r \leq t}\Delta(\phi_{\cdot, t})_{r} \Delta M_{r}. 
\end{split}
\end{equation}
Thus, we find that 
\begin{equation*}
    \begin{aligned}
    \| \delta (Y-M)_{u,v} \|_{L^{m}} \leq &|(\phi_{u,v}-\Id)|\| Y_{u}\|_{L^{m}}+ \bigg \| \int_{u}^{v} (\phi_{r, v}-\Id) \dd M_{r} \bigg \|_{L^{m}}+ \bigg \| \sum_{u < r \leq v} |\Delta (\phi_{\cdot, v})|_{r}  \Delta M_{r} \bigg \|_{L^{m}}\\
    \lesssim & \|(\phi-\Id)\|_{p,[u,v]}+ \| M \|_{p,m,[u,v]}+ \|\bX\|_{p,[u,v]}\|M\|_{p,m,[u,v]}
    \end{aligned}
\end{equation*}
and similarly 
\begin{equation*}
    \begin{split}
        \Big \|\E_{u}\Big[\delta Y_{u,v}- A_{u}Y_{u}\delta X_{u,v}\Big] \Big\|_{L^{m}} &= \Big \|\E_{u}\Big[\delta (Y-M)_{u,v}- A_{u}Y_{u}\delta X_{u,v}\Big] \Big\|_{L^{m}}  \\
        &\leq  |(\phi_{u,v}-\Id) -A_{u}\delta X_{u,v}| \|Y_{u}\|_{L^{m}}\\
        &\lesssim_{C, \|A,A'\|_{\infty; \cL}} \Big\|\sup_{r\in [\cdot,\cdot]} |R^{(\phi_{r,\cdot}-\Id)}| \Big\|_{\frac{p}{2}, [u,v]}
    \end{split}
\end{equation*}
This shows that indeed $(Y,AY) \in \sV^{p,m}_X$. Let us for simplicity write 
\begin{equation*}
  Y^{\sharp}_{s,t}\coloneqq \delta (Y-M)_{s,t}- A_{s} Y_{s} \delta X_{s, t}- (A'_{s} + A^{2}_{s})Y_{s}\mathbb{X}_{s,t} 
\end{equation*}
and define $(\phi-\Id)^{\sharp}$ analogously.
Then, by \eqref{eq: Duhamel proof id_1} we have for any $(s,t)\in \Delta_{[0,T]}$ that  
\begin{equation*}
\begin{split}
    \| \E_{s}[Y^{\sharp}_{s,t}] \|_{L^{m}}\lesssim &|(\phi- \Id)^{\sharp}_{s,t}| \| Y_{s}\|_{L^{m}}
\end{split}
\end{equation*}
and thus indeed $\E_{\cdot}[Y^{\sharp}] \stackrel{\frac{p}{3}}{=} 0$ in $L^m$, which implies that $(Y, AY, (A^2+A') Y) \in \sV^{p,m}_{\bX}$ as wanted. By \cref{corollary: Doob-Meyer for strongly controlled} we know that 
\begin{equation*}
    Y_t = \xi + M^{Y}_t +\int_0^t (A,A')_r Y_r \dd \bX_r,
\end{equation*}
where $M^Y$ is the martingale defined by the RRS convergence 
\begin{equation*}
    M^Y_t = \lim_{|\cP| \to 0} \sum_{(u,v) \in \cP} \delta Y_{u,v}- \E_{u}[\delta Y_{u,v}].
\end{equation*}
Thus, all we have left to show is that $M^Y\equiv M$.
To this end we find that analogously to \eqref{eq:jump_Volterra}, using \eqref{eq: Markov property Duhamel} we have that
\begin{equation}\label{eq:Duhamel_forward_jump_Volterra}
\begin{split}
    &\Delta \bigg( \int_{s}^{\cdot} (\phi_{r, \cdot}-\Id)  \dd M_{r} \bigg)_{t} = \lim_{u \uparrow t} \bigg(\int_{u}^{t}  \phi_{r, t} \dd M_{r}- \Delta M_{t} + (\phi_{u,t} - \Id) \int_{s}^{u} \phi_{r, u} \dd M_{r} \bigg)\\
    &=(\phi_{t-, t}-\Id) \Delta M_{t} +(\phi_{t-,t} - \Id) Y_{t-} + (\phi_{s, t-} - \phi_{s, t})Y_{s}- (\phi_{t-,t} -\Id) \lim_{u \uparrow t}\sum_{s< r \leq u} \Delta( \phi_{\cdot, u})_{r} \Delta M_{r}\\
    &= (\phi_{t-,t}-\Id)\bigg(\Delta M_t - (\phi_{s,t-}-\Id) Y_s+ \delta Y_{s,t-} - \sum_{s<r<t} \Delta( \phi_{\cdot, t-})_{r} \Delta M_{r}\bigg)
\end{split}
\end{equation}
where we used, that by \cite[Theorem 5.1]{fzk23} and dominated convergence, which is applicable due to \eqref{eq:Duhamel_proof_boundphi}, we know that
\begin{equation}\label{eq: limit to 0 of jump differences in Duhmael}
\begin{aligned}
    \lim_{u \uparrow t} &\bigg\| \sum_{s< r \leq u} \Delta( \phi_{\cdot, u})_{r} \Delta M_{r} - \sum_{s< r < t} \Delta( \phi_{\cdot, t-})_{r} \Delta M_{r} \bigg \| _{L^{m}}^{m}\\
    &\lesssim \lim_{u \uparrow t}   \bigg(\bigg(\sum_{u< r < t} |\Delta( \phi_{\cdot, t-})_{r}|^{p}  + \sum_{s < r < t} |\Delta( \phi_{\cdot, t-})_{r}- \Delta( \phi_{\cdot, t-})_{r}|^p\bigg)^{\frac{1}{p}} \|\delta M_{s,t}\|_{L^{m}}^{m}=0.
\end{aligned}
\end{equation}
By \eqref{eq: Duhamel proof id_1} we clearly have $\E_{s}[\delta Y_{s,t}]= (\phi_{s,t}- \Id) Y_s$ for any $(s,t)\in \Delta_{[0,T]}$ and therefore, we find again by \eqref{eq: Duhamel proof id_1} for the germ
\begin{align*}
   &\Xi_{u,v} :=\delta (Y-M)_{u,v}- \E_{u}[\delta Y_{u,v}] \\
    &=\int_u^{v-} (\phi_{r,v-}-\Id)\dd M_r +  \sum_{u<r \leq v} \Delta (\phi_{\cdot,v})_r \Delta M_r 
    \\&= \int_u^{v-} (\phi_{r,v-}-\Id)\dd M_r+ (\phi_{v-,v}-\Id)\bigg(\sum_{u<r < v} \Delta (\phi_{\cdot,v-})_r \Delta M_r+ (\phi_{u,v-}-\Id) Y_u +\delta Y_{u,v-}\bigg)
\end{align*}
together with \eqref{eq: Duhamel Lm bound} and \cite[Theorem 5.1]{fzk23} that for any $(u,v) \in \Delta_{[0,T]}$, clearly $\E_{u}[\Xi_{u,v}]=0$ and 
\begin{equation*}
    \|\Xi_{u,v}\|_{L^m} \lesssim \big(\|\phi-\Id\|_{p,[u,v)}+ \|Y\|_{p,m,[u,v)}\big) \Big(\|M\|_{p,m,[u,v]}(1+\|\phi-\Id\|_{p,[u,v]}) + \|\phi-\Id\|_{p,[u,v]}\Big)
\end{equation*}
and thus by \cite[Lemma B.2]{AllanPieper2026} that indeed $M^Y-M \equiv 0$.
\end{proof}

\begin{remark}
    A variation of the above proof together with using either Kolmogorov-criteria ($p < m$) or the weak stability property ($m \leq p$) could have also been used to establish the stability estimate \eqref{ref:Key} without relying on \cref{subsection:Well-posedness and stability of semi-linear RSDEs}. The next two theorems would follow analogously. 
\end{remark}

\begin{theorem}
    Let $p \in (2, 3), m\in [2, \infty)$, $\bX \in \sV^{p}([0,T]; \R^{d_{X}})$ and
    \\
    $(A, A')\in \sV^{p}_{X} \cL(\mathbb{R}^{d_{Y}}; \cL(\mathbb{R}^{d_{X}}; \mathbb{R}^{d_{Y}})) $. Further let $M \in V^{p}\cM^{m}$ and $\xi \in L^{m}(\mathfrak{F}_{0})$. Define $(Y_{t})_{t \in [0,T]}$  as
    \begin{equation*}
        Y_{t}\coloneqq \phi_{0, t}\xi + \int_{0}^{t} \phi_{s, t}\dd M_{s} + \sum_{0 \leq s \leq t} \Delta(\phi_{\cdot, t})_{s} \Delta M_{s}
    \end{equation*} and 
        $\mathcal{Y}\coloneqq (Y, AY, (A^{2}+A')Y; M, A \bullet \d M)$. Then $\mathcal{Y}$ is the unique \emph{good} $L^m$-solution to \eqref{eq:inhomogeneous_linear_RSDE}.
\end{theorem}

In combination with \cref{thm:fixed_point_Banach} and the results of \cref{subsection:Well-posedness and stability of semi-linear RSDEs} this yields the following mild representation of the semi-linear RSDE \eqref{section:ref:semilin_RSDE}:
\begin{theorem}
\label{thm:semilinearDuhamel}
   Let $p \in (2,3), m\in [2, \infty)$, $\bX \in \sV^{p}([0,T]; \R^{d_{X}}), \\
    (A, A')\in \sV^{p}_{X} \cL(\mathbb{R}^{d_{Y}}; \cL(\mathbb{R}^{d_{X}};\mathbb{R}^{d_{Y}}))$ and $\xi \in L^{m}(\mathfrak{F}_{0})$. Further assume $\sigma = \sigma_t (\omega, x)$ be a predictable
  coefficient field, Lipschitz and of linear growth uniformly in $(t, \omega)$ and $\langle N\rangle_{T}\in L^{\infty}$. Then there is  $(Y_{t})_{t \in [0,T]}$  such that 
    \begin{equation*}
        Y_{t}\coloneqq \phi_{0, t}\xi + \int_{0}^{t} \phi_{s, t} \sigma_{s}(Y_{s-})\dd N_{s} + \sum_{0 \leq s \leq t} \Delta(\phi_{\cdot, t})_{s} \sigma_{s}(Y_{s-}) \Delta N_{s}.
    \end{equation*} Let
        $\mathcal{Y}\coloneqq (Y, AY, (A^{2}+A')Y; \sigma(Y) \bullet \d N, A \sigma(Y) \bullet \d N)$. Then $\mathcal{Y}$ is the unique \emph{good} $L^m$-solution to \eqref{section:ref:semilin_RSDE}. 
\end{theorem}

\section{Well-posedness of RSDEs with bounded coefficients }
\label{section:Well-posedness of fully non-linear RSDEs}
In this section we recall the existence and stability properties of $L^{m, \infty}$-solutions to RSDEs with bounded coefficients  of the form 
\begin{equation}\label{eq: non-linear RSDE}
Y_{t}= \xi + \int_{0}^{t} b_{s}(Y_{s}) \dd s+ \int_{0}^{t} \sigma_{s} (Y_{s-}) \dd N_{s} + \int_{0}^{t} (f, f')_{s} (Y_{s}) \dd \mathbf{X}_{s}.
\end{equation}
We then show the validity of the fixed-point approach described in \cref{Section:main_theorem} in this mixed-moment setting by showing that the needed bound \eqref{ref:Key} is indeed fulfilled due to the known local Lipschitz-properties of such solutions.

\begin{theorem}[Theorem 5.2 in \cite{AllanPieper2026}]\label{theorem: existence and estimates for solutions to RSDEs}
Let $p \in [2,3)$ and $m \in [2,\infty)$. Let $b$ be progressive bounded Lipschitz, let $\sigma$ be predictable bounded Lipschitz, and let $(f,f') \in \sV^{p,m,\infty}_X \Lip^3$ be a stochastic controlled vector field. Let $\xi \in L^m(\mathfrak{F}_{0})$,  $N \in \cM^{p, m, \infty}$ be a c\`adl\`ag martingale, and  $\bX = (X,\X) \in \sV^p$ be a c\`adl\`ag rough path.

Then there exists a process $Y$, which is unique up to indistinguishability, such that $Y$ has almost surely c\`adl\`ag sample paths, $(Y,Y') \in \sV^{p,m,\infty}_X$ is a stochastic controlled path, where $Y' = f(Y)$, and such that, almost surely, \eqref{eq: non-linear RSDE} holds for every $t \in [0,T]$.

Moreover, if $\xi, \tilde{\xi} \in L^m(\mathfrak{F}_{0})$, $N, \tN \in \cM^{p,m,\infty}$ are two c\`adl\`ag martingales and $\bX, \tbX \in \sV^p$ are two c\`adl\`ag rough paths, such that for some constant $K>0$
\begin{equation*}
   \|N\|_{p,m,\infty,[0,T]} \vee \|\tN\|_{p,m,\infty,[0,T]}\vee \|\bX\|_{p,[0,T]}\vee \|\tbX\|_{p,[0,T]} \leq K 
\end{equation*} and if $Y, \tY$ are the solutions to \eqref{main_ref:RSDE} corresponding to the data $(\xi,N,\bX)$ and $(\tilde{\xi},\tN,\tbX)$ respectively, then we have that
\begin{equation}\label{eq: Lipschitz continuity of solution map}
\begin{split}
\|&Y - \tY\|_{p,m,[0,T]} + \|Y' - \tY'\|_{p,m,[0,T]} + \|\E_{\edot} (R^{Y,Y'} - R^{\tY,\tY'})\|_{\frac{p}{2},m,[0,T]}\\
&\leq C \big( \|\xi - \tilde{\xi}\|_{L^m} + \|N - \tN\|_{p,m,[0,T]} + \|\bX - \tbX\|_{p,[0,T]} \big),
\end{split}
\end{equation}
where the constant $C$ depends only on $p, q, T, K$, $\|b\|_{C^1_b}$, $\|\sigma\|_{C^1_b}$, $\|f,f'\|_{\infty;3}$, $\llbracket f, f' \rrbracket_{p,m,\infty,[0,T]}$ and $\llbracket D f, D f' \rrbracket_{p,m,\infty,[0,T]}$.
\end{theorem}

 Using the estimates from \cite[Section 4]{AllanPieper2026} we can formulate the following a priori estimates for RSDEs with bounded coefficients. This can be deducted using  similar arguments as in \cite[Proposition 4.5]{FHL21}. We omit its proof for brevity:

\begin{proposition}\label{proposition: a priori estimates non-linear RSDE} 
    Let $p \in [2,3)$ and $m \in [2,\infty)$. Let $b$ be progressive bounded Lipschitz, $\sigma$ be predictable bounded Lipschitz, and $(f,f') \in \sV^{p,m,\infty}_X \Lip^3$ be a stochastic controlled vector field such that $(D f,D f') \in \sV^{p,m,\infty}_X \Lip^2$.
    Let $K>0$ be a constant so that 
    \begin{equation*}
        \|b\|_{\infty}\vee \|\sigma\|_{\infty}\vee \llbracket f,f'\rrbracket_{p,m, \infty, [0,T]} \vee \|f,f'\|_{\infty;3}\vee \|\bX\|_{p,[0,T]}\vee \|N\|_{p,m,\infty, [0,T]} \leq K.
    \end{equation*}
    Then, for any $L^{m,\infty}$-solution $(Y,Y') \in \sV^{p,m,\infty}_X$ of \eqref{eq: non-linear RSDE} we have that 
    \begin{equation} \label{eq: priori bound for linear RSDE}
    \begin{split}
       \|Y\|_{p,m,\infty,   [0,T]}+\|\E_\cdot R^{Y}\|_{\frac{p}{2},\infty,[s,t]} \leq C
    \end{split}
    \end{equation}
    for a constant $C>0$ which only depends on $p,m$ and $K$.
\end{proposition}

In the continuous setting the following local-Lipschitzness result in the coefficient fields has already been shown in \cite[Theorem 4.9]{FHL21}. As, however, in the corresponding work for càdlàg settings \cite{AllanPieper2026}, the analogous result is missing we give a proof here for the convenience of the reader:

\begin{lemma}\label{lemma: stability in vfs for non-linear RSDEs}
  Assume $p,m, \xi, b, \sigma$ and $(f,f')$ as well as  $(\tf,\tf')$ are as in \cref{theorem: existence and estimates for solutions to RSDEs} and denote by $(Y,Y') \in \sV^{p,m,\infty}_X$ and $(\tY,\tY')\in \sV^{p,m,\infty}_X$ the two solutions to \eqref{eq: non-linear RSDE} with initial data $(\xi,b,\sigma, (f,f'),M, \bX)$ and $(\xi,b,\sigma, (\tf, \tf'), M, \bX)$ respectively provided by \cref{theorem: existence and estimates for solutions to RSDEs}. Then, we have that 
  \begin{equation}\label{eq: lipschitz in vector fields}
  \begin{split}
       \Big\|& \sup_{v\in [0,T]} | Y_{v}  -  \tY_{v}| \Big\|_{L^m}+ \|Y - \tY\|_{p,m,[0,T]} + \|Y' - \tY'\|_{p,m,[0,T]} + \|\E_{\edot} (R^{Y,Y'} - R^{\tY,\tY'})\|_{\frac{p}{2},m ,[0,T]}\\
       &\leq C \Big(\|f-\tf,f'-\tf'\|_{\infty; 2} + \llbracket f,f'; \tf, \tf' \rrbracket_{p,m, \infty, [0,T]}\Big)
  \end{split}
  \end{equation}
  for a constant $C>0$ which only depends on $p,m, T, \|M\|_{p,m,\infty, [0,T]}, \|\bX\|_{p,[0,T]}$, $\|f,f'\|_{\infty;3}$,\\ $\llbracket f,f' \rrbracket_{p,m, \infty, [0,T]},\|\tf,\tf'\|_{3,\infty}$ and $\llbracket \tf,\tf' \rrbracket_{p,m, \infty, [0,T]}$.
\end{lemma}
\begin{proof}
    Clearly, we have that 
    \begin{align*}
        &\bigg \|\sup_{v\in [s,t]} \bigg|\int_s^v b_r(Y_r)- b_r(\tY_r) \dd r \bigg| \bigg\|_{L^m} + \bigg \| \int_0^{\cdot} b_r(Y_r)- b_r(\tY_r) \dd r \bigg\|_{p,m, [s,t]}\\
        &\lesssim \Big(\|Y_s-\tY_s\|_{L^m} + \|Y-\tY \|_{p,m, [s,t)} \Big)(t-s)
    \end{align*}
    and similarly by \cite[Lemma B.3]{AllanPieper2026} and BDG we also have that 
    \begin{align*}
        &\bigg \|\sup_{v\in [s,t]} \bigg|\int_s^v \sigma_r(Y_{r-})- \sigma_r(\tY_{r-}) \dd r \bigg\|_{L^m} + \bigg \| \int_0^{\cdot} \sigma_r(Y_{r-})- \sigma_r(\tY_{r-}) \dd M_r \bigg| \bigg\|_{p,m, [s,t]} \\
        &\lesssim  \Big(\|Y_s-\tY_s\|_{L^m} + \|Y-\tY \|_{p,m, [s,t)}\Big) \|M\|_{p,m,\infty, [s,t]}.
    \end{align*}
    Finally, by \cite[Lemma 4.6 \& Lemma 4.9]{AllanPieper2026} we further know that 
    \begin{align*}
        &\bigg\|\sup_{v\in [s,t]} \bigg|\int_s^{v} (f,f')_r(Y_r)-(\tf,\tf')_r(\tY_r) \dd \bX_r\bigg| \bigg\|_{L^m} +  \bigg\|\int_0^{\cdot} (f,f')_r(Y_r)-(\tf,\tf')_r(\tY_r) \dd \bX_r\bigg\|_{p,m,[s,t]}\\
        & \lesssim \Big(\|f(Y)-\tf(\tY)\|_{p,m,[s,t)} + \|D f(Y)f(Y)-D \tf(\tY)\tY +f'(Y)-\tf'(\tY)\|_{p,m,[s,t)} \\
        &\quad + \|\E_{\edot} R^{(f(Y)-\tf(\tY))}\|_{\frac{p}{2},m,[s,t)} +\|f_s(Y_s)-\tf_s(\tY_s)\|_{L^{m}} \\
        & \quad+ (1+\|X\|_{p,[s,t)}) \sup_{u \in [s,t)} \|D f_u(Y_u)f(Y_u)- D \tf_u(\tY_u)\tf(\tY_u) +f'_u(Y_u)-\tf'_u(\tY_u)\|_{L^{m}}\Big) \|\bX\|_{p,[s,t]}
    \end{align*}
    thus by \cite[Lemma 4.13]{AllanPieper2026} together with \cref{proposition: a priori estimates non-linear RSDE} we can conclude that 
    \begin{align*}
        \Big\|& \sup_{v\in [s,t]} |\delta Y_{s,v}  - \delta \tY_{s,v}| \Big\|_{L^m}+ \|Y - \tY\|_{p,m,[s,t]} +\|f(Y)-\tf(\tY)\|_{p,m,[s,t]}\\
        &\quad + \|Y' - \tY'\|_{p,m,[s,t]} + \|\E_{\edot} (R^{Y,Y'} - R^{\tY,\tY'})\|_{\frac{p}{2},m ,[s,t]}\\
        &\leq C\Big(\|Y_s-\tY_s\|_{L^m}+ \|Y - \tY\|_{p,m,[s,t)} + \|Y' - \tY'\|_{p,m,[s,t)} + \|\E_{\edot} (R^{Y,Y'} - R^{\tY,\tY'})\|_{\frac{p}{2},m ,[s,t)}\Big)w_1(s,t)^{\frac{1}{p}}\\
        &\quad+ \eta
    \end{align*}
    where we set 
    \begin{equation*}
       w_1(s,t) := \|\bX\|^p_{p,[s,t]}+ \|M\|^p_{p,m, [s,t]}+ (t-s)^p
    \end{equation*}
    \begin{equation*}
    \begin{split}
         \eta &:= C\Big(\|(f-\tf), (f'-\tf')\|_{\infty;2} + \llbracket f,f ; \tf, \tf' \rrbracket_{p,m,[0,T]}\Big) \big(\|\bX\|_{p,[0,T]}+\|Y\|_{p,m,\infty, [0,T]}\big)
    \end{split}
    \end{equation*}
    where we used that 
    \begin{equation*}
    \begin{split}
        \|f(Y)-\tf(\tY)\|_{p,m,[s,t]} &\lesssim \|Y_s-\tY_s \|_{L^m}+\|Y-\tY\|_{p,m,[s,t]} \\
        &\quad + \Big(\llbracket f-\tf \rrbracket_{p,m,[s,t]} +\|(f-\tf), (f'-\tf')\|_{\infty; 2}\Big)\|Y\|_{p,m,\infty, [s,t]}
    \end{split}
    \end{equation*}
    and where the constant $C>0$ only depends on $p,m, T, \|M\|_{p,m,\infty, [0,T]}, \|\bX\|_{p,[0,T]}$, $\|f,f'\|_{\infty;3}$,\\ $\llbracket f,f' \rrbracket_{p,m, \infty, [0,T]},\|\tf,\tf'\|_{\infty;3}$ and $\llbracket \tf,\tf' \rrbracket_{p,m, \infty, [0,T]}$. Thus, the claim follows by an application of \cref{lemma: gronwall type result p var} and using the uniform bound on $\|Y\|_{p,m,\infty,[0,T]}$ provided by \cref{proposition: a priori estimates non-linear RSDE}.
\end{proof}

\begin{corollary}\label{corollary: the fixed point argument is applicable for mixed moment assumptions}
    Let $p \in (2,3)$, let $\bX = (X,\X) \in \sV^p$ be a c\`adl\`ag rough path, $\xi \in L^2$ be $\mathfrak{F}_0$-measurable and $b$ be progressive bounded Lipschitz, let $\sigma$ be predictable bounded Lipschitz, and let $(f,f') \in \sV^{p}_X \Lip^{3}$ be a deterministic controlled vector field. Then, for every $N\in V^p L^{2,\infty}$ which additionally fulfils that there exists a mild control $\rho$ such that $\|\delta \langle N \rangle_{s,t}\|_{L^{\infty}} \leq \rho(s+,t)$ for any $(s,t) \in \Delta_{[0,T]}$, there exists a non-empty closed subset $\cB \subset \cM^2$ which is invariant under $\Phi$ as defined in \eqref{eq: defn Phi} such that for any $M\in \cB$ the equation \eqref{ref:Key} is fulfilled and the solution found through \cref{thm:fixed_point_Banach} is the unique good solution to \eqref{main_ref:RSDE}.
\end{corollary}

\begin{proof}
    For brevity we set $b\equiv 0$ but the general case can easily be concluded by a direct adaptation of the arguments presented here. Define, $\nu := 4 (1+\|\sigma\|_{\infty}^2) \|N\|_{p,2,\infty, [0,T]}$ and set 
    \begin{equation*}
        \cB^{p,\infty}(\nu):= \{Y \in V^p L^{2,\infty} : \|Y\|_{p,2,\infty,[0,T]} \leq \nu\},
    \end{equation*}
    then we know by \cite[Lemma 2.10]{AllanPieper2026} that $\cB^{p,\infty}(\nu)$ is a closed subset of $V^p L^2$. Set
    \[\cB := \cB^{p,\infty}(\nu) \cap \cM.\]
   then one easily checks that $N\in \cB$ by the choice of $\nu$ and $\cB \subset \cM^2$ is a closed subset with respect to $\||\cdot|^{*}_T\|_{L^2}$. Indeed, the latter follows by the previously mentioned closedness of $\cB^{p,\infty}$ and by the fact that for any martingale $M\in \cB$ we have by \cite[Remark 2.14]{AllanPieper2026} that
   \begin{equation}\label{eq: equivalence p,p-norm with cM^p norm}
       \|M\|_{p,2,[0,T]}\leq \|M\|_{2,2,[0,T]} \sim \big\||M|^{*}_{T} \big\|_{L^2}.
   \end{equation}
   To see that $\cB$ is indeed invariant under $\Phi$ we apply the (conditional) BDG inequality and find that for any $(s,t) \in \Delta_{[0,T]}$ we have that 
   \begin{equation*}
       \bigg\|\int_s^t \sigma_s(Y^M_{s-}) \dd N_s\bigg\|_{2,\infty, s} \leq 4 \|\sigma\|^2_{\infty} \|\delta [N]_{s,t}\|^{\frac{1}{2}}_{1, \infty, s}
   \end{equation*}  
   which implies that even for any $M\in \cM^2$ we have that $\|\Phi(M)\|_{p,2,\infty,[0,T]} \leq \nu$ and thus especially that $\cB$ is invariant under $\Phi$.
   We further, need to show that
   \begin{equation*}
       \Big\| \sup_{s\in [0,T]} |Y^M_{s}-Y^{\tM}_s| \Big\|_{L^2} \lesssim \big\| |M-\tM|^{*}_T \big\|_{L^2}. 
   \end{equation*}
   where the multiplicative constant is independent of $M$ and $\tM$ as well as $Y^M$ and $Y^{\tM}$ are defined as in \eqref{eq: additive rsde}. Firstly, notice that by \cref{thm:consistency_of_solutions I} we especially know that $Y^M$ and $Y^{\tM}$ are the unique good solutions to \eqref{eq: additive rsde}. Thus, we know by \cref{theorem: existence and estimates for solutions to RSDEs} that for any $(s,t) \in \Delta_{[0,T]}$ we have that 
   \begin{align*}
       &\big\|Y^M - \tY^{\tM} \big\|_{p,2,[s,t]} + \big \|(Y^M)' - (\tY^{\tM})' \big\|_{p,2,[s,t]} + \big\|\E_{\edot} R^{Y^M-\tY^{\tM},(Y^M-\tY^{\tM})'} \big \|_{\frac{p}{2},2,[s,t]}\\
       &\lesssim \|M-\tM\|_{p,2,[s,t]}
   \end{align*}
   and thus, by \cite[Lemma 4.6]{AllanPieper2026} as well as \cite[Lemma 4.13]{AllanPieper2026} and because $\|M-\tM\|_{p,2,[s,t]} \leq \|M-\tM\|_{2,2,[s,t]}$ also 
   \begin{align*}
           &\Big\| \sup_{v \in [s,t] } |\delta Y^M_{s,v} - \tY^{\tM}_{s,v}| \Big\|_{L^2}\\
           &\leq \bigg\|\sup_{v\in [s,t]}\int_{s}^{v} (f,f')_r(Y^M_r) -(f,f')_r(Y^{\tM}_r) \dd \bX_r \bigg\|_{L^2} + \big\||M-\tM|^{*}_T \big\|_{L^2}\\
           &\lesssim \Big(\|Y^M_s-Y^{\tM}_s\|_{L^2}+ \|(Y^M)'_s- (Y^{\tM})'_s\|_{L^2} + \big\|Y^M - \tY^{\tM} \big\|_{p,2,[s,t]} \\
           &\qquad + \big \|(Y^M)' - (\tY^{\tM})' \big\|_{p,2,[s,t]} + \big\|\E_{\edot} R^{Y^M-\tY^{\tM},(Y^M-\tY^{\tM})'} \big \|_{\frac{p}{2},2,[s,t]}\Big)\|\bX\|_{p,[s,t]}\\
           &\quad +\big\|M-\tM \big\|_{2,2,[s,t]}.
   \end{align*}
   Thus, we may apply \cref{lemma: gronwall type result p var}\footnote{with $Y= Y^M- Y^{\tM}$, $R=R^{(Y^M - \tY^{\tM})}$ $R^1= (Y^M)'-(Y^{\tM})'$ and $R^2=X$ and $w_1(s,t)=\|\bX\|^p_{p,[s,t]}$ as well as $\eta \lesssim \|M-\tM\|_{2,2,[0,T]}$.} to conclude the claimed Lipschitz estimate \eqref{ref:Key}. Thus, we may apply \cref{thm:fixed_point_Banach} and \cref{thm:consistency_of_solutions I} and thus know indeed that the found solution is a good $L^{2,\infty}$-solution.
\end{proof}

\begin{remark}
    Notice that we of course already knew by \cref{theorem: RSDE solution is L^q RSM} that for any $b$ progressive bounded Lipschitz, any $\sigma$ predictable bounded Lipschitz, and $(f,f') \in \sV^{p}_X \Lip^{3}$ any deterministic controlled vector field such that $(D f,D f') \in \sV^{p}_X \Lip^2$, that the $L^{2, \infty}$-solution to \eqref{eq: mixed moment RSDE} is also a good $L^{2, \infty}$-solution of almost sure $V^{\tilde{p}}$-regularity for any $p<\tp$. Nonetheless, \cref{corollary: the fixed point argument is applicable for mixed moment assumptions} confirms that even in mixed moments with bounded coefficients as previously studied in \cite{AllanPieper2026, FHL21} the fixed point approach laid out in \cref{Section:main_theorem} is a viable alternative proof strategy and lets us recover the goodness of the solution even in $p$-Var.
\end{remark}
\addtocontents{toc}{\protect\setcounter{tocdepth}{1}}
\appendix
\crefalias{section}{appendix}
\section{Some aspects of càdlàg rough path theory}
\label{appendix:deterministic_rough_path_theory}
Here we collect some results on the theory of deterministic càdlàg rough path, that are not yet available in the literature.
\subsection{Rough It\^{o} formula for strongly controlled rough paths}
The following is a càdlàg rough It\^{o} formula \cite[Theorem 2.12]{Friz_Zhang_2018} on the level of strongly controlled rough paths. 
\begin{theorem}
\label{theorem:roughIto}
    Let $p \in [2, 3), \mathbf{X}\in \mathscr{V}^{p}([0,T]; V)$ and $\mathcal{Y}\coloneqq(Y, Y', Y'')\in \mathscr{V}^{p}_{\mathbf{X}}([0,T]; W)$. Let $F \in \Lip^{3}(W; U)$ for some $\delta >0$. Then $(D(F)Y', D^{2}F(Y)(Y', Y')+ DF(Y)Y'')\in \mathscr{V}_{\mathbf{X}}^{p}([0, T]; \mathcal{L}(V; W))$ and 
    \begin{equation*}
    \begin{aligned}
        F(Y_{t})= & F(Y_{0})+ \int_{0}^{t} (DF(Y_{s}) Y'_{s}, DF(Y_{s}) Y''_{s}+ D^{2} F(Y_{s}) (Y'_{s}, Y'_{s}))\dd \mathbf{X}_{s}\\
        &+ \frac{1}{2} \int_{0}^{t} D^{2}F(Y_{s})(Y'_{s}, Y'_{s}) \dd [\mathbf{X}]_{s}\\
        &+ \sum_{0 \leq s\leq t} F(Y_{s}) - F(Y_{s-}) - DF(Y_{s-}) Y'_{s-} \Delta X_{s}- DF(Y_{s-}) Y''_{s-} \Delta \mathbb{X}_{s} \\
        & \quad - \frac{1}{2} D^{2} F(Y_{s-})(Y'_{s-}, Y'_{s-}) (\Delta X_{s})^{\otimes 2}, 
    \end{aligned}    
    \end{equation*}
    where the terms on the r.h.s. are rough integrals, Young integrals or sums converging in RRS sense.
    \begin{proof}
    First note that by \cref{def:scRP}  and Taylor's formula it holds that
    \begin{equation*}
    \begin{aligned}
       F(Y_{t-})- F(Y_{s})\stackrel{\frac{p}{3}}{=} &DF(Y_{s}) \delta Y_{s, t-}+ \frac{1}{2} D^{2} F(Y_{s
       }) (\delta Y_{s,t-}, \delta Y_{s, t-})\\
       \stackrel{\frac{p}{3}}{=}& DF(Y_{s})Y'_{s} \delta X_{s,t-}+ DF(Y_{s}) Y''_{s} \mathbb{X}_{s, t-} +\frac{1}{2} D^{2}F(Y_{s}) (Y'_{s}, Y'_{s}) (\delta X_{s,t-})^{\otimes 2}\\
       = &  DF(Y_{s})Y'_{s} \delta X_{s,t}+ DF(Y_{s}) Y''_{s} \mathbb{X}_{s, t} +\frac{1}{2} D^{2}F(Y_{s}) (Y'_{s}, Y'_{s}) (\delta X_{s,t})^{\otimes 2} + A^{1}_{s,t}\\
       =& DF(Y_{s})Y'_{s} \delta X_{s,t}+ DF(Y_{s}) Y''_{s} \mathbb{X}_{s, t} +D^{2}F(Y_{s}) (Y'_{s}, Y'_{s}) \mathbb{X}_{s,t}\\
       &+ \frac{1}{2} D^{2}F(Y_{s}) (Y'_{s}, Y'_{s}) \delta [\mathbf{X}]_{s,t} + A^{1}_{s,t}, 
    \end{aligned} 
    \end{equation*}
    where by Chen's relation
    \begin{equation*}
    \begin{aligned}
        -A^{1}_{s, t} =& DF(Y_{s}) Y'_{s} \Delta X_{t} + DF(Y_{s}) Y''_{s} (\Delta \mathbb{X}_{t}+ \delta X_{s,t-} \otimes \Delta X_{t})+ D^{2} F(Y_{s}) (Y'_{s} , Y'_{s}) \delta X_{s, t-} \otimes \Delta X_{t} \\
        &+ \frac{1}{2}D^{2} F(Y_{s}) (Y'_{s} , Y'_{s}) (\Delta X_{t})^{\otimes 2}
    \end{aligned}
    \end{equation*}
    Now note that by Taylor's formula it follows for $A^{2}_{s}\coloneqq DF(Y_{s-}) \Delta Y_{s} + \frac{1}{2} D^{2}F(Y_{s-})(\Delta Y_{s})^{\otimes 2}$ that for every $\epsilon > 0$ there is a subset $J^{\epsilon}\subset [0,T]$ it holds
    \begin{equation*}
       \bigg|\sum_{s \in (J^{\epsilon})^{c}} F(Y_{s})-F(Y_{s-}) - A^{2}_{s} \bigg|\lesssim \bigg|\sum_{s \in (J^{\epsilon})^{c}}  |\Delta Y_{s}|^{3} \bigg|\leq \epsilon
    \end{equation*}
    which yields 
    \begin{equation*}
        \lim_{|\mathcal{P}| \downarrow 0} \sum_{[t_{i}, t_{i+1}] \in \mathcal{P}} F(Y_{t_{i+1}}) - F(Y_{t_{i+1}-}) - A^{2}_{t_{i+1}}= \sum_{s \in [0,T]} F(Y_{s})-F(Y_{s-}) - A^{2}_{s}
    \end{equation*}
    in RRS-sense by taking refinements of partitions in $J^{\epsilon}$ as $\epsilon \downarrow 0$. 
      Now using \cref{remark:scRPandRI} and the jump-structure of the rough integral it follows
      \begin{equation*}
      \begin{aligned}
          A^{1}_{s, t}+ A^{2}_{t}=& \left (DF(Y_{t-}) Y'_{t-} -DF(Y_{s}) Y'_{s} - DF(Y_{s}) Y''_{s}\delta X_{s,t-}- D^{2} F(Y_{s}) (Y'_{s}, Y'_{s}) \delta X_{s,t-} \right) \Delta X_{t} \\
          &+ (DF(Y_{t-}) Y''_{t-} - DF(Y_{s}) Y''_{s}) \Delta \mathbb{X}_{t} \\
          &+ \frac{1}{2} ( D^{2} F(Y_{t-}) (Y'_{t-}, Y'_{t-}) - D^{2} F(Y_{s}) (Y'_{s}, Y'_{s}))(\Delta X_{t})^{\otimes 2}\stackrel{\frac{p}{3}}{=}0.
      \end{aligned}    
      \end{equation*}
      Hence the claim follows by applying \cref{lemma:expansion_notation}. 
    \end{proof}
\end{theorem}
The advantage of \cref{theorem:roughIto} over \cite[Theorem 2.12]{Friz_Zhang_2018} is that it directly implies  a càdlàg version of the rough product rule \cite[Exercise 7.4]{FRIZHAIRER}, which is included as a special case of \cref{corollary:rs_product_formula}.
\begin{remark}
   Note that \cref{theorem:roughIto} can easily be adjusted to $Y$ being of the form $
       Y_{t}= Y_{0}+
       \int_{0}^{t}(Y'_{s} , Y''_{s}) \dd \mathbf{X}_{s} + A_{t}$ for a bounded variation  drift $A\in V^{1}([0,T]; W)$. The corresponding composition rule then reads as in \cref{theorem:roughIto} with the additional terms $ \int_{0}^{t} DF(Y_{s}) \dd A_{s}$ and $(-DF(Y_{s-})\Delta A_{s})_{s \in [0,t]}$ added in the jump correction.
\end{remark}
\subsection{On càdlàg RDEs and solution semi-flows}
\label{appendix:càdlàg_RDEs}
We start this subsection by investigating an interesting property when working with flows of ordinary/ rough differential equations driven by a càdlàg signal. 
\begin{example}
\label{example:counterexample_flow}
    Let $(x_t)_{t\in [0,T]} \subset \mathbb{R}^{d\times d}$ be a a càdlàg finite variation path. Consider for $\xi \in \mathbb{R}^{d}$ the  ODE
    \begin{equation}
    \label{example:non_invertible_ODE}
         y_t^{s, \xi} = \xi + \int_{s}^{t}  \dd x_r y_{r-}
    \end{equation}
    and note that $y_{t}= y_{t-} +  \Delta x_{t} y_{t-} $. Setting $d =1$ and $x \equiv -\mathbbm{1}_{\cdot \in [t_{0}, T]}$ then implies for the solution semi-flow $\phi_{s,t}(\xi)\coloneqq y^{s, \xi}_{t}$ of \eqref{example:non_invertible_ODE} that
    \begin{align*}
        \phi_{s,t}=\begin{cases}
            1 &; \text{ if } s<t<t_0 \text{ or } t_0\leq s\\
            0 &; \text{ if } s< t_0\leq t
        \end{cases},
    \end{align*}
    which is clearly a non-invertible semi-flow. For $d=2, 3$ one can set for instance $\Delta x_{t_{0}}=\mathcal{O}$, with $\mathcal{O}$ being a rotation matrix by $\pi$ radians. See \cite[Example 1.12]{Friz_Zhang_2018} for a similar example. 
\end{example}
\cref{example:counterexample_flow} implies that in general also solution semi-flows of RDEs driven by càdlàg rough paths are not invertible. This poses a significant technical caveat in \cref{subsection: Duhamel formula}.
The following is a collection of useful properties of solution semi-flows of RDEs driven by càdlàg rough paths along càdlàg rough vector fields.
\begin{proposition}
\label{prop:càdlàg_rough_flows}
    Let $p \in [2,3)$ and $\mathbf{X}=(X, \mathbb{X})\in \mathscr{V}^{p}([0,T]; V)$. Let either $(f, f')\in \mathscr{V}^{p}_{X}\Lip^{3}_{x}(W; \mathcal{L}(V;W))$ or $(f, f')\in \mathscr{V}^{p}_{X}\mathcal{L}(W; \mathcal{L}(V; W))$. Then for any $y \in W$ there exists $(Y^{s, y}, (Y^{s, y})')\in \mathscr{V}^{p}_{X}([s,T]; W)$, where $(Y^{s, y})'=f(Y^{s, y})$ and for any $t\in [0, T]$\footnote{Extending the solution onto the whole interval might seem strange at first, but of course is always possible. The reason we do this is to ensure that the left-limit $\lim_{r \uparrow s}\phi_{r, t}$ in \eqref{item:rough_flow_3} makes sense.} it holds
    \begin{equation}
    \label{flow:RDE}
        Y_{t}^{s,y} = y+ \int_{s}^{t\vee s} (f_{r\vee s}(Y_{r}^{s,y}), Df_{r\vee s}(Y_{r}^{s,y}) f_{r \vee s}(Y^{s, y}) + f'_{r\vee s}(Y^{s,y}_{r})\dd \mathbf{X}_{r}.
    \end{equation}
    We define $\phi: 
    \Delta_{[0,T]}\times W\to W$ by $
    \phi_{s,t}(y)\coloneqq Y^{s,y}_{t}$.Then it holds
    \begin{enumerate}
        \item \label{item:rough_flow_1} $\phi$ is a semi-flow on $[0,T]$ i.e. for any $(s,u,t)\in 
        \Delta_{[0,T]}^{(3)}$ it holds $\phi_{s,t}= \phi_{u,t} \circ \phi_{s,u}$.
        \item \label{item:rough_flow_2} for any $s\in [0,T)$ and $y \in W$ the map $[s, T]\ni t\mapsto \phi_{s,t}(y)\in W$ is càdlàg and $(t\mapsto \phi_{s,t}(y) \in V^{p}([0,T]; W))$ for any $s \in [0,T]$
        \item \label{item:rough_flow_3} for any $t \in (0, T]$ and $y \in W$ the map $[0,t]\ni s \mapsto \phi_{s,t}(y) \in W$ is càdlàg 
        and $(s\mapsto \phi_{s,t}(y) \in V^{p}([0,T]; W))$ for any $t \in [0,T]$
        \item \label{item:rough_flow_4}
        If $\gamma > k + (p-1)$ for $k \in \mathbb{N}$, then for any $(s, t) \in \Delta_{[0,T]}$ it holds $\phi_{s,t}\in C^{k}(W; W)$. 
    \end{enumerate}
\end{proposition}
\begin{proof}
The existence and uniqueness of solutions of \eqref{flow:RDE} follows analogously to \cite[Theorem 3.9]{Friz_Zhang_2018}.
    The claim \eqref{item:rough_flow_1} follows directly by the uniqueness of solutions. The claim \eqref{item:rough_flow_2} follows by the càdlàgness and finite $p$-Var property of the rough integral from \cite[Lemma 2.9]{Friz_Zhang_2018}. The claim \eqref{item:rough_flow_4} can be shown by carefully following the proof of \cite[Theorem B.4]{DauseFrizJentzenSong2026} using the stability of controlled rough paths and integrals in the càdlàg setting à \cite[Lemma 3.4.-3.6]{Friz_Zhang_2018}. Thus we are left to verify \eqref{item:rough_flow_3}. First we show the càdlàgness of paths. Denoting $\mathbf{X}^{\cdot \vee s}\coloneqq (X_{\cdot \vee s}, \mathbb{X}_{\cdot \vee s, \cdot \vee s})\in \mathscr{V}^{p}([0,T]; V)$ it holds $(f, f')_{\cdot \vee s}\in \mathscr{V}^{p}_{X^{\cdot \vee s}}\Lip^{2}$ and  $(Y^{s, y}, f(Y^{s,y}))\in \mathscr{V}^{p}_{X^{\cdot \vee s}}([0,T_{0}]; W)$ and for any $t \in [0,T_{0}]$
    \begin{equation}
    \label{eq:extended_RDE}
        Y^{s,y}_{t} = y + \int_{0}^{t} (f, f')_{r\vee s} (Y_{r}^{s,y})\dd \mathbf{X}^{\cdot \vee s}_{r}.
    \end{equation}
    Hence by the proceeding claims $Y^{s,y}$ is the unique solution of \eqref{eq:extended_RDE} on $[0,T_{0}]$ w.r.t. the driver $\mathbf{X}^{\cdot \vee s}$. Now let $(s_{k})_{k \in \mathbb{N}}\subset [0,T_{0}]$ and suppose first that $s_{k}\downarrow s$. Then we note that for any partition $\Pi$ of $[0,T_{0}]$ it holds
    \begin{equation*}
    \begin{aligned}
       &\sum_{[u,v]\in \Pi} |\delta X_{u,v}^{\cdot \vee s}- \delta X_{u,v}^{\cdot \vee s_{k}}|^{p} \\
      &= \sum_{[u, v]\in \Pi} \left(\mathbbm{1}_{u \leq s < v < s_{k}} + \mathbbm{1}_{u \leq s < s_{k} \leq v} + \mathbbm{1}_{s< u <v < s_{k}}+ \mathbbm{1}_{s< u< s_{k}\leq v}\right)|\delta X_{u,v}^{\cdot \vee s}- \delta X_{u,v}^{\cdot \vee s_{k}}|^{p}\\
      &\eqqcolon \sum_{[u,v]\in \Pi} I + II + III+ IV,
    \end{aligned}
    \end{equation*}
    with 
    \begin{equation*}
    \begin{aligned}
        I&= | \delta X_{s,v}|^{p} \mathbbm{1}_{u \leq s < v < s_{k}}; \quad II = | \delta X_{s,s_{k}}|^{p} \mathbbm{1}_{u \leq s < s_{k} \leq v}; \\ III&=|\delta X_{u,v}|^{p} \mathbbm{1}_{s< u <v < s_{k}}; \quad IV= |\delta X_{u,s_{k}}|^{p} \mathbbm{1}_{s< u< s_{k}\leq v}.
    \end{aligned}
    \end{equation*}
    We note that the term $III$ is the only term that can be non-zero more than once in the above sum and thus 
    \begin{equation*}
    \begin{aligned}
       \sum_{[u,v]\in \Pi} |\delta X_{u,v}^{\cdot \vee s}- \delta X_{u,v}^{\cdot \vee s_{k}}|^{p}&\leq 3 \operatorname{Osc}(X; [s, s_{k}])^{p} + \sum_{[u,v]\in \Pi\cap (s, s_{k})} |X_{u,v}|^{p}\\
       &\leq 3 \operatorname{Osc}(X; [s, s_{k}])^{p} + \Vert X \Vert_{p; (s, s_{k})}^{p}
    \end{aligned}
    \end{equation*}
    In combination with analogous arguments for $\mathbb{X}$ it follows 
    \begin{equation*}
      \Vert \mathbf{X}^{\cdot \vee s}; \mathbf{X}^{\cdot \vee s_{k}} \Vert_{p; [0,T]}\coloneqq \Vert X^{\cdot \vee s} - X^{\cdot \vee s_{k}} \Vert_{p; [0,T]}+ \Vert \mathbb{X}^{\cdot \vee s}- \mathbb{X}^{\cdot \vee s_{k}}\Vert_{\frac{p}{2}; [0,T]} \stackrel{k \to \infty}{\longrightarrow}0. 
    \end{equation*}
    Therefore by the local Lipschitz-continuity of the solution map \cite[Theorem 3.8]{Friz_Zhang_2018} it follows 
    \begin{equation*}
        \Vert Y^{s, y}- Y^{s_{k}, y}\Vert_{p; [0,T]} \stackrel{k \to \infty}{\longrightarrow}0
    \end{equation*}
    and therefore $[0,t]\ni s\mapsto \phi_{s,t}(y)$ is continuous from the right. Now for $s_{k}\uparrow s$ we introduce 
    \begin{equation*}
    \mathbf{X}^{\cdot \vee s-}\coloneqq (X^{\cdot \vee s-}, \mathbb{X}^{\cdot \vee s-}); \quad X^{\cdot \vee s-}_{t}\coloneqq X_{s-}\mathbbm{1}_{s< t}+ X_{t} \mathbbm{1}_{s\geq t}; \quad \mathbb{X}^{\cdot \vee s-}_{u,v} \coloneqq X_{s-, v}\mathbbm{1}_{u< s\leq v}+ \mathbb{X}_{u,v}\mathbbm{1}_{s\leq u\leq v}. 
    \end{equation*}
    One directly checks that $\mathbf{X}^{\cdot \vee s-}$ defines a càdlàg rough path. Analogously to the above one verifies that $\Vert \mathbf{X}^{\cdot \vee s_{k}}- \mathbf{X}^{\cdot \vee s-}\Vert_{p; [0,T]}\to 0$ as $k \to \infty$ Thus denoting the solution w.r.t. $\mathbf{X}^{\cdot \vee s-}$ by $Y^{s-, y}$ we get convergence of $Y^{s_{k}, y}\to Y^{s-, y}$ in $p$-Var. To see that $(s \mapsto \phi_{s,t}(y)\in V^{p}([0,T]; W)$ for any $t \in [0,T], y \in W$ note that by the semi-flow property, Lipschitzness of the solution map w.r.t. the initial condition \cite[Theorem 3.9]{Friz_Zhang_2018} and \cite[Proposition 2.16]{Friz_Zhang_2018} for any $(u, v)\in \Delta_{[0,T]}$ it holds
    \begin{equation*}
    \begin{aligned}
        |\phi_{u,t}(y)-\phi_{v, t}(y)|&\leq |\phi_{v, t}(\phi_{u,v}(y))-\phi_{v, t}(y)|\lesssim \bigg |\int_{u}^{v} (f_{r}, f'_{r})(Y^{u, y}_{r}) \dd \mathbf{X}_{r} \bigg|\\
        &\lesssim \Vert R^{f(Y^{u,y}), f'(Y^{u, y})}\Vert_{\frac{p}{2}; [u,v)}\Vert X \Vert_{p; [u, v]}+ \Vert f(Y^{u, y})\Vert_{p; [u, v)} \Vert \mathbf{X}\Vert_{p; [u,v]},
    \end{aligned}
    \end{equation*} where the implicit constants can be chosen independently of $(u,v)$. This yields the claim.
\end{proof}
Considering \cref{example:counterexample_flow} we also note that in the presence of 'only' finitely many jumps, one can circumvent the global-in-time non-invertibility of the flow by inverting on fixed intervals dictated by the jumps of $\mathbf{X}$. 
\begin{lemma}
\label{lemma:finite_jumps_flow_inversion}
 Let $p \in [2,3)$ and $\mathbf{X}=(X, \mathbb{X})\in \mathscr{V}^{p}([0,T]; V)$. Let $(f, f')\in \mathscr{V}^{p}_{X}\Lip^{3}(W; \mathcal{L}(V: W))$. Then for any $y \in W$ and $s\in [0,T]$, we have that $(Y^{s, y}, (Y^{s, y})')\in \mathscr{V}^{p}_{X}([s,T]; W)$, where $(Y^{s, y})'=f(Y^{s, y})$ and $Y^{s,y}$ solves \eqref{flow:RDE} for any $s \in [0, T_{0}]$. We denote the semi-flow as above by $\phi_{s,t}(y)\coloneqq Y^{y, s}_{t}$. Suppose that $X$ and $\mathbb{X}$ have at most finitely many jump-times $(t_{i})_{i=1}^{N} \subset [0,T]$. Then for any $i=1, \dots N-1$ and any $t \in [t_{i}, t_{i+1})$ the map $\phi_{t_{i}, t}:\mathbb{R}^{d_{Y}}\to \mathbb{R}^{d{Y}}$ is a homeomorphism. 
\end{lemma}
\begin{proof}
Note that for any $i=1, \dots N-1$ and any $t \in [t_{i}, t_{i+1})$ $\mathbf{X}$ is continuous on $[t_{i}, t]$. Thus the homeomorphism property of the flow on this interval follows by suitably adapting the proof of \cite[Theorem B.6]{DauseFrizJentzenSong2026} to the setting of continuous $p$-Var, which we leave to the reader.
\end{proof}
We emphasize that in the càdlàg setting we can however have countably many jumps that possibly accumulate. 
\section{Flow-transform methods for RSDEs}
\label{appendix:section_flow_transform}
We summarize some well-known results on the method of \emph{(rough) flow-transform} for the well-posedness of RSDEs and reframe them in the context of the present manuscript. Initiated in \cite{CDFO13} and continued in an intrinsic manner in \cite[Section 4]{DauseFrizJentzenSong2026} flow-transform methods for RSDEs are applied in the following setting:
\\
Let $\mathbf{X}\in \mathscr{V}^{p, c}([0,T]; \mathbb{R}^{d_{X}})$ for $p \in [2,3)$ and let $(f, f')\in \mathscr{V}^{p, c}_{X}\Lip^{2}$. Then suppose that for any $s\in [0,T]$ and $y \in \mathbb{R}^{d_{Y}}$ there  are unique paths $[s,T] \ni t \mapsto (\phi_{s, t}(y), \phi'_{s,t}(y))$ with $(\phi_{s, \cdot}(y), \phi'_{s,\cdot}(y))\in \mathscr{V}^{p, c}_{X}([s,T]; \mathbb{R}^{d_{Y}})$ such that $\phi_{s, t}'(y)= f_{t}(\phi_{s, t}(y))$ and 
    \begin{equation*}
        \phi_{s,t}(y) = y+ \int_{s}^{t} (f, f')_{r}(\phi_{s,r}(y))\dd \mathbf{X}_{r}.
    \end{equation*}
We denote by $\phi: \Delta_{[0,T]} \times \mathbb{R}^{d_{Y}}\to \mathbb{R}^{d_{Y}}$ the corresponding semi-flow with the short-hand notation $\phi_{t}\coloneqq \phi_{0, t}$. Rough flow-transform methods now rely on the following crucial assumption: 
\begin{assumption}
\label{assumptipon:flow}
$\phi$ is a flow of $C^{3}$-diffeomorphisms i.e. it holds
\begin{enumerate}[label=(\roman*)]
    \item \label{assumptions:flow_1} For any $(s,t) \in \Delta_{[0,T]}$ $\phi_{s,t}: \mathbb{R}^{d_{Y}}\to \mathbb{R}^{d_{Y}}$ is a homeomorphism with its inverse denoted by $\phi_{t,s}\coloneqq \phi^{-1}_{s,t}$.
    \item \label{assumptions:flow_2} For any $(s,t) \in [0,T]^{2}$ it holds  $\phi_{s,t} \in \Lip^{3}(\mathbb{R}^{d_{Y}}; \mathbb{R}^{d_{Y}})$.
\end{enumerate}
\end{assumption}
 \begin{proposition}[Proposition 4.20 in \cite{DauseFrizJentzenSong2026}]\label{prop:flow_transform} Suppose \cref{assumptipon:flow} holds and $M\in \mathcal{M}^{2}$ is continuous with absolutely continuous covariance bracket $\langle M \rangle$. Then the following claims are equivalent:
    \begin{enumerate}
        \item \label{ex:RSDE:item_1}$\mathcal{Y}^{M}=(Y^{M}, f(Y^{M}), Df(Y^{M}) f(Y^{M}) + f'(Y^{M}), b(Y^{M}); M , Df(Y^{M}) \bullet \d M) $ is an $\mathbf{X}$-scRSM. In particular, $Y^{M}$ is an $L^{0}$-solution of the RSDE
        \begin{equation*}
        Y^{M}_{t}= Y^{M}_{0} + \int_{0}^{t} b_{s}(Y^{M}_{s}) \dd s+ M_{t} + \int_{0}^{t} (f, f')_{s}(Y^{M}_{s})  \dd \mathbf{X}_{s},
        \end{equation*}
        \item \label{ex:RSDE:item_2} Define $\tilde{Y}^{M}_{t} \coloneqq \phi^{-1}_{t}(Y_{t}^{M})$. Then $\tilde{Y}^{M}$ solves the SDE
    \begin{equation}
    \label{ex:RSDE:SDE}
    \begin{aligned}
        \tilde{Y}^{M}_{t} =&Y_{0}^{M} + \int_{0}^{t} D\phi^{-1}_{s}(\phi_{s}(\tilde{Y}^{M}_{s})) b_{s}(\phi_{s}(\tilde{Y}^{M}_{s}))\dd s+ \int_{0}^{t}  D \phi^{-1}_{s}(\phi_{s}( \tilde{Y}^{M}_{s}))\dd M_{s} \\
        &+ \frac{1}{2} \int_{0}^{t} D^{2} \phi^{-1}_{s}( \phi_{s}(
        \tilde{Y}^{M}_{s})) \dot{\langle M \rangle_{s}} \dd s
    \end{aligned}
    \end{equation}
    \end{enumerate}
    \end{proposition}
This approach however comes with two considerable drawbacks:
\begin{enumerate}[label=(\roman*)]
    \item When $\mathbf{X}$ is only assumed to be càdlàg , $\phi$ will in general only be a semi-flow i.e. non-invertible. See \cref{example:counterexample_flow} for a situation where the homeomorphic property \cref{assumptipon:flow}\ref{assumptions:flow_1} fails. 
    \item The $C^{3}$-property \cref{assumptipon:flow}\ref{assumptions:flow_2} requires additional regularity on the vector fields $(f, f')$ beyond the one's necessary to ensure well-posedness. In particular it would require $f_{t}\in \Lip^{5}, f_{t}'\in \Lip^{4}$(see \cite[Appendix B]{DauseFrizJentzenSong2026} for details). 
\end{enumerate}
While these two restrictions rule out utilizing flow-transform methods both in the semi-linear \cref{section:On semi-linear RSDEs} and non-linear \cref{section:Well-posedness of fully non-linear RSDEs} in this work, we emphasize that the conditions collected in  \cref{assumptipon:flow} are natural in the theory of rough PDEs \cite{BuginiFrizStannat2024, BFS25, DauseFrizJentzenSong2026}. 
\begin{corollary}
\label{corollary:flow transform good solution}
   Suppose $\xi \in L^{m}(\mathfrak{F}_{0}), N\in \mathcal{M}^{m, \infty}$ for $m \in [2, \infty)$ is continuous with absolutely continuous covariance bracket $\langle N \rangle$. Further let $b, \sigma$ be progressive  resp. predictable bounded Lipschitz and suppose \cref{assumptipon:flow} holds. Then there exists a unique good $L^{m,\infty}$-solution to \eqref{eq: non-linear RSDE} denoted by $\mathcal{Y}\coloneqq (Y, f(Y), Df(Y)f(Y)+ f'(Y), b(Y); \sigma(Y) \bullet \d N, DF(Y) \sigma(Y) \bullet \d N)$, where $Y_{t}= \phi_{t}(\tilde{Y}_{t})$ and $\tilde{Y}$ denotes the unique solution of the SDE
   \begin{equation}
   \label{eq:flow_transform_full_SDE}
    \begin{aligned}
        \tilde{Y}_{t} =&\xi  + \int_{0}^{t} D\phi^{-1}_{s}(\phi_{s}(\tilde{Y}_{s})) b_{s}(\phi_{s}(\tilde{Y}_{s}))\dd s+ \int_{0}^{t}  D \phi^{-1}_{s}(\phi_{s}( \tilde{Y}_{s}))\sigma_{s} (\phi_{s}(\tilde{Y}_{s})) \dd N_{s} \\
        &+ \frac{1}{2} \int_{0}^{t} D^{2} \phi^{-1}_{s}( \phi_{s}(
        \tilde{Y}_{s})) (\sigma_{s}\sigma_{s}^{\top})(\phi_{s}(\tilde{Y}_{s}))\dd \langle N\rangle_{s}
    \end{aligned}
    \end{equation}
\end{corollary}
\begin{proof}
  We note that by using the absolute continuity of $\langle N \rangle$, \eqref{eq:flow_transform_full_SDE} can be seen as a standard RSDE as considered in \cite{AllanPieper2026}. By standard well-posedness criteria for SDEs, \eqref{eq:flow_transform_full_SDE} has a unique solution denoted by $\tilde{Y}$ and therefore $\tilde{Y}$ is an $L^{m, \infty}$- rIp in the sense of \cref{def:rIp}. By \cref{prop:flow_transform} we see that $\mathcal{Y}$ is indeed an $L^{0}$-solution to \eqref{eq: non-linear RSDE} and using \cref{lemma: composition for RSM}, \cref{lemma: existence of rough stochastic integrals} we see that $\mathcal{Y}$ is actually also an $L^{m, \infty}$-solution.
\end{proof}
\begin{remark}
    Alternatively one can proof \cref{corollary:flow transform good solution} by \cref{Section:main_theorem} using that for $Y^{M}$ as defined in \eqref{eq: additive rsde} it holds $Y^{M}_{t}= \phi_{t}(\tilde{Y}^{M}_{t})$ with 
    \begin{equation*}
     \begin{aligned}   \tilde{Y}^{M}_{t} =&\xi  + \int_{0}^{t} D\phi^{-1}_{s}(\phi_{s}(\tilde{Y}^{M}_{s})) b_{s}(\phi_{s}(\tilde{Y}^{M}_{s}))\dd s+ \int_{0}^{t}  D \phi^{-1}_{s}(\phi_{s}( \tilde{Y}_{s})) 
     \dd M_{s} \\
        &+ \frac{1}{2} \int_{0}^{t} D^{2} \phi^{-1}_{s}( \phi_{s}(
        \tilde{Y}_{s})) \dd \langle M \rangle_{s},
    \end{aligned}
    \end{equation*}
    which allows us to easily verify the stability condition \eqref{ref:Key} using the global Lipschitzness of the RDE-flow $\phi$ \cite[Theorem 8.15]{FRIZHAIRER}.
\end{remark}
\section{A discrete approximation approach to consistency}\label{appendix: A discrete approximation approach to consistency}
In this section we prove the consistency result \cref{thm:consistency_of_solutions II} under additional knowledge of the continuity of the solutions to RSDEs and the existence of an appropriate discretisation. 
\subsection{Consistency under existence of appropriate discretisations}
We first recall the following definition of the piecewise constant discretisation of a $p$-rough path given in \cite[Chapter 5.3]{Friz_Zhang_2018}:
\begin{definition}\label{def: discretisation of rp}
    For any càdlàg path $X: [0,T] \to V$ as well as  $\cP= \{t_i\}_{i=0}^{N}$ any finite partition of $[0,T]$, we define its \emph{discretisation along $\cP$} to be defined by 
    \begin{equation}
        X^{\cP}_{t}:= \sum_{i=0}^{N-1} X_{t_{i}} \1_{[t_i,t_{i+1})}(t) + X_{T}\1_{\{T\}}(t)
    \end{equation}
    for any $t\in [0,T]$. Analogously, for vector-fields $f: [0,T]\times W \to U$ we define the discretisation $f^{\cP}$ by discretising the path $([0,T]\ni t \mapsto f_{t}(y))$ as above for any $y \in W.$
    Especially, for $p\in [2,3)$, $\bX=(X,\X) \in \mathscr{V}^{p}([0,T]; V)$ as $\cP= \{t_i\}_{i=0}^{N}$ a finite partition of $[0,T]$ we define $\bX^{\cP}=(X^{\cP}, \mathbb{X}^{\cP})\in \mathscr{V}^{p}([0,T]; V)$ by setting $\mathbb{X}_{0, s}^{\cP}= \mathbb{X}^{\cP}_{0,t_{i}}$ if $s\in [t_{i}, t_{i+1})$, $\mathbb{X}_{0,T}^{\cP}=\mathbb{X}_{0,T}$ and $\mathbb{X}_{s,t}^{\cP}$ by enforcing the Chen relation w.r.t. $X^{\cP}$ i.e. for $(s,t) \in \Delta_{[0,T]}$
    \begin{equation*}
        \mathbb{X}^{\cP}_{s,t}\coloneqq \mathbb{X}^{\cP}_{0,t}- \mathbb{X}^{\cP}_{0,s}- \delta X^{\cP}_{0, s}\otimes \delta X_{s, t}^{\cP}.
    \end{equation*}
\end{definition}

Equipped with this notion of discretisation we can formulate the specific notion of ``stability'' of solutions to RSDEs that we will need to show the wanted consistency between solutions found through \cref{thm:fixed_point_Banach} and mixed-moment solution:

\begin{definition}\label{def: A Davie-type solution}
    Let $p\in [2,3)$, $m\in \{0\}\cup [1,\infty)$ and $\bX$ be a given $p$-rough path as well as $N\in \cM^m$ and $b, \sigma$ be given progressively measurable and respectively predictable vector fields as well as $(f,f')$ fulfil \cref{assumption:vector_fields}. Suppose there exists an $L^m$-solution $\cY$ to
    \begin{equation}
    \label{eq: RSDE for Davie type}
        Y_{t}=Y_{0}+ \int_{0}^{t} b_{s}(Y_{s}) \dd s + \int_{0}^{t} \sigma_{s}(Y_{s-}) \dd N_{s} + \int_{0}^{t}(f, f')_{s}(Y_{s}) \dd \bX_{s}.
    \end{equation}
    We say that $\cY$ is a \emph{weakly stable solution} along $(\cP^n; (f^n, (f^n)'))_{n\in \N}$ if $\{\cP^n\}_{n\in \N}$ is a sequence of partitions of $[0,T]$ with vanishing mesh-size,  and $((f^n, (f^n)'))_{n\in \N}$ is a sequence of pairs of vector fields satisfying 
    \begin{enumerate}
        \item $(f^n, (f^n)')\in L^{0}_{\operatorname{prog}}(\Omega, \sV^{p}_{X^{\cP^n}}\Lip^{3}(; \cL(\R^{d_{X}}; \R^{d_{Y}})))$ if $(f,f')$ fulfils \eqref{assumption:item:vector_fields_1}
        \item $(f^n, (f^n)'; G^n)\in \sH_{X^{\cP^n}}^{p}\cL(\R^{d_{Y}}; \cL(\mathbb{R}^{d_{X}}; \R^{d_{Y}}))$ for some sequence $(G^n)_{n \in \mathbb{N}} \subset \mathcal{M}$ if $(f, f')$ fulfils \eqref{assumption:item:vector_fields_2}
    \end{enumerate}
    such that the $L^m$-solutions $\cY^n$ to the rough SDE
    \begin{equation}\label{eq: RSDE with discrete rough path}
        Y^n_t = Y_0+ \int_0^t b_s(Y^n_s) \dd s+ \int_0^t \sigma_s(Y^n_s) \dd N_s + \int_0^t (f^{n},(f^{n})')_s(Y^n_s) \dd \bX^{\cP^n}_s
    \end{equation}
    exist for every $n\in \N$ and  such that 
    \begin{equation*}
        \lim_{n\to \infty} |Y^{n}_t- Y_t| =0 
    \end{equation*}
    as limit in probability for every $t\in [0,T]$. In this case we call $(f^n, (f^n)')$ an \emph{($L^m$-)admissible sequence for $Y$.}
\end{definition}
\begin{lemma}\label{lemma: discrete consistency result for random RDEs}
    Let $p\in [2,3)$ and $m\in [2,\infty)$  as well as $\bX \in \sV^p$ and $N \in \cM$. Further, assume that $ (f, f')$ satisfies \cref{assumption:vector_fields} and $(\sigma, \sigma ')\in L^{0}_{pred}(\Omega, \mathscr{V}^{p}_{\bN}\Lip^{3}$ for $\bN:= (N, \Pi(N; N))$ the Itô-lift of $N$.
    Then suppose there exists an $L^m$-solution $Y$ to \eqref{eq: RSDE for Davie type} and an $L^{0}$-solution to  \eqref{eq: RSDE for Davie type} denoted by $\tY$ such that there exists a sequence $(\mathcal{P}^{n};(f^n, (f^n)'))_{n\in \N}$, which is $L^m$-admissible for $Y$ and $L^{0}$-admissible for $\tilde{Y}$ in the sense of \cref{def: A Davie-type solution}. Then $Y$ and $\tilde{Y}$ coincide up to indistinguishability.
\end{lemma}

\begin{proof}
   Let $(Y_t)_{t\in [0,T]}$ be the $L^m$-solution to \eqref{eq: RSDE for Davie type} and write $(Y^n)_{n\in \N}$ for the $L^m$-solutions to \eqref{eq: RSDE with discrete rough path} which exist by assumption.
   Write $\cP^n := \{t^n_i\}_{i=0}^{N_n}$, then  by definition we know that $\bX^{\cP^n}$ is constant on any of the intervals $[t^n_i, t^n_{i+1})$. 
   Therefore, using \cite[Lemma 4.7]{AllanPieper2026}one easily shows by induction  that for every $i\leq N_n-1$ and any $t\in [t^n_i, t^{n}_{i+1})$ we have that 
   \begin{equation}\label{eq: form of Y^n}
       Y^n_{t} = Y^{n,i}_{t}
   \end{equation}
   as well as
   $Y^n_T= \xi_{N_n}$ where we recursively define $(Y^{n,i}_{t})_{t\in [t_i^n, t^n_{i+1})}$ to be the unique solution to the SDE given by
   \begin{equation}\label{eq: SDE for discretisation}
       Y^{n,i}_t= \xi_i + \int_{t^n_i}^{t} b_s(Y^{n,i}_s) \dd s + \int_{t^n_i}^{t} \sigma_{s}(Y^{n,i}_{s-}) \dd N_s.
   \end{equation}
   for $\xi_0 = \xi$ and for any $i \in {1, \ldots, N_n}$ we set 
   \begin{align*}
       \xi_i &= Y^{n,i-1}_{t^n_{i}-}+ \sigma_{t^n_{i}-}(Y^{n,i-1}_{t^n_{i}-})\Delta N_{t^n_i} \\
       & \quad + f^n_{t^n_{i}-}(Y^{n,i-1}_{t^n_{i}-}) \delta X_{t^n_{i-1},t^n_i} +(D f^n_{t^n_{i}-}(Y^{n,i-1}_{t^n_i-})f^n_{t^n_{i}-}(Y^{n,i-1}_{t^n_i-})+ (f^n)'_{t^n_{i}-}(Y^{n,i-1}_{t^n_{i}-})) \X_{t^n_{i-1} , t^n_{i}} \in L^m.
   \end{align*}
   We emphasise, that by \cite[Lemma 4.7]{AllanPieper2026} $\xi_i- Y^{n,i-1}_{t^n_{i}-}$ is defined to be the jump of the RSDE \eqref{eq: RSDE with discrete rough path} at $t^n_{i}$.
   We first show the claim in the case of $(f,f')$ fulfilling \eqref{assumption:item:vector_fields_1} of \cref{assumption:vector_fields}.
   For that consider $(\tY, \tY')$ the $L^0$-solution to \eqref{eq: RSDE for Davie type} which is almost surely defined as the solution to the randomised RDE
    \begin{equation*}
        \tY_t= Y_0 + \int_0^t ((f\oplus \sigma), (f\oplus \sigma)')_r(\tY_r) \dd \sJ(\bX, N)_r
    \end{equation*}
    by \cref{thm:FZK_RSDE} and similarly $(\tY^n, (\tY^n)')$ the, again by \cref{thm:FZK_RSDE} almost surely defined, solution to 
    \begin{equation*}
        \tY^n_t= Y_0+ \int_0^t ((f^{n} \oplus \sigma), (f^{n}\oplus \sigma)')_r(\tY^n_r) \dd \sJ\big(\bX^{\cP^n}, N \big)_r.
    \end{equation*}
    Now notice, however, that  by definition for any $(s,t)\in \Delta_{[0,T]}$
    \begin{equation*}
        \Pi(X^{\cP^n}, N)_{s,t} = \int_{s}^{t} \delta X^{\cP^n}_{s ,r-} \dd N_r
    \end{equation*}
    where the integral is considered as a classical Itô integral and thus as $\delta X^{\cP^n}_{s,r-}=0$ for any $(s,r) \in \Delta_{[t^n_i,t^n_{i+1}]}$ we indeed may conclude that $\Pi(X^{\cP^n}, N)_{s,t}=0$ for every $(s,t) \in \Delta_{[t^n_i,t^n_{i+1}]}$. 
   Similarly, for the term $\Pi(N, X^{\cP^n})$ we may use that by \cite[Theorem 1.3]{fzk23} it holds for any $(s, t) \in \Delta_{[0,T]}$ that
    \begin{equation*}
        \Pi(N, X^{\cP^n})_{s,t}= \int_{s}^{t} \delta N_{s, r-} \dd X^{\cP^n}_r,
    \end{equation*}
    where the integral on the RHS is considered as a Riemann-Stieltjes integral, but as $X^{\cP^n}$ is constant on $[t^n_i,t^n_{i+1})$ we may conclude that $\Pi(N, X^{\cP^n})_{s,t}=0$ for any $(s,t) \in \Delta_{[t^n_i,t^n_{i+1})}$. Further, using the previous expression, one easily checks, that for any $t\in [0,T]$ we have that
    \begin{equation}
    \begin{split}
         \Delta \Pi(N, X^{\cP^n})_{t} := \lim_{s\uparrow t} \int_{s}^{t} N_{s, r-} \dd X^{\cP^n}_r = \Delta \bigg(\int_{0}^{\cdot}  N_{r-} \dd X^{\cP^n}_r\bigg)_{t} - N_{t-} \Delta X^{\cP^n}_{t}=0
    \end{split}
    \end{equation}
    and analogously one also finds that $\Delta \Pi(N,N)_t =0$ for every $t\in [0,T]$.\\    
    Thus, we have for any $0\leq i\leq N_n-1$ and for any $(s,t) \in [t^n_i,t^n_{i+1})$ that
    \begin{equation*}
         \sJ \big(\bX^{\cP^n}, N \big)_{s,t}= \bigg(
        \bigg(\begin{array}{c}
             0 \\
              \delta N_{s,t}
        \end{array}\bigg), 
        \bigg(\begin{array}{cc}
              0  & 0 \\
              0  & \Pi(N,N)_{s,t}
        \end{array}\bigg) 
    \end{equation*}
    as well as for any $0\leq i\leq N_n-1$
    \begin{equation*}
        \Delta \sJ \big(\bX^{\cP^n}, N\big)_{t^n_{i+1}} = \bigg(
        \bigg(\begin{array}{c}
              \delta X_{t^n_i, t^n_{i+1}} \\
              \Delta N_{t^n_{i+1}}
        \end{array}\bigg), 
        \bigg(\begin{array}{cc}
             \X_{t^n_i, t^n_{i+1}} & 0 \\
              0  &0
        \end{array}\bigg)
        \bigg).
    \end{equation*}
    Thus, again arguing by induction, using that we know the canonical jump structure of rough integrals (see \cite[Lemma 2.9]{Friz_Zhang_2018}) as well as the classical consistency of rough integrals against (Itô-)lifted semimartingales with their corresponding Itô counter parts (see e.g. \cite[Lemma 4.35]{ChevyrevFriz2019}) one can easily verify that $\tY^n$ also fulfils the equality \eqref{eq: form of Y^n} and thus by the uniqueness of the solution to the SDE \eqref{eq: SDE for discretisation} (see e.g. \cite[Theorem 16.3.11]{Cohen2015}) we find 
    \begin{equation*}
        \tY^n\equiv Y^n
    \end{equation*}
    up to indistinguishability.
   Therefore, we indeed know by uniqueness of limits in probability for any $t\in [0,T]$ that almost surely $Y_t = \tY_t $.
    As we assume that both $Y$ and $\tY$ are càdlàg almost surely, this implies that indeed $Y \equiv \tY$ up to indistinguishability and thus the claim under the assumption that $(f,f')$ fulfils \eqref{assumption:item:vector_fields_1} in \cref{assumption:vector_fields}.\\
    Finally, assume that $(f,f')$ fulfils \eqref{assumption:item:vector_fields_2} in \cref{assumption:vector_fields} instead meaning that there exists a martingale $(G_t)_{t\in [0,T]}$ of appropriate dimensions such that $(f,f'; G) \in \sH_{X}^{p}\cL$. Then once again we have that the $L^0$-solutions to \eqref{eq: RSDE for Davie type} and \eqref{eq: RSDE with discrete rough path} are given by the almost surely defined solutions to randomised RDEs, in this case also including the corresponding martingale parts of the respective vector fields. More precisely, we know by \cref{thm:FZK_RSDE} that $(\tY, \tY')$ the $L^0$-solution to \eqref{eq: RSDE for Davie type} is given by the almost sure solution to
    \begin{align*}
        \tY_t = \xi + \int_0^t b_s(\tY_s) \dd s + \int_0^t \big((f\oplus \sigma \oplus 0),(f \oplus \sigma \oplus 0)'\big)(\tY_s) \dd \sJ(\bX, (N, G))_s.
    \end{align*}
    and analogously the $L^0$-solution $\tY^n$ to \eqref{eq: RSDE with discrete rough path} is given by the almost surely defined solution to the (random) RDE given by 
    \begin{align*}
        \tY^n_t &= \xi + \int_0^t b_s(\tY^n_s) \dd s + \int_0^t \big((f^n \oplus \sigma \oplus 0),(f^n \oplus \sigma \oplus 0)'\big)(\tY^n_s) \dd \cJ(\bX^{\cP^n}, (M, G^n))_s
    \end{align*}
    for every $n\in \N$. Thus, we may once again check in the same way as before, that $\tY^n$ and $Y^n$ solve the same SDEs on $[t^n_i,t^n_{i+1})$ for every $n\in \N$ and $i\leq N$ and thus once again may conclude the claim by uniqueness of such solutions and the assumed weak stability property.
\end{proof}

\subsection{Existence of weakly stable solutions}
Given \cref{def: A Davie-type solution} we discuss in this subsection which type of choices we typically have in mind for $\{\cP^n\}_{n\in \N}$ and the corresponding admissible sequences $(f^n, (f^n)')_{n\in \N}$.
The following result, which is a direct extension of \cite[Lemma B.2]{AllanKwossekLiuPromel2025} will be the main tool to establish the existence of $\{\cP^n\}_{n\in \N}$ and corresponding admissible sequences $((f^n,(f^n)'))_{n\in \N}$ by a direct discretisation of both the rough path and the vector fields in the deterministic jump case. \\
As the proof therein can be word for word extended to allow paths in arbitrary Banach spaces we refer to said reference for a proof:

\begin{lemma}\label{lemma: AKLP approximation result}
Let $W$ be an arbitrary Banach space and $F\colon [0,T] \mapsto W$ a càdlàg path and $\{\cP^n\}_{n\in \N}$ a sequence of partitions of $[0,T]$ with vanishing mesh size. Further, let 
\begin{equation*}
    J_F:=\{t\in [0,T] \ : \ F_t \neq F_{t-}\}.
\end{equation*}
Then the following are equivalent 
\begin{enumerate}
    \item $J_F \subseteq \liminf_{n\to \infty} \cP^n$
    \item $F^{\cP^n}_t \longrightarrow F_t$ for every $t\in [0,T]$  
    \item $F^{\cP^n} \longrightarrow F$ uniformly in $t$.
\end{enumerate}
\end{lemma}

An immediate consequence is the following result for the discretisation of rough paths:
\begin{lemma}\label{lemma: existence of cP^n}
    Let $p\in [2,3)$ and $\bX\in \mathscr{V}^{p}$, then for any sequence of partitions $\{\cP^n\}_{n\in \N}$ of $[0,T]$ with vanishing mesh-size we have that 
    \begin{equation*}
        \lim_{n \to \infty} \|\bX^{\cP^n}- \bX\|_{\infty, [0,T]} =0,
    \end{equation*}
    if and only if $\{\cP^n\}_{n\in \N}$ fulfills 
    \begin{equation}\label{eq: Davie type partition for X}
        \{t\in [0,T] \ : \ \Delta \bX_t \neq 0\} \subseteq \liminf_{n\to \infty} \cP^n.
    \end{equation}
\end{lemma}

 Similarly, we can also conclude a corresponding discrete approximation result for controlled paths and vector fields:
\begin{lemma}\label{lemma: discretisation of controlled paths}
    Let $p\in [2,3)$ and $(Y,Y')\in \sV^{p}_X([0,T],W)$ some controlled path taking values in some Banach space $W$. Further, let $\cP= \{t_i\}_{i=0}^{N}$ be a partition of $[0,T]$ then $(Y^{\cP}, (Y')^{\cP}) \in \sV^{p}_{X^{\cP}}$. Further, if $\{\cP^n\}_{n\in \N}$ is a sequence of partitions of $[0,T]$ such that 
    \begin{equation}\label{eq: assumption on cPn for convergence of controlled paths}
        \{t\in [0,T] \ : \ |\Delta X_t| \vee |\Delta Y'_t| \vee |\Delta R^{Y,Y'}_t| > 0\} \subseteq \liminf_{n\to \infty} \cP^n, 
    \end{equation}
    then also
    \begin{equation*}
        \sup_{t\in [0,T]}\big(|Y^{\cP^n}_t-Y_t|+ |(Y')^{\cP^n}_t-  Y'_t| \big) + \sup_{(s,t) \in \Delta_{[0,T]}} \Big|R^{Y^{\cP^n},(Y')^{\cP^n}}_{s,t}-R^{Y,Y'}_{s,t}\Big| \longrightarrow 0
    \end{equation*}
    as $n\to \infty$.
\end{lemma}

\begin{proof}
    One immediately checks that indeed $(Y^{\cP}, (Y')^{\cP}) \in \sV^{p}_{X^{\cP}}$.
    Consider the path 
    \begin{align*}
        \mathscr{Y} \colon  [0,T] \ni t \mapsto (Y_t, Y'_t, R^{Y,Y'}_{0,t})
    \end{align*}
    which is càdlàg and takes values in some Banach space. Further, notice that 
    \begin{align*}
        \mathscr{Y}^{\cP^n}_t = \Big(Y^{\cP^n}_t, (Y')^{\cP^n}_t, R^{Y^{\cP^n}, (Y')^{\cP^n}}_{0,t} \Big).
    \end{align*}
    Once again, by \cref{lemma: AKLP approximation result} as well as Chen's relation for remainders (see  \eqref{eq: Chen for remainders}), we find that $\mathscr{Y}^{\cP^n} \to \mathscr{Y}$ uniformly in time as $n\to \infty$ if and only if 
    \begin{equation*}
         \{t\in [0,T] \ : \ |\Delta \mathscr{Y}_t|> 0\} \subseteq \{t\in [0,T] \ : \ |\Delta X_t| \vee |\Delta Y'_t| \vee |\Delta R^{Y,Y'}_t| > 0\} \subseteq \liminf_{n\to \infty} \cP^n.
    \end{equation*}
    To conclude the claim we notice, that for any $(s,t) \in \Delta_{[0,T]}$ we have that 
    \begin{equation*}
        |R^{Y^{\cP^n}, (Y')^{\cP^n}}_{s,t}-R^{Y,Y'}_{s,t}|\leq 2 \sup_{t\in [0,T]} |R^{Y^{\cP^n}, (Y')^{\cP^n}}_{0,t}-R^{Y}_{0,t}|+ |\delta Y^{\cP^n}_{0,s}\otimes \delta X^{\cP^n}_{s,t}-\delta Y_{0,s}\otimes \delta X_{s,t}|.
    \end{equation*}
\end{proof}

By treating controlled vector fields as controlled paths in the Banach space of bounded continuous functions endowed with the uniform topology\footnote{Notice that in \cref{lemma: AKLP approximation result} we did not assume the Banach space to be finite dimensional} we may also conclude the following

\begin{corollary}\label{corollary: discretisation of vector fields}
Let $p\in [2,3)$, $E, W,\bar{W}$ be some finite dimensional Banach spaces, $X$ some $E$-valued càdlàg path of finite $p$-Var and $(f,f') \in \sV^p_X\Lip^3([0,T]; \cL(W,\bar{W}))$, then for any sequence of partitions $\{\cP^n\}_{n\in \N}$ with 
 \begin{equation*}
     \Big\{t\in [0,T] \ : \ |\Delta X_t|_{E} \vee \sup_{x\in W}|\Delta f'_t(x)|_{\cL(E; \cL(W, \bar{W}))} \vee \sup_{x\in  W}|\Delta R^{Y,Y'}_t(x)|_{\cL(W,\bar{W})} > 0 \Big\} \subseteq \liminf_{n\to \infty} \cP^n
 \end{equation*}
 we have for any $\tp>p$ that 
 \begin{align*}
     &\Big\| \sup_{x\in W} |f^{\cP^n} (x)- f(x)|_{\bar{W}} \Big\|_{\tp, [0,T]} \vee \Big\| \sup_{x\in W} |(f')^{\cP^n} (x)- f'(x)|_{\cL(E; \bar{W})} \Big\|_{\tp, [0,T]} \\
     &\vee \Big\| \sup_{x\in W} \big|R^{f^{\cP^n}, (f')^{\cP^n}} (x)- R^{f,f'}(x) \big|_{\bar{W}} \Big\|_{\frac{\tp}{2}, [0,T]} \longrightarrow 0
 \end{align*}
 as $n\to \infty$.
\end{corollary}

Using these approximation results we can show that in the case where the vector fields only have deterministic jumps a straightforward discretisation of both the rough path and the vector fields along a sequence of partitions which exhaust all the jump times of both the vector field and the rough path produces an $L^m$-admissible sequence:

\begin{lemma}\label{lemma: randomised RDE is Davie}
     Let $p\in [2,3)$, $m\in [2,\infty) \cup \{0\}$, $\mathbf{X} \in \mathscr{V}^{p}$ and $N \in \cM^m$. Further, assume that $(\sigma, \sigma') \in L^{0}_{pred}(\Omega, \sV^{p}_{\bN} \operatorname{Lip}^{3})$ as well as that $(f, f')$ is a pair of vector fields fulfilling \cref{assumption:vector_fields} with \textbf{deterministic jump times}. Then, any $L^{0}$-solution $\cY$ to \eqref{eq: RSDE for Davie type} is weakly stable along any sequence $(\cP^n; (f^{\cP^n}, (f')^{\cP^n}))_{n\in \N}$ where the sequence of partitions $\{\cP^n\}_{n\in \N}$  fulfils
    \begin{equation}\label{eq: assumption on cPn for Davie type}
          \{t\in [0,T] \ : \ |\Delta \bX_t| \vee |\Delta f'_t| \vee |\Delta R^{f,f'}_t| > 0\} \subseteq \liminf_{n\to \infty} \cP^n.
    \end{equation}
    If additionally, $(f,f') \in \sV^{p,m, \infty}_{X}\Lip^{3}$  or $(f,f') \in \sV^{p,m,\infty}_{X} \cL$ respectively, then also any $L^m$-solution to \eqref{eq: RSDE for Davie type} is weakly stable along any sequence $(\cP^n; (f^{\cP^n}, (f')^{\cP^n}))_{n\in \N}$ such that the sequence of partitions $\{\cP^n\}_{n\in \N}$  fulfils \eqref{eq: assumption on cPn for Davie type} and thus the $L^m$-solutions is especially a good $L^m$-solution.
\end{lemma}

\begin{proof}
We start with the case $m=0$ and $(f,f')$ fulfills \eqref{assumption:item:vector_fields_1} then let
    \begin{equation*}
        \begin{split}
            & \cY = (Y, f(Y), D f(Y)f(Y)+ f'(Y), b(Y); \sigma(Y) \bullet \d N, D f(Y) \sigma(Y) \bullet \d N)\\
            &\cY^n=  (Y^n, f^{\cP_n}(Y^n), D f(Y^n) f^{\cP_n}(Y^n)+ (f')^{\cP_n}(Y^n), b(Y^n);\\
            &\hspace{40pt} \sigma(Y^n) \bullet \d N, D f^{\cP_n}(Y^n) \sigma(Y^n) \bullet \d N)
        \end{split}
    \end{equation*}
    denote $L^{0}$-solutions to \eqref{eq: RSDE for Davie type} and \eqref{eq: RSDE with discrete rough path} respectively. We know by \cref{thm:FZK_RSDE} that $(Y,f(Y))$ is almost surely a solution to the (random) RDE given by
    \begin{equation*}
            Y_{t}\coloneqq \xi + \int_{0}^{t} (f\oplus \sigma)_{s}(\omega) (Y_{s}(\omega)) \dd \sJ(\bX ; N)_{s},
    \end{equation*}
    and analogously $(Y^n,f^{\cP^n}(Y^n))$ is almost surely a solution to the corresponding random RDE with $\bX$ replaced by $\bX^{\cP^n}$ as well as $(f,f')$ replaced by $(f^{\cP^n}, (f')^{\cP^n})$.
    By a standard interpolation argument (see e.g. \cite[Lemma 5.2]{Friz_Zhang_2018}) for any $p<\tp$, we have that 
    \begin{align*}
        \|\bX^{\cP^n}- \bX\|_{\tp,[0,T]} \vee \big\llbracket f^{\cP^n},(f')^{\cP^n}; f,f' \big\rrbracket_{\tp, [0,T]} \longrightarrow 0 \quad \P\text{-a.s.}
    \end{align*}
    as $n\to \infty$, where we also used that $\|\bX^{\cP^n}\|_{p,[0,T]}\leq \|\bX\|_{p,[0,T]}$ and $\|f^{\cP^n},(f')^{\cP^n}\|_{p,[0,T]} \leq \|f,f'\|_{p,[0,T]}$ and thus also 
    \begin{equation*}
       \big\llbracket f^{\cP^n}\oplus \sigma, (f^{\cP^n}\oplus \sigma)';   f\oplus \sigma, (f\oplus \sigma)' \big\rrbracket_{\tp,[0,T]} \longrightarrow 0 \quad \P\text{-a.s.}
    \end{equation*}
    as $n\to \infty$. Therefore, by \cite[Theorem 1.8]{fzk23} we know that 
    \begin{equation*}
        \bigg\| \Big\|\sJ \big(\bX^{\cP^n}, N \big)- \sJ(\bX, N) \Big\|_{\tp,[0,T]} \bigg\|_{L^1} \longrightarrow 0
    \end{equation*}
    as $n\to \infty$ and, thus, there exists a subsequence $(n_k)_{k\in \N}$ such that this convergence also holds almost surely. Hence, by the continuous dependence on initial data of rough differential equations (implicitly included in \cref{lemma: stability in vfs for non-linear RSDEs}), we indeed have that 
    \begin{equation*}
        \sup_{t\in [0,T]} |\tY^{n_k}_t- \tY_t| \longrightarrow 0
    \end{equation*}
    as $k\to \infty$ almost surely. The case $m=0$ but $(f,f')$ fulfills \eqref{assumption:item:vector_fields_2} may be handled analogously using e.g. the deterministic version of \cref{lemma: stability of linear RSDEs without integrability loss} (implicitly included therein) instead. 
    Thus, assume now that $2<m$ and that $(f,f')$ fulfills \eqref{assumption:item:vector_fields_1} and $(f,f')\in V^{p,m, \infty}_{X}\Lip^{3}_X$. Then, the claim follows immediately by \eqref{eq: lipschitz in vector fields}.
    Finally, if $2\leq m$ and $(f,f')$ fulfils  \eqref{assumption:item:vector_fields_2} as well as $(f,f')\in \sV^{p,m, \infty}_{X} \cL$ then we may apply \cref{lemma: stability of linear RSDEs without integrability loss} to conclude once again.
\end{proof}

\begin{remark}
    In light of the previous result, one might try to generalise the theorem by replacing the deterministic sequence of partitions with a random one. However, discretising $\bX$ along random partitions produces a random rough path which may be correlated with the underlying martingale. In the mixed-moment approach of \cite{FHL21, AllanPieper2026}, this dependence might prevent the martingale cancellations in the remainder which are essential for the stochastic sewing argument.\footnote{See also Remark 3.14 in the first arXiv version of \cite{FHL21} for a related discussion.} The integration-by-parts approach of \cite{fzk23} faces an analogous problem, since the bracket term in \cref{lemma: joint lift and bracket} need not converge when the rough path and martingale are correlated; see \cite[Example 1.2]{fzk23}.
\end{remark}

With the linear RSDE solved by the Jacobian of an RSDE solution in mind we are specifically interested in the case of (semi)-linear RSDEs with random coefficients, for which \cref{corollary: discretisation of vector fields} is not applicable. As a first result in the direction of random coefficient fields, the previous result can be immediately improved to the following 

\begin{corollary}\label{corollary: consistency for linear RSM with deterministic jumps}
    Let $p\in [2,3)$, $m\in [2,\infty)$, $M\in \cM$, $\bX= (X,\X) \in \sV^p$ and  assume that $(f,f'; G)$ fulfils \eqref{assumption:item:vector_fields_2} of \cref{assumption:vector_fields} such that $f-G$ and $f'$ only have \textbf{deterministic jump times}. Further, assume that $(\sigma, \sigma') \in L^{0}_{pred}(\Omega, \sV^{p}_{\bM} \operatorname{Lip}^{3})$. Then any $L^0$-solution to \eqref{eq: RSDE for Davie type} is weakly stable along $(\cP^n, ((f-G)^{\cP^n} + G, (f')^{\cP^n}))_{n\in \N}$ for any sequence of partitions $\{\cP^n\}_{n\in \N}$ of $[0,T]$ fulfilling \eqref{eq: assumption on cPn for Davie type} for $(A,A'):= (f-G,f')$.
    If further $(f,f') \in \cV^{p,m,\infty}_{X}\cL$ and $N \in \cM^{p,m,\infty}$ then the $L^0$-solution is the unique good $L^m$-solution to \eqref{eq: RSDE for Davie type}.
\end{corollary}

\begin{proof}
We follow a similar line of argument to the one presented in \cref{lemma: randomised RDE is Davie}. We denote $(f^{n}, (f^{n})')\coloneqq ((f-G)^{\mathcal{P}^{n}}+ G, (f')^{\mathcal{P}^{n}})$ and $(A^n, (A^n)')$ accordingly and emphasize that by \cref{lemma:RSM<->cRP} this corresponds explicitly to a $L^0_{ad}(\Omega; \mathscr{V}^{p}_{(X^{\mathcal{P}^{n}}, G)}\cL)$ vector field. Now define $(f^{n}\oplus, \sigma \oplus 0), (f^{n}\oplus \sigma \oplus 0)'$ analogously to \cref{def:tensorized vf}. It holds that
\begin{equation}\label{eq: convergence of discretised vf with det jumps}
    \begin{split}
         &\llbracket f \oplus \sigma \oplus 0,(f \oplus \sigma \oplus 0)';   f^{n}\oplus \sigma \oplus 0,(f^{n}\oplus \sigma \oplus 0)'\rrbracket_{\tp,[0,T]}\\
         &= \llbracket A \oplus \sigma \oplus 0,(A \oplus \sigma \oplus 0)';   A^{n}\oplus \sigma \oplus 0,(A^{n}\oplus \sigma \oplus 0)'\rrbracket_{\tp,[0,T]}\longrightarrow 0 \quad \P\text{-a.s.}
    \end{split}
\end{equation}
and therefore especially, that $(\cP^n,(f^n,(f^n)'))$ is indeed an $L^0$-admissible sequence.
If additionally $(f,f') \in \sV^{p,m,\infty}_{X}\cL$ and $M\in \cM^{p,m,\infty}$, then also $(f^n,(f^n)') \in \sV^{p,m,\infty}_{X}\cL$ and one may easily check using \eqref{eq: convergence of discretised vf with det jumps}, that the assumptions of \cref{lemma: stability of linear RSDEs without integrability loss} are fulfilled, which again implies that$(\cP^n,(f^n,(f^n)'))$ is also an $L^m$-admissible sequence and thus the claim follows as in the proof of \cref{lemma: randomised RDE is Davie}.
\end{proof}

\begin{proof}[Proof of \cref{proposition: good solutions for Jacobians of flows}]
    We first prove the existence of both the $L^m$- as well as the $L^0$-solution to \eqref{eq: Psi RSDE}.
    Notice, that by \cref{theorem: existence and estimates for solutions to RSDEs} together with \cref{theorem: RSDE solution is L^q RSM} we indeed know that there exists a unique good $L^{m,\infty}$-solution $Y$ to \eqref{eq: RSDE with det. coeff.}. Further, by the same arguments as in \cite[Lemma 3.6]{BCN24} one easily shows that $(\Psi(Y), D \Psi(Y)f(Y) + \Psi'(Y)) \in \sV^{p,m,\infty}_{X}\cL$ and therefore by \cref{proposition: exsitence of additive linear RSDEs with random coeff} we know that there exists a unique $L^m$-solution to \eqref{eq: Psi RSDE} which we will denote by $\tY$. By \cref{lemma: composition for RSM} we also know that $(\Psi(Y), D \Psi(Y) f(Y) + \Psi'(Y); \D\Psi(Y)\sigma(Y_{\cdot-})\bullet \d N) \in \sH^p_X \cL$ and thus by \cref{thm:FZK_RSDE} we know that there also exists a unique $L^0$-solution $\overline{Y}$ to \eqref{eq: Psi RSDE}.  By \cref{lemma: discrete consistency result for random RDEs}, in order to show that these two solutions indeed coincide, it suffices to show that both of these solutions are weakly stable along any sequence of partitions $\{\cP^n\}_{n\in \N}$ of $[0,T]$ with
    \[ \{t\in [0,T] \ : \ |\Delta \bX_t|\vee |\Delta f'_t|\vee |\Delta R^{f,f'}_t| \vee |\Delta \Psi'_t| \vee |\Delta R^{\Psi,\Psi'}_t| > 0  \} \subseteq \liminf_{n\to \infty} \cP^n\] 
    with admissible sequence given by
    \[\big(\Psi^{\cP^n} (Y^n), (D \Psi)^{\cP^n}(Y^n) f^{\cP^n}(Y^n) \big)_{n\in \N}\]
    where $(Y^n, f(Y^n))$ is the unique good $L^{m,\infty}$-solution, which exists due to \cref{theorem: existence and estimates for solutions to RSDEs} and \cref{thm:consistency_of_solutions I}, to 
    \begin{equation*}
        Y^n_{t}=\xi + \int_{0}^{t} b_s(Y^n_{s}) \dd s+\int_{0}^{t} \sigma_s(Y^n_{s-}) \dd N_{s} + \int_{0}^{t}f^{\cP^n}_s(Y^n_{s}) \dd \bX^{\cP^n}_{s}.
    \end{equation*}
  By \cref{lemma: composition for RSM} together with again the same arguments as in \cite[Lemma 3.6]{BCN24}, we know that $(\Psi^{\cP^n}(Y^n), (D \Psi)^{\cP^n}(Y^n) Y^n + (\Psi')^{\cP^n}(Y^n), (\D\Psi)^{\cP^n}(Y^n)\sigma(Y^n_{\cdot-})\bullet \d N) \in \sH^{p}_X\cL \cap \sV^{p,m,\infty}_{X} \cL$. Further, we know by \cref{lemma: existence of cP^n} and \cref{corollary: discretisation of vector fields} that indeed for any $p<\tp$ we have that 
    \begin{equation*}
        \big \llbracket f^{\cP^n}, (f')^{\cP^n}; f,f' \big\rrbracket_{\tp,[0,T]} \vee \big\|\bX^{\cP^n}- \bX \big\|_{\tp, [0,T]} \longrightarrow 0
    \end{equation*}
    as $n \to \infty$ and thus by \eqref{eq: Lipschitz continuity of solution map} also that 
    \begin{equation*}
        \Big\|\sup_{t\in [0,T]} |Y^n_t-Y_t| \Big\|_{L^m} \longrightarrow 0
    \end{equation*}
    as $n\to \infty$. Again by \cref{corollary: discretisation of vector fields} we further find that also 
    \begin{equation*}
        \big \llbracket \Psi^{\cP^n}, (\Psi')^{\cP^n} ;  \Psi, \Psi' \big \rrbracket_{\tp, [0,T]} \longrightarrow 0 
    \end{equation*}
    as $n\to \infty$ and thus we have up to choosing an appropriate subsequence, that 
    \begin{align*}
        &\sup_{t\in [0,T]} |\Psi^{\cP^n}_t(Y^n_t)-\Psi_t(Y_t)|_{\cL} \\
        & \leq  \sup_{t\in [0,T]} \Big(|\Psi^{\cP^n}_t(Y^n_t) -\Psi_t(Y^n_t)|_{\cL} +  \sup_{t\in [0,T]}  |\Psi_t(Y^n_t) - \Psi_t(Y_t)|_{\cL}\Big)\\
        &\leq \llbracket \Psi^{\cP^n},(\Psi')^{\cP^n} ; \Psi, \Psi' \rrbracket_{\tp, [0,T]} + \|(\Psi, \Psi')\|_{\infty; 2} \sup_{t\in [0,T]} |Y^n_t-Y_t| \longrightarrow 0 \qquad \P\text{-a.s.}
    \end{align*}
    as $n\to \infty$ and similarly 
    \begin{align*}
        |D \Psi^{\cP^n}_t(Y^n_t)f^{\cP^n}(Y^n_t)-\D\Psi_t(Y_t)f(Y_t)| \longrightarrow 0
    \end{align*}
    as $n\to \infty$. Hence, by \cref{lemma: stability of linear RSDEs without integrability loss} we know that also 
    \begin{equation*}
        \Big\|\sup_{t\in [0,T]} |\tY^n_{t}- \tY_{t}| \Big\|_{L^m} \longrightarrow 0 \qquad \P\text{-a.s.}
    \end{equation*}
    as $n\to \infty$ which shows that $(\Psi(Y^n), D \Psi(Y^n) Y^n)_{n\in \N}$ is indeed an $L^m$-admissible sequence for $\tilde{Y}$.
    Finally, one may verify, similarly, that the implicitly included deterministic version of \cref{lemma: stability of linear RSDEs without integrability loss}, applied to $(\overline{Y}^n, \Psi^{\cP^n}(Y^n) \overline{Y}^n)$ the unique $L^0$-solution to \eqref{eq: Psi RSDE} with initial condition $y_0$, the vector fields $b, (\Psi^{\cP^n}(Y^n) , D \Psi^{\cP^n}_t(Y^n)f^{\cP^n}(Y^n))$ and the stochastic and rough drivers $M, \bX^{\cP^n}$, respectively, implies that also 
    \begin{align*}
       \sup_{t\in[0,T]} |\overline{Y}_t - \overline{Y}^n_t| \longrightarrow 0 \quad \P\text{-a.s.} 
    \end{align*}
    as $n\to \infty$, which implies the claim as $\tilde{Y}^{n}=\overline{Y}^{n}$ for any $n$ see \cref{lemma: discrete consistency result for random RDEs}.
\end{proof}
\section{Further technical results}
\label{section:Further technical results}
\subsection{Proof of a Kolmogorov-result for Remainders}

\begin{proof}[Proof of \cref{thm: Kolmogorov for some remainders}]
    We follow the proof of \cref{theorem: a p-variation kolmogorov result} and define the control
    \begin{align*}
        w(s,t) := (t-s) + \|R\|^{\frac{p}{2}}_{\frac{p}{2}, m, [s.t]}+ \|R^{2}\|^p_{p, m, [s,t]}
    \end{align*}
    and consider $\cD$ the $w(\cdot, \cdot-)$-dyadics. Assume without loss of generality that $\Delta R_T= \Delta R^{2}_T=0$. As before we find for any $(s,t) \in \Delta_{[0,T]}$ a sequence of indices $\{I^{s,t}_n\}_{n\in \N}$ with the same properties as described in the proof of \cref{theorem: a p-variation kolmogorov result}. Using \eqref{eq: Chen for remainders} we may thus bound
    \begin{equation}\label{eq: split for R_s,t}
    \begin{split}
        |R_{s,t}|^{\frac{\tp}{2}}  &\lesssim   \bigg(\sum_{n=1}^\infty \sum_{i \in I^{s,t}_n} |R_{t^n_i,t^n_{i+1}-}| \bigg)^{\hspace{-2pt}\frac{\tp}{2}} +\bigg(\sum_{n=1}^\infty \sum_{i \in I^{s,t}_n} |\Delta R_{t^n_{i+1}}| \bigg)^{\hspace{-2pt}\frac{\tp}{2}}\\
        &\quad + \bigg(\sum_{n=1}^\infty \sum_{i \in I^{s,t}_n} |\delta R^{(1)}_{t^n_{i},t^n_{i+1}-}| |\Delta R^{(2)}_{t^n_{i+1}}|  \bigg)^{\hspace{-2pt}\frac{\tp}{2}}+ \bigg(\sum_{n=1}^\infty \sum_{i \in I^{s,t}_n} |\delta R^{(1)}_{s,t^n_{i+1}}| |\delta R^{(2)}_{t^n_i,t^n_{i+1}}| \bigg)^{\hspace{-2pt}\frac{\tp}{2}},
    \end{split}
    \end{equation}
    where we write $\Delta R_{t} := \lim_{s \uparrow t} R_{s,t}$ which one can easily check exists as a limit in $L^m$ and the multiplicative constant only depends on $\tp$.
    Thus, by Young's inequality, the bound $|R^{(1)}_{u,v}| \leq \|R^{(1)}\|_{\tp, [s,t]}$ for any $(u,v) \subset [s,t]$ and $(\sum_{n=1}^\infty a_n)^{\tp} \leq C \sum_{n=1}^\infty n^\gamma a_n^{\tp}$ for any $\gamma >\tp-1$ (see the proof of \cite[Lemma 2]{LedouxLyonsQian2002}) we find that 
    \begin{align*}
         |R_{s,t}|^{\frac{\tp}{2}}&\lesssim \sum_{n=1}^\infty n^\gamma \sum_{i \in I^{s,t}_n} |R_{t^n_i,t^n_{i+1}-}|^{\frac{\tp}{2}} +\sum_{n=1}^\infty n^\gamma  \sum_{i \in I^{s,t}_n} |\Delta R_{t^n_{i+1}}|^{\frac{\tp}{2}} \\
         &\quad + \|R^{(1)}\|_{\tp,[s,t]}^{\tilde{p}}+\sum_{n=1}^\infty n^\gamma  \sum_{i \in I^{s,t}_n} |\delta R^{(2)}_{t^n_i, t^n_{i+1}-}|^{\tp} +\sum_{n=1}^\infty n^\gamma  \sum_{i \in I^{s,t}_n} |\Delta R^{(2)}_{t^n_{i+1}}|^{\tp}.
    \end{align*}
    Hence, for any partition $\{0 = t_0 < t_1 < \cdots < t_k = T\} \subset \cD$ of the interval $[0,T]$ we find that almost surely 
    \begin{equation*}
    \begin{aligned}
        \sum_{j=0}^{k-1}|R_{t_j,t_{j+1}}|^{\frac{\tp}{2}} \lesssim & \sum_{n=1}^\infty n^\gamma \sum_{i=0}^{2^n-1} |R_{t^n_i,t^n_{i+1}-}|^{\frac{\tp}{2}} +\sum_{n=1}^\infty n^\gamma  \sum_{i =0}^{2^n-1} |\Delta R_{t^n_{i+1}}|^{\frac{\tp}{2}} \\ 
        & +\sum_{n=1}^\infty n^\gamma  \sum_{i=0}^{2^n-1} |\delta R^{(2)}_{t^n_i, t^n_{i+1}-}|^{\tp} +\sum_{n=1}^\infty n^\gamma  \sum_{i=0}^{2^n-1} |\Delta R^{(2)}_{t^n_{i+1}}|^{\tp} +\|R^{(1)}\|_{\tp,[0,T]}^{\tp}\\
        \eqqcolon &  Z +\|R^{(1)}\|_{\tp,[0,T]}^{\tp}.
    \end{aligned}
    \end{equation*}
    Finally, using that $\frac{\tp}{2} < q$ as well as $\tp <n$ we can conclude as in the proof of \cref{theorem: a p-variation kolmogorov result} that $\|Z\|^{\tp}_{L^{\tq}}< \infty$ which indeed implies that almost every realisation of $R$ has finite $\frac{\tp}{2}$-variation on $\cD$. Arguing again as in  the proof of \cref{theorem: a p-variation kolmogorov result} and using the assumed right-continuity in probability we may conclude the first claim.\\
    To show the second claim, we may assume without loss of generality that we work with the càdlàg modification found in the first claim. Following the same line of argument as before we find by \eqref{eq: split for R_s,t} for any partition $\{0 = t_0 < t_1 < \cdots < t_k = T\} \subset \cD$ of the interval $[0,T]$, that almost surely 
    \begin{align*}
        &\bigg(\sum_{j=0}^{k-1}|R_{t_j,t_{j+1}}|^{\frac{\tp}{2}}\bigg)^{\frac{2}{\tp}} \lesssim  \bigg(\sum_{n=1}^\infty n^\gamma \sum_{i=0}^{2^n-1} |R_{t^n_i,t^n_{i+1}-}|^{\frac{\tp}{2}}\bigg)^{\frac{2}{\tp}} +\bigg(\sum_{n=1}^\infty n^\gamma  \sum_{i =0}^{2^n-1} |\Delta R_{t^n_{i+1}}|^{\frac{\tp}{2}}\bigg)^{\frac{2}{\tp}} \\ 
        & \quad +\bigg(\sum_{j=0}^{m-1} \bigg(\sum_{n=1}^\infty \sum_{i \in I^{t_j,t_{j+1}}_n} |\delta R^{(1)}_{t^n_{i},t^n_{i+1}-}| |\Delta R^{(2)}_{t^n_{i+1}}|  \bigg)^{\hspace{-2pt}\frac{\tp}{2}}\bigg)^{\frac{2}{\tp}}+ \bigg(\sum_{j=0}^{m-1} \bigg(\sum_{n=1}^\infty \sum_{i \in I^{t_j,t_{j+1}}_n} |\delta R^{(1)}_{t_j,t^n_{i+1}}| |\delta R^{(2)}_{t^n_i,t^n_{i+1}}| \bigg)^{\hspace{-2pt}\frac{\tp}{2}}\bigg)^{\frac{2}{\tp}}\\
        &\leq \bigg(\sum_{n=1}^\infty n^\gamma \sum_{i=0}^{2^n-1} |R_{t^n_i,t^n_{i+1}-}|^{\frac{\tp}{2}}\bigg)^{\frac{2}{\tp}} +\bigg(\sum_{n=1}^\infty n^\gamma  \sum_{i =0}^{2^n-1} |\Delta R_{t^n_{i+1}}|^{\frac{\tp}{2}}\bigg)^{\frac{2}{\tp}} \\ 
        & \quad +\bigg(\sum_{j=0}^{m-1}\| R^{1}\|^{\frac{\tp}{2}}_{\tp,[t_j,t_{j+1}]} \bigg(\sum_{n=1}^\infty \sum_{i \in I^{t_j,t_{j+1}}_n}  |\Delta R^{(2)}_{t^n_{i+1}}|  \bigg)^{\hspace{-2pt}\frac{\tp}{2}}\bigg)^{\frac{2}{\tp}}+ \bigg(\sum_{j=0}^{m-1}\| R^{(1)}\|^{\frac{\tp}{2}}_{\tp, [t_j,t_{j+1}]} \bigg(\sum_{n=1}^\infty \sum_{i \in I^{t_j,t_{j+1}}_n}  |\delta R^{(2)}_{t^n_i,t^n_{i+1}}| \bigg)^{\hspace{-2pt}\frac{\tp}{2}}\bigg)^{\frac{2}{\tp}}\\
        &\lesssim \bigg(\sum_{n=1}^\infty n^\gamma \sum_{i=0}^{2^n-1} |R_{t^n_i,t^n_{i+1}-}|^{\frac{\tp}{2}}\bigg)^{\frac{2}{\tp}} +\bigg(\sum_{n=1}^\infty n^\gamma  \sum_{i =0}^{2^n-1} |\Delta R_{t^n_{i+1}}|^{\frac{\tp}{2}}\bigg)^{\frac{2}{\tp}} \\ 
        &\quad +\|R^{(1)}\|_{\tp, [0,T]}\bigg( \sum_{n=1}^\infty n^{\gamma} \sum_{i=0}^{2^n-1}  |\Delta R^{(2)}_{t^n_{i+1}}|^{\tp} \bigg)^{\frac{1}{\tp}} + \|R^{(1)}\|_{\tp, [0,T]}\bigg( \sum_{n=1}^\infty n^{\gamma} \sum_{i=0}^{2^n-1}  |\delta R^{(2)}_{t^n_i,t^n_{i+1}-}|^{\tp} \bigg)^{\frac{1}{\tp}}\\
        &=: Z_1+Z_2+Z_3+Z_4.
    \end{align*}
    The terms $Z_1$ and $Z_2$ can be treated completely analogous to the proof of \cref{theorem: a p-variation kolmogorov result} which especially shows that
    \begin{equation*}
        \|Z_1\|_{L^m} + \|Z_2\|_{L^m} \lesssim \|R\|_{\frac{p}{2},m, [0,T]}
    \end{equation*}
    For $Z_3$ and $Z_4$ we find similarly that 
    \begin{align*}
        \|Z_4\|_{L^m}\leq \Big\|\|R^{(1)}\|_{p,[0,T]}\Big\|_{L^{m_1}} \bigg( \sum_{n=1}^\infty n^{\gamma} \sum_{i=0}^{2^n-1} \| \delta R^{(2)}_{t^n_i,t^n_{i+1}-} \|_{L^{m_2}}^{\tp} \bigg)^{\frac{1}{\tp}}\lesssim \Big\|\|R^{(1)}\|_{p,[0,T]}\Big\|_{L^{m_1}} \|R^{(2)}\|_{p,m,[0,T]}
    \end{align*}
    By a similar argument we find $\|Z_3\|_{L^m} \lesssim \|\|R^{(1)}\|_{\tp,[0,T]}\|$ and thus the claim\footnote{Because we can work with the càdlàg modification it suffices to show the bound for partitions in $\cD$.}.
\end{proof}

\subsection{A càdlàg rough stochastic Grönwall lemma}
The following result may be interpreted as a form of Grönwall's lemma for $p$-variations of controlled paths. The result is implicitly already included in the extension results from local to global stability estimates of R(S)DEs with jumps (see e.g. the proofs of \cite[Theorem 3.9]{Friz_Zhang_2018} and \cite[Theorem 5.2]{AllanPieper2026}):

\begin{lemma}\label{lemma: gronwall type result p var}
Let $p,m\in[2,\infty)$, let $Y,R^1,S^1\in V^pL^m$, and let
$R^2,S^2\in V^pL^\infty$. Let $R$ be an integrable two-parameter
process with $\|\mathbb E_{\cdot}R\|_{\frac{p}{2},m,[0,T]}<\infty$ satisfying
\begin{equation}\label{eq: Chen for remainders perturbed}
\delta R_{s,u,t}
=\delta R^1_{s,u}\otimes\delta R^2_{u,t}
+\delta S^1_{s,u}\otimes\delta S^2_{u,t}.
\end{equation}
Suppose that $w_1$ is a c\`adl\`ag control and $\eta\geq 0$ some constant such that
\begin{equation}
\|R^2\|_{p,\infty,[s,t]}\leq w_1(s,t)^{1/p} \qquad \text{and} \qquad
    \|S^1\|_{p,m,[0,T]}\|S^2\|_{p,\infty,[0,T]}
    \leq \eta
\end{equation}
and that, for some $C>0$, $\eta\geq 0$ and all $0\leq s<t\leq T$,
\begin{equation}\label{eq: assumption Grönwall type lemma}
\begin{aligned}
&\Big\|\sup_{v\in[s,t]}|\delta Y_{s,v}|\Big\|_{L^m}
+\|Y\|_{p,m,[s,t]}+\|R^1\|_{p,m,[s,t]}
+\|\E_{\cdot} R\|_{\frac{p}{2},m,[s,t]}\\
&\quad\leq C\Big(\|Y_s\|_{L^m}+\|R^1_s\|_{L^m}
+\|Y\|_{p,m,[s,t)}+\|R^1\|_{p,m,[s,t)}
+\|\E_{\cdot}R\|_{\frac{p}{2}, m ,[s,t)}\Big)
w_1(s,t)^{1/p}
+\eta.
\end{aligned}
\end{equation}
Set $\alpha = (2C)^{-p}$. Then, there exist constants $\Gamma, \beta >0$ which only depend on $C$ and $p$ such that
\begin{equation}\label{eq: gronwall inequality}
\begin{aligned}
&\Big\|\sup_{v\in[0,T]}|\delta Y_{0,v}|\Big\|_{L^m}
+\|Y\|_{p,m,[0,T]}+\|R^1\|_{p,m,[0,T]}
+\|\mathbb E_{\cdot}R\|_{\frac{p}{2},m,[0,T]}\\
&\quad\leq \Gamma \exp\Big(\beta \Big(1+ N_\alpha(w_1(\cdot, \cdot-))\Big)\Big(1+\sup_{0< t \leq T}
w_1(t-,t)^{\frac{1}{p}} \Big)\Big)
\Big(\|Y_0\|_{L^m}+\|R^1_0\|_{L^m}
+\eta\Big),
\end{aligned}
\end{equation}
where $N_{\alpha,[0,T]}(w_1(\cdot,\cdot-))$ is the size of the greedy partition as defined as in \cite[Definition 4.6]{AllanPieperTeichmann2025}.
\end{lemma}

\begin{remark}
    On a first glance the perturbed Chen's relation in \eqref{eq: Chen for remainders perturbed} might seem unusual. The example one should have in mind is $R= R^Y- R^{\tY}$ for some $Y\in \sV^{p,m}_X$ and $\tY \in \sV^{p,m}_{\tX}$. Then, we find that 
    \begin{equation*}
        \delta R_{s,u,t}= \delta (Y-\tY)_{s,u} \otimes \delta X_{u,t} + \delta \tY_{s,u} \otimes \delta (X-\tX)_{u,t}.
    \end{equation*}
    Thus, we would choose $R^1= Y-\tY$, $R^2= X$ and $S^1= \tY$ as well as $S^2= X-\tX$.
\end{remark}

\begin{proof}[Proof of \cref{lemma: gronwall type result p var}]
Write 
\begin{equation*}
J:= \sup_{0< t \leq T}
w_1(t-,t)^{\frac{1}{p}}
\end{equation*}
and choose $\{t_i(\alpha)\}_{i=0}^{N_{\alpha}(w_1(\cdot, \cdot-))+1}$ the greedy partition of $w_1(\cdot,\cdot-)$ as defined in \cite[Definition 4.6]{AllanPieperTeichmann2025} and write $N_{\alpha}$ and $\{t_i\}_{i=0}^{N_{\alpha}+1}$ for brevity.
Then, we know that especially $w_1(s, t)^{\frac{1}{p}} \leq \frac{1}{2C}$ for every $(s,t) \subset [t_i,t_{i+1})$ and we have that 
\begin{equation}\label{eq: first bound Grownall on small time interval}
\begin{split}
&\Big\|\sup_{v\in [s,t]} |\delta Y_{s,v}|\Big\|_{L^m} +\|Y\|_{p,m, [s,t]} + \|R^1\|_{p, m, [s,t]} + \|\E_{\cdot}R\|_{\frac{p}{2},m, [s,t]} \\
&\leq \frac{1}{2} \big( \|Y_s \|_{L^m}+ \|R^1_s \|_{L^m}+ \|Y\|_{p,m,[s,t)} + \|R^1\|_{p,m,[s,t)}+ \|\E_{\cdot} R\|_{\frac{p}{2}, m, [s,t)} \big) + \eta.
\end{split}
\end{equation}
Therefore it also holds that
\begin{equation*}
\Big\|\sup_{v\in [s,t]} |\delta Y_{s,v}|\Big\|_{L^m} +\|Y\|_{p,m, [s,t]} + \|R^1\|_{p, m, [s,t]} + \|\E_{\cdot} R\|_{\frac{p}{2},m, [s,t]}\leq   \|Y_s\|_{L^m}+ \| R^1_s \|_{L^m} + 2 \eta
\end{equation*}
and thus especially
\begin{equation*}
\begin{split}
&\Big\|\sup_{v\in [t_{i-1},t_i)} |\delta Y_{t_{i-1},v}|\Big\|_{L^m} +\|Y\|_{p,m, [t_{i-1},t_i)} + \|R^1\|_{p, m,[t_{i-1},t_i)} + \|\E_{\cdot} R\|_{\frac{p}{2},m, [t_{i-1},t_i)}\\
&\leq   \|Y_{t_{i-1}}\|_{L^m}+ \| R^1_{t_{i-1}} \|_{L^m} + 2 \eta.
\end{split}
\end{equation*}
Further, we know that for every $i$ we have by \eqref{eq: assumption Grönwall type lemma} and as
\begin{equation*}
\lim_{s\uparrow t_i} \big(\|Y\|_{p,m,[s,t_i)} + \|R^1\|_{p,m,[s,t_i)}+ \|\E_{\cdot} R\|_{\frac{p}{2}, m, [s,t_i)} \big)=0
\end{equation*} 
that 
\begin{align*}
\|\Delta Y_{t_i} \|_{L^m} &+ \|\Delta R^1_{t_i} \|_{L^m} + \limsup_{s\uparrow t_i}\|\E_s R_{s,t_i}\|_{L^m}\\
&\leq C\big(\|Y_{t_i-}\|_{L^m}+\|R^1_{t_i-}\|_{L^m} \big) w_1(t_i-, t_i)^{\frac{1}{p}}+\eta\\
&\leq C J \big(\|Y_{t_{i-1}}\|_{L^m}+\|R^1_{t_{i-1}}\|_{L^m} +\|Y\|_{p,m, [{t_{i-1}},t_i)}+\|R^1\|_{p,m, [{t_{i-1}}, t_i)} \big) +\eta
\end{align*}
and therefore applying \eqref{eq: first bound Grownall on small time interval} on $[t_{i-1},t_i)$ we find that
\begin{equation}\label{eq: bound  Y_p,[s,t] contraction}
\begin{split}
&\Big\|\sup_{v\in [t_{i-1},t_i]} |\delta Y_{t_{i-1},v}|\Big\|_{L^m}+ \|Y\|_{p,m,[t_{i-1},t_i]}+\|R^1\|_{p, m, [t_{i-1},t_i]} \\
&\lesssim_p (1+2CJ)\Big(\|Y_{t_{i-1}}\|_{L^m}+ \|R^1_{t_{i-1}}\|_{L^m} +\eta\Big)
\end{split}
\end{equation}
for some constant $c_p>0$ which only depends on $p$.
Similarly, using \eqref{eq: Chen for remainders perturbed} we find that 
\begin{equation}\label{eq: bound  Remainder,[s,t] contraction}
\begin{split}
\|\E_{\cdot} R\|_{\frac{p}{2},m,[t_{i-1},t_i]} &\lesssim_p \|\E_{\cdot} R\|_{\frac{p}{2},m,[t_{i-1},t_i)}+ \limsup_{s\uparrow t_{i}}\|\E_{s} R_{s,t_i}\|_{L^m}  \\
&\quad + \|R^1\|_{p,m,[t_{i-1}, t_i)} \|\Delta R^2_{t_i}\|_{L^{\infty}} +  \|S^1\|_{p,m,[t_{i-1}, t_i)} \|\Delta S^2_{t_i}\|_{L^{\infty}}\\
&\lesssim_{p,C} \Big(1+ 2CJ\Big)\Big(\|Y_{t_{i-1}}\|_{L^m}+ \|R^1_{t_{i-1}}\|_{L^m}+\eta\Big),
\end{split}
\end{equation}
where we especially used that by assumption we have that
\begin{equation*}
    \|S^1\|_{p,m,[t_{i-1}, t_i)} \|\Delta S^2_{t_i}\|_{L^{\infty}} \leq \|S^1\|_{p,m,[0,T]}\|S^2\|_{p,\infty,[0,T]} \leq \eta.
\end{equation*}
Piecing these bounds together we find a constant $c_{p,C}\geq 1$ only depending on $p$ and $C$ such that 
\begin{align*}
    &\Big\|\sup_{v\in [t_{i-1},t_i]} |\delta Y_{t_{i-1},v}|\Big\|_{L^m}+ \|Y\|_{p,m,[t_{i-1},t_i]}+\|R^1\|_{p, m, [t_{i-1},t_i]}+ \|\E_{\cdot} R\|_{\frac{p}{2},m,[t_{i-1},t_i]} \\
    &\lesssim_{p,C}  \Big(1+ 2CJ\Big)\Big(\|Y_{t_{i-1}}\|_{L^m}+ \|R^1_{t_{i-1}}\|_{L^m}+\eta+\|S^1\|_{p,m,[0,T]}\|S^2\|_{p,\infty,[0,T]}\Big)\\
    &\leq  i c_{p,C}^i (1+J)^{i} \Big(\|Y_{0}\|_{L^m}+ \|R^1_{0}\|_{L^m}+\eta\Big)
\end{align*}

Thus, by iteratively applying \eqref{eq: bound  Y_p,[s,t] contraction}
\begin{equation}\label{eq: Gronwall bound without remainder}
\begin{split}
& \Big\|\sup_{v\in [0,T]} |\delta Y_{0,v}|\Big\|_{L^m} + \|Y\|_{p,m,[0,T]}+ \|R^1\|_{p,m,[0,T]} \\
&\leq  \|Y_0\|_{L^m} + (1+N_{\alpha})^{\frac{p-1}{p}} \bigg(\sum_{i=0}^{N_{\alpha}}\Big(\Big\|\sup_{v\in [t_i,t_{i+1}]} |\delta Y_{t_i,v}|\Big\|_{L^m}+  \|Y\|_{p,m,[t_i,t_{i+1}]}+\|R^1\|_{p,m,[t_i,t_{i+1}]}\Big)^p\bigg)^{\frac{1}{p}}\\
&\leq\|Y_0\|_{L^m} +  (1+N_{\alpha})^{\frac{p-1}{p}} c_{p,C}^{N_{\alpha}+1} (1+J)^{N_{\alpha}+1}\Big(\|Y_{0}\|_{L^m}+ \|R^1_{0}\|_{L^m}+\eta\Big) \bigg(\sum_{i=1}^{N_{\alpha}+1} i^p\bigg)^{\frac{1}{p}}\\
&\leq  \exp\Big(\beta (1+N_{\alpha})(1+J)\Big)\Big(\|Y_{0}\|_{L^m}+ \|R^1_{0}\|_{L^m}+\eta\Big)
\end{split}
\end{equation}
for $\beta = 3+\log(c_{p,C})$.
Similarly, by \eqref{eq: Chen for remainders perturbed}  and \eqref{eq: bound  Remainder,[s,t] contraction} we find that
\begin{align*}
&\|\E_{\cdot}R\|_{\frac{p}{2},m, [0,T]} \\
&\lesssim_p (1+N_{\alpha})^{\frac{p-2}{p}} \bigg(\sum_{i=0}^{N_{\alpha}}  \|\E_{\cdot}R\|^{\frac{p}{2}}_{\frac{p}{2},m,[t_i,t_{i+1}]} + \|R^1\|^{\frac{p}{2}}_{p,m,[t_i,t_{i+1}]}\|R^2\|^{\frac{p}{2}}_{p,\infty,[t_{i+1},T]}  + \|S^1\|^{\frac{p}{2}}_{p,m,[t_i,t_{i+1}]}\|S^2\|^{\frac{p}{2}}_{p,\infty,[t_{i+1},T]}\bigg)^{\frac{2}{p}}\\
&\lesssim_p (1+N_{\alpha})^{\frac{p-1}{p}}\bigg(\bigg(\sum_{i=0}^{N_{\alpha}}  \|R^1\|^{p}_{p,m,[t_i,t_{i+1}]}\bigg)^{\frac{1}{p}} \|R^2\|_{p,\infty,[0,T]}+ \|S^1\|_{p,m,[0,T]}\|S^2\|_{p,\infty,[0,T]}\bigg)\\
&\quad + (1+N_{\alpha})^{\frac{p-2}{p}}\bigg(\sum_{i=0}^{N_{\alpha}}  \|\E_{\cdot}R\|^{\frac{p}{2}}_{\frac{p}{2},m,[t_i,t_{i+1}]}\bigg)^{\frac{2}{p}}\\
&\lesssim_{p,C} \Big(1+ J\Big)^2 (1+N_{\alpha})^{\frac{4p-1}{p}}\bigg(\bigg(\sum_{i=0}^{N_{\alpha}}  \|R^1\|^{p}_{p,m,[t_i,t_{i+1}]}\bigg)^{\frac{1}{p}}+ \eta \bigg)+ (1+N_{\alpha})^{\frac{p-2}{p}}\bigg(\sum_{i=0}^{N_{\alpha}}  \|\E_{\cdot}R\|^{\frac{p}{2}}_{\frac{p}{2},m,[t_i,t_{i+1}]}\bigg)^{\frac{2}{p}},
\end{align*}
where we applied \cite[Lemma A.5]{AllanPieperTeichmann2025} to $R^2$. By analogous arguments as in \eqref{eq: Gronwall bound without remainder} we may thus conclude.
\end{proof}

As a first application we show the following a priori bound for linear RSDEs (already contained in \cite{BCN24} in the continuous setting):

\begin{proposition}\label{proposition: a priori estimates additive linear RSDE}
    Let $p \in [2,3)$, $m \in [2,\infty)$, $\bX$ be a $p$-rough path, $\xi\in L^m(\mathfrak{F}_0)$, $b$ a stochastic Lipschitz vector field of linear growth from $\R^{d_Y}$ to $\R^{d_Y}$, $M\in \cM^{p,m}$ as well as $(A,A') \in \sV^{p,m,\infty}_{X} \cL(W, \cL(\R^n, W))$. Then, if  $(Y_t)_{t\in[0,T]}$ is an $L^m$-solution to
    \begin{equation}\label{eq: linear_RSDE}
        Y_t =  \xi+ M_t + \int_0^t b_s(Y_s) \dd s + \int_0^t (A,A')_s Y_s\dd \bX_s
    \end{equation}
    and let $K>0$ be a constant such that
    \begin{equation*}
        \|b\|_{\infty} \vee \|\xi\|_{L^m} \vee T \vee \|A,A'\|_\infty \vee \|A,A'\|_{p,m,\infty,[0,T]} \vee \|N\|_{p,m,[0,T]} \vee \|\bX\|_{p,[0,T]} \leq K,
    \end{equation*}
    then it holds
     \begin{equation} \label{eq: a priori bound for linear RSDE}
        \Big\| \sup_{v \in [0,T]}|Y_v| \Big\|_{L^m} +\|Y\|_{p,m,[0,T]}+\|\E_\cdot R^{Y}\|_{\frac{p}{2},m,[0,T]} \leq \tilde{\Gamma}(p,m,K),
    \end{equation}
    where the constant $\tilde{\Gamma}(p,m,K)>0$ only depends on $p$, $m$ and $K$.
\end{proposition}

\begin{proof}
    By Fubini's theorem and the assumed linear growth we have that 
    \begin{equation*}
        \bigg\|\sup_{v\in [s,t]} \bigg| \int_s^v b_u(Y_u) \dd u \bigg|\bigg\|_{L^m} \leq  \bigg\| \bigg| \int_s^t |b_u(Y_u)| \dd u \bigg|\bigg\|_{L^m}\leq \|b\|_{\infty}\Big(1+ \|Y_s\|_{L^m}+ \|Y\|_{p,m,[s,t)}\Big)(t-s).
    \end{equation*}
    By the BDG inequality, we have that 
    \begin{align*}
       \Big\| \sup_{v\in [s,t]} |M_v| \Big\|_{L^m}\lesssim \|M\|_{p,m,[s,t]}.
    \end{align*}
    Further, by \cite[Lemma 4.6]{AllanPieper2026} and analogous estimates as in \cite[Proposition 3.3]{BCN24} we have that
    \begin{align*}
        \bigg\|\sup_{v\in [s,t]} \bigg|\int_s^v \big(A_u Y_u,  A'_u Y_u + A_u^2 Y_u \big)\, \dd \bX_u \bigg| \bigg\|_{L^m}\lesssim
          \Big( \|Y_s\|_{L^m}+\|Y\|_{p,m,[s,t)} + \|\E_\cdot R^{Y}\|_{\frac{p}{2},m,[s,t)}\Big)\|\bX\|_{p,[s,t]},
    \end{align*}
    where the multiplicative constant only depends on $p,m$ and $K$. Thus, together with \cite[Lemma 4.8]{AllanPieper2026} we have that 
    \begin{align*}
        &\Big\| \sup_{v\in [s,t]}|Y_v| \Big\|_{L^m} +\|Y\|_{p,m,[s,t]}+\|\E_\cdot R^{Y}\|_{\frac{p}{2},m,[s,t]} \\
        &\lesssim_{p,m,K} \Big( \|Y_s\|_{L^m}+\|Y\|_{p,m,[s,t)} + \|\E_\cdot R^{Y}\|_{\frac{p}{2},m,[s,t)}\Big)\big(\|\bX\|_{p,[s,t]}+(t-s)\big) + (t-s)+ \|N\|_{p,m, [s,t]}
    \end{align*}
    for any $(s,t)\in \Delta_{[0,T]}$ and thus the claim follows by \cref{lemma: gronwall type result p var}.
\end{proof}

It will turn out to be useful to also have the following variant of the previous result, which follows directly by the bounds in \cite[Lemma 4.3]{AllanPieper2026}, \cite[Proposition 3.3]{BCN24} and the previously established a priori bound:

\begin{corollary}\label{corollary: a priori control bound linear additive RSDE}
    Let $p,m,\xi, b, N, \bX, (A,A')$ and $K$ as in \cref{proposition: a priori estimates additive linear RSDE}, then we also have for every $(s,t) \in \Delta_{[0,T]}$, that 
    \begin{equation*}
       \|Y\|_{p,m,[s,t]} \lesssim \|\bX\|_{p,[s,t]}+ \|M\|_{p,m, [s,t]}
    \end{equation*}
    and
    \begin{equation*}
       \|\E_{\cdot}R^{Y}\|_{\frac{p}{2}, m,[s,t]} \lesssim (\|\bX\|_{p,[s,t]}+ \|A,A'\|_{p,m, \infty,[s,t]})^2
    \end{equation*}
    where the multiplicative constant may only depend on $p,m$ and $K$.
\end{corollary}

We will need the following sequential stability result for linear rough SDEs. Notice, that local Lipschitz estimates for the solutions of linear rough SDEs with respect to the coefficient fields (see e.g. \cite[Theorem 3.13]{BCN24}) usually include a loss of integrability which would rule out linear RSDEs driven by general Lévy noise. 

\begin{lemma}\label{lemma: stability of linear RSDEs without integrability loss}
     Let $p \in [2,3)$, $m \in [2,\infty)$ and let $\bX$ be a $p$-rough path and $(\bX^n)_{n \in \N}$ a sequence of $p$-rough paths such that
     \begin{equation*}
         \|\bX^n- \bX\|_{\tp,[0,T]} \longrightarrow 0
     \end{equation*}
     as $n\to \infty$ for some $\tp>p$ as well as  $(A,A') \in \sV^{p,m,\infty}_{X} \cL(W, \cL(\R^n, W))$ and $((A^n, (A^n)')_{n\in \N})\subset \sV^{p,m,\infty}_{X^n}\cL(W, \cL(\R^n, W))$ a sequence such that 
     \begin{equation*}
         \sup_{t\in [0,T]}\Big(|A_t- A^n_t|_{\cL} + |A'_t- (A^n)'_t|_{\cL}\Big) \longrightarrow 0 \qquad \P\text{-a.s.}
     \end{equation*}
     as $n\to \infty$. Also assume that there exists a $K>0$ such that 
     \begin{equation}\label{eq: uniform boundedness of A^ns}
        \sup_n\|A^n, (A^n)'\|_{p,m,\infty, [0,T]} \vee     \|A, A'\|_{p,m,\infty, [0,T]} \vee \sup_n\|\bX^n\|_{p,[0,T]} \vee \|\bX\|_{p,[0,T]} \leq K.
     \end{equation}
     Then, if $(Y,AY)$ and $(Y^n, A^nY^n)$ are the solutions to \eqref{eq: linear_RSDE} with initial data $(\xi, M, (A,A'),\bX)$ and  $(\xi, M, (A^n,(A^n)'),\bX^n)$ respectively, we have that
     \begin{equation*}
         \Big\|\sup_{t\in [0,T]} |Y^n_t-Y_t| \Big\|_{L^m} \longrightarrow 0
     \end{equation*}
     as $n \to \infty$.
\end{lemma}

\begin{proof}
    Fix any $\epsilon>0$, then by definition, \cite[Lemma 4.6 \& Lemma 4.9]{AllanPieper2026} and adding zero we have that 
    \begin{align*}
         &\Big\| \sup_{v\in[s,t]} |\delta Y_{s,v}-\delta Y^{n}_{s,v}| \Big\|_{L^m} +\|Y- Y^{n}\|_{\tp ,m,[s,t]}+\|\E_\cdot (R^{Y}-R^{Y^{n}})\|_{\frac{\tp}{2},m,[s,t]}\\
         &\lesssim \Big(\|(A-A^n) Y\|_{\tp,m ,[s,t)}+ \|(A'-(A^n)') Y\|_{\tp,m ,[s,t)}+ \|(A^{\otimes 2}-(A^n)^{\otimes 2}) Y\|_{\tp,m ,[s,t)} \\
         &\qquad + \|(A'_s - (A^n)'_s) Y_s\|_{L^m}  +\|(A^{\otimes 2}_s- (A^n)^{\otimes 2}_s)  Y_s\|_{L^m}+\|(A_s-A^n_s) Y _s\|_{L^{m}}\Big) \|\bX\|_{\tp,[s,t]}\\
         &\quad + \|\E_{\edot}(R^{AY}-R^{A^nY^n}) \|_{\frac{\tp}{2},m ,[s,t)} \|\bX\|_{\tp,[s,t]}\\
         &\quad +\Big( \|A^n_s (Y _s- Y^n_s)\|_{L^{m}}+ \|A^n (Y- Y^n)\|_{\tp,m, [s,t)} +\|(A^n)' (Y-Y^n)\|_{\tp, m ,[s,t)} \\
         &\qquad+\|(A^n)^{\otimes 2} (Y-Y^n)\|_{\tp, m ,[s,t)}+ \| (A^n)'_s (Y_s-Y^n_s)\|_{L^m}+\|(A^n)^{\otimes 2}_s  (Y_s-Y^n_s)\|_{L^m}\Big)\|\bX\|_{\tp,[s,t]}\\
         &\quad+ \Big(\|AY\|_{\tp,m,[0,T)} + \|(AY)'\|_{\tp,m,[0,T)} + \|\E_{\edot} R^{AY}\|_{\frac{\tp}{2},m,[0,T)}\\
         &\qquad + \sup_{u \in [0,T)} \|(AY)'_u\|_{L^{m}} + \sup_{u \in [0,T)} \|A_u Y_u\|_{L^{m}}\Big) \|\bX - \bX^n\|_{\tp, [s,t]}
    \end{align*} 
    Further, notice that 
    \begin{equation*}
        R^{AY}_{s,t}= \delta A_{s,t}\delta Y_{s,t} + A_s R^{Y,Y'}_{s,t}+ R^{A}_{s,t} Y_s
    \end{equation*}
    and as the analogous relation holds for $R^{A^n Y^n}$ we therefore find that 
    \begin{align*}
        &\|\E_{\edot}(R^{AY}-R^{A^nY^n}) \|_{\frac{\tp}{2},m ,[s,t)} \\
        &\lesssim \|\E_{\edot}(\delta A- \delta A^n) \delta Y \|_{\frac{\tp}{2},m ,[s,t)}+\|\E_{\edot}(R^A- R^{A^n})Y \|_{\frac{\tp}{2},m ,[s,t)} +  \|\E_{\cdot} (A^n-A)R^{Y}\|_{\frac{\tp}{2}, m, [s,t)}\\
        &\quad + \Big(\|Y_s-Y^n_s\|_{L^m} + \|Y-Y^n\|_{\tp ,m,[s,t)} + \|\E_{\cdot} R^{Y,Y'}-R^{Y^n}\|_{\frac{\tp}{2}, m, [s,t)}\Big),
    \end{align*}
    where we use the abuse of notation for $(A_s)_{s\in [0,T]}$ a one-parameter process and $(B_{s,t})_{(s,t)\in \Delta_{[0,T]}}$ to write $(AB)_{s,t}:= A_s B_{s,t}$.
    Thus, using the assumed bounds on $(A,A'), (A^n, (A^n)')$ the a priori bounds on $Y$ and $Y^n$ provided by \cref{proposition: a priori estimates additive linear RSDE} where we note that these bounds may be chosen independently of $n\in \N$ due to \eqref{eq: uniform boundedness of A^ns} and similar calculations to those in the proof of \cite[Theorem 3.3]{BCN24}, we find that 
    \begin{align*}
         &\Big\| \sup_{v\in[s,t]} |\delta Y_{s,v}-\delta Y^{n}_{s,v}| \Big\|_{L^m} +\|Y- Y^{n}\|_{\tp ,m,[s,t]}+\|AY-A^nY^n\|_{\tp,m, [s,t]} +\|\E_\cdot (R^{Y}-R^{Y^{n}})\|_{\frac{\tp}{2},m,[s,t]}\\
         &\lesssim \Big(\|Y_s-Y^n_s\|_{L^m} + \|Y-Y^n\|_{\tp ,m,[s,t)} + \|\E_{\cdot} R^{Y,Y'}-R^{Y^n,(Y^n)'}\|_{\frac{\tp}{2}, m, [s,t)}\Big)\\
         &\quad \times (\|\bX\|_{\tp,[s,t]} + \|A^n\|_{\tp,m,\infty, [s,t]})  + \eta_n
    \end{align*}
    where we used that $\|A^n(Y-Y^n)\|_{\tp,m, [s,t]} \lesssim \|Y-Y^n\|_{p,m,[s,t]} + \|A^n\|_{\tp,m,\infty,[s,t]} \|Y_s-Y^n_s\|$ and define the control $\eta_n$ to be 
     \begin{equation}\label{eq: definition of eta_n}
    \begin{split}
         \eta_n &:= \|(A-A^n) Y\|_{\tp,m ,[0,T)}+ \|(A'-(A^n)') Y\|_{\tp,m ,[0,T)}+ \|(A^{\otimes 2}-(A^n)^{\otimes 2}) Y\|_{\tp,m ,[0,T)}\\
        & \quad + \sup_{v\in [0,T]}\Big(\|(A'_v - (A^n)'_v) Y_v\|_{L^m}  +\|(A^{\otimes 2}_v- (A^n)^{\otimes 2}_v)  Y_v\|_{L^m}+ \|(A_v-A^n_v) Y _v\|_{L^{m}}\Big) \\
        &\quad + \|\E_{\edot}(\delta A- \delta A^n) \delta Y \|_{\frac{\tp}{2},m ,[0,T)}+\|\E_{\edot}(R^A- R^{A^n})Y \|_{\frac{\tp}{2},m ,[0,T)} +  \|\E_{\cdot} (A^n-A)R^{Y}\|_{\frac{\tp}{2}, m, [0,T)}\\
        &\quad+ \|\bX-\bX^n\|^{\tp}_{\tp,[0,T]}.
    \end{split}
    \end{equation}
    Thus, we may apply \cref{lemma: gronwall type result p var} with $R= R^{Y,Y'}-R^{Y^n}$, $R^1 = (Y'-(Y^n)')= (AY-A^n Y^n)$, $R^2= X$, $S^{1}= (Y^n)'=A^n Y^n$, $S^{2}= X-X^n$, $w^n_1(s,t):=(\|\bX\|^{\tp}_{\tp,[s,t]} + \|A^n\|^{\tp}_{\tp,m,\infty, [s,t]})$ and $\eta:= \eta_n$, to find that 
    \begin{equation*}
        \Big\|\sup_{t\in [0,T]} |Y^n_t-Y_t| \Big\|_{L^m} \leq C_1 \big(\|(A_0-A^n_0)\xi\|_{L^m}+ \eta_n\big)
    \end{equation*}
    where the constant $C_1$ is only depending on $K, \|\bX\|_{\tp,[0,T]}, p,m$ and $T$ and especially is independent of $n\in \N$.
    By interpolation we find that 
    \begin{equation*}
        \|(A-A^n) Y\|_{\tp,m ,[0,T]} \lesssim \Big\|\sup_{s\in [0,T]} |(A-A^n)_s Y_s| \Big\|_{L^m}^{(1-\frac{p}{\tp})} (\|AY\|_{p,m,[0,T]}+\|A^n Y\|_{p,m, [0,T]})^{\frac{p}{\tp}}.
    \end{equation*}
    As we know by \cref{proposition: a priori estimates additive linear RSDE} that
    $$\Big\|\sup_{t\in [0,T]} |A_t| \Big\|_{L^{\infty}}\vee \Big\|\sup_{t\in [0,T]} |A^n_t| \Big\|_{L^{\infty}}\vee \Big\|\sup_{t\in [0,T]} |Y_t| \Big\|_{L^m} <\infty$$
    we have by dominated convergence, that
    \begin{equation*}
        \Big\|\sup_{s\in [0,T]} |(A-A^n)_s Y_s| \Big\|_{L^m}^{\frac{p}{\tp}} \longrightarrow 0
    \end{equation*}
    and thus 
    \begin{align*}
        \|(A-A^n) Y\|_{\tp,m ,[0,T]} &\lesssim \Big\|\sup_{s\in [0,T]} |(A-A^n)_s Y_s| \Big\|_{L^m}^{(1-\frac{p}{\tp})}(\|AY\|_{p,m,[0,T]}+\|A^n Y\|_{p,m, [0,T]})^{\frac{p}{\tp}}\\
        &\lesssim \Big\|\sup_{s\in [0,T]} |(A-A^n)_s Y_s| \Big\|_{L^m}^{(1-\frac{p}{\tp})}\longrightarrow 0 
    \end{align*}
    as $n\to \infty$ as well. 
    Similarly, we can treat the terms $(A^{\otimes 2}- (A^n)^{\otimes 2})Y$ and $(A'-(A^n)')Y$, apprearing in \eqref{eq: definition of eta_n}.
    Also using interpolation, the known bounds as well as DCT we find that 
    \begin{align*}
         \|\E_{\cdot}(\delta A- \delta A^n) \delta Y \|_{\frac{\tp}{2},m ,[s,t)} \lesssim \Big\|\sup_{(s,t) \in [0,T]} (\delta A_{s,t}-\delta A^n_{s,t})\delta Y_{s,t} \Big\|^{1- \frac{p}{\tp}}_{L^m} \longrightarrow 0.
    \end{align*}
    as well as 
    \begin{align*}
        \|\E_{\cdot} (A^n-A)R^{Y}\|_{\frac{\tp}{2}, m, [0,T]}\lesssim \Big \|\sup_{(s,t)\in \Delta_{[0,T]}} (A_s-A^n_s) R^{Y,Y'}_{s,t} \Big\|^{1-\frac{p}{\tp}}_{L^m} \longrightarrow 0
    \end{align*}
    as $n\to \infty$.
    Finally, we note that 
    \begin{align*}
        \sup_{(s,t)\in \Delta_{[0,T]}} |R^A_{s,t}-R^{A^n}_{s,t}|&\lesssim \sup_{s\in [0,T]} |A_s-A^n_s|+ \sup_{s\in [0,T]} |A'_s-(A^n)'_s|\|\bX\|_{\tp,[0,T]}\\
        &\quad+ \sup_{s\in [0,T]} |(A^n)'_s| \|\bX-\bX^n\|_{\tp, [0,T]} \to 0 \quad \P\text{-a.s.}
    \end{align*}
    and thus by similar interpolation arguments as before we also know that 
    \begin{equation*}
       \|\E_{\edot}(R^A- R^{A^n})Y \|_{\frac{\tp}{2},m ,[0,T]} \longrightarrow 0
    \end{equation*}
    which all together gives $\eta_n \to 0$ as $n\to \infty$, and thus yet another DCT argument shows that
    \begin{equation*}
        \Big\|\sup_{t\in [0,T]} |Y^n_t-Y_t| \Big\|_{L^m} \leq C_1 \big(\|(A_0-A^n_0)\xi\|_{L^m}+ \eta_n\big) \longrightarrow 0
    \end{equation*}
    as $n\to \infty$ as wanted.
\end{proof}

\printbibliography
\end{document}